\documentclass[a4paper,11pt]{article}
\usepackage[margin=1in]{geometry}

\usepackage{graphicx,url,subfig}
\RequirePackage[OT1]{fontenc}
\RequirePackage{amsthm,amsmath,amssymb,amscd}
\RequirePackage[numbers]{natbib}
\RequirePackage[colorlinks,citecolor=blue,urlcolor=blue]{hyperref}
\usepackage{graphics}
\usepackage{enumerate}
\usepackage{datetime}
\usepackage{etex}
\usepackage[normalem]{ulem} % either use this (simple) or
\usepackage{soul} % use this (many fancier options)
\makeatother
\numberwithin{equation}{section}
\allowdisplaybreaks
\usepackage{mathrsfs}

\usepackage{amssymb, latexsym, amsmath, graphics, fullpage, epsfig, amsthm, relsize, pgf, tikz, amsfonts, makeidx, latexsym, ifthen, calc}
\usepackage[colorlinks,citecolor=blue,urlcolor=blue]{hyperref}
\usetikzlibrary{arrows}
\usepackage{mathrsfs}
\usepackage[shortlabels]{enumitem}
\usepackage{tikz}
\usepackage{amsmath}
\usetikzlibrary{arrows}
\usetikzlibrary{shadows.blur}
\usepackage{pgfplots}
\usepackage{mwe} % For dummy images
\usepackage{multirow}
\usepackage{dcolumn}
\newcolumntype{2}{D{.}{}{2.0}}
\usepackage{multicol}
\numberwithin{equation}{section}
\theoremstyle{plain}
\newtheorem{theorem}{Theorem}[section]
\newtheorem{corollary}[theorem]{Corollary}
\newtheorem{lemma}[theorem]{Lemma}
\newtheorem{proposition}[theorem]{Proposition}

\theoremstyle{definition}
\newtheorem{definition}[theorem]{Definition}

\newtheorem{assumption}[theorem]{Assumption}

\newtheorem{remark}[theorem]{Remark}

\newtheorem{example}[theorem]{Example}

\newcommand{\E}{\mathbb{E}}
\newcommand{\W}{\dot{W}}

\newcommand{\ud}{\ensuremath{\mathrm{d} }}
\newcommand{\Ceil}[1]{\left\lceil #1 \right\rceil}

\newcommand{\Norm}[1]{\left\|  #1   \right\|}

\newcommand{\R}{\mathbb{R}}

\begin{document}

\title{Interpolating the Stochastic Heat and Wave Equations\\ with Time-independent Noise:\\
Solvability and Exact Asymptotics\footnote{Department of Mathematics and
Statistics, Auburn University, Auburn, Alabama.}}
\author{Le Chen\footnote{Email: \url{le.chen@auburn.edu}.}~
	and
Nick Eisenberg\footnote{Emails: \url{nze0019@auburn.edu};~\url{nickeisenberg@gmail.com}.}}
\date{\today}

\maketitle

\begin{abstract}
	\noindent In this article, we study a class of stochastic partial differential equations with
	fractional differential operators subject to some time-independent multiplicative Gaussian noise.
	We derive sharp conditions, under which a unique global $L^p(\Omega)$-solution exits for all $p\ge
	2$. In this case, we derive exact moment asymptotics following the same strategy in a recent work
	by Balan {\it et al} \cite{BCC21}. In the case when there exits only a local solution, we
	determine the precise deterministic time, $T_2$, before which a unique $L^2(\Omega)$-solution
	exits, but after which the series corresponding to the $L^2(\Omega)$ moment of the solution blows
	up. By properly choosing the parameters, results in this paper interpolate the known results for
	both stochastic heat and wave equations.  \end{abstract}


%60H15=SPDEs
%60H07 stochastic calculus of variations and the Malliavin calculus
%37H15 Multiplicative ergodic theory, Lyapunov exponents
\vspace{1mm}


\tableofcontents
% \pagebreak
\section{Introduction}
In this paper we study the following {\it stochastic partial differential equation} (SPDE) with
fractional differential operators:
\begin{equation}
	\label{E:SPDE}
	\begin{cases}
		\left(\partial^b_t + \frac{\nu}{2}(-\Delta)^{a/2}\right)\: u(t,x) = I^r_t \left[\sqrt{\theta}\: u(t,x)\: \dot W(x) \right] & \text{$x\in \R^d$, $t>0$}, \\
		u(0,\cdot) = 1                                                                                                             & b \in (0,1],               \\
		u(0,\cdot) = 1, \quad \partial_t u(0,\cdot) = 0                                                                            & b \in (1,2),
	\end{cases}
\end{equation}
where $a \in (0,2)$, $b \in (0,2)$, $r\ge 0$, $\nu > 0$ and $\theta>0$. Here the noise $W =
\left\{ W(\phi) : \phi \in \mathcal{D}(\R^d)  \right\}$ is a centered and time-independent
Gaussian process, defined on a complete probability space $(\Omega, \mathcal{F}, P)$, with mean zero
and covariance
\begin{align*}
		\mathbb{E}\left[ W(\phi) W(\psi) \right] = \int_{\R^d} \mathcal{F}\phi(\xi)\overline{\mathcal{F}\psi(\xi)} \mu(\ud \xi) =: \langle \phi, \psi \rangle_{\mathcal{H}},
\end{align*}
where $\mu$ refers to the {\it spectral measure}, which is assumed to be a nonnegative and
nonnegative definite tempered measure on $\R^d$. Let $\gamma$ be the Fourier transform of $\mu$ (see Section \ref{SS:Malliavin}),
which is also a nonnegative and nonnegative measure on $\R^d$ thanks to Bochner's theorem.
Throughout the paper, we use $\mathcal{F}\phi(\xi) = \int_{\R^d} \exp(-ix \xi) \phi(x) \ud
x$ to denote the Fourier transform of a test function $\phi$.

In \eqref{E:SPDE}, $(-\Delta)^{a/2}$ refers to the {\it fractional Laplacian} of order $a$,
$\partial^b_t$ denotes the {\it Caputo fractional differential operator}
\begin{align*}
	\partial^b_tf(t) :=
	\begin{cases}
		\dfrac{1}{\Gamma(m-b)} \int_0^t \ud \tau \dfrac{f^{m}(\tau)}{(t-\tau)^{b+1-m}} & \text{if } m-1 < b < m, \bigskip \\
		\dfrac{\ud^m}{\ud t^m}f(t)                                                     & \text{if } b=m,
	\end{cases}
\end{align*}
where $m$ is an integer, and $I^r_t$ refers to the {\it Riemann-Liouville fractional integral} of order $r>0$
\begin{align*}
	I^r_tf(t) := \frac{1}{\Gamma(r)} \int_0^t (t-s)^{r-1}f(s)\ud s, \quad \text{for } t>0,
\end{align*}
with the convention that when $r=0$, $I^0_t = \text{Id}$ reduces to the identity operator. The
fundamental solutions to \eqref{E:SPDE} consist of a triplet:
\begin{align*}
	\left\{ Z_{a,b,r,\nu,d}(t,x),\:  Z^*_{a,b,r,\nu,d}(t,x), \: Y_{a,b,r,\nu,d}(t,x) : t >0, x \in \R^d \right\}.
\end{align*}
Each member of the triplet can be expressed explicitly in terms of the {\it Fox H-function} which is
much more complicated than the fundamental function for either the heat or wave equation; See
Theorem 4.1 of \cite{CHN19}. Since we are interested in the constant one initial condition (and zero
initial velocity when $b>1$), Theorem 4.1 ({\it ibid.}) implies that the corresponding solution to
the homogeneous equation (i.e. the solution when there is no driving source) is equal to the
constant one. Hence, by setting $G(t,x) := Y_{a,b,r,\nu,d}(t,x)$, \eqref{E:SPDE} can be written as
the following stochastic integral equation:
\begin{equation}
	\label{E:spdeint}
	u(t,x) = 1 + \sqrt{\theta}\int_0^t \left( \int_{\R^d}G(t-s,x-y)u(s,y) W(\delta y) \right) \ud s,
\end{equation}
where the stochastic integral is in the {\it Skorohod} sense; see Definition \ref{D:Sol} below.  In
the following, the fundamental solution will exclusively refer to $G(t,x)$, which is indeed a smooth
function for $x\ne 0$. Our results rely on the following assumption for the nonnegativity of
$G(t,x)$:

\begin{assumption}[Nonnegativity]
	\label{A:Nonnegative}
	Assume that the fundamental solution $G(t,x)$ is nonnegative for all $t>0$ and $x\in\R^d$.
\end{assumption}

\begin{remark}
	\label{R:Nonnegative}
	% Regarding the nonnegativity assumption --- Assumption \ref{A:Nonnegative},
	Thanks to Theorem 4.6 of \cite{CHN19} (see also Theorem 3.1 of \cite{CHHH17} for the case when $r=0$),
	we have the following four groups of sufficient conditions \footnote{Note that when $d\ge 1$,
		$b=1$ and $a\in (0,2]$, part (1) of \cite[Theorem 4.6]{CHN19} says that the fundamental solution
		$Y$, which is the fundamental solution $G$ in this paper, is nonnegative provided $r=0$ or
	$r>1$. Indeed, because in this case $Z$ is always nonnegative, for $r>0$, $Y$ as a fractional
	integral of $Z$ (see (4.5), {\it ibid.}),  $Y$, or our $G$, should also be nonnegative. We thank
	Guannan Hu who pointed out to us this observation.}, under either group of which $G(t,\cdot)$ is
	nonnegative (see Figure \ref{F:Nonnegative} for an illustration) :
	\begin{enumerate}
		\item $d \ge 1$, $b\in(0,1]$, $a\in(0,2]$, $r\ge0$;
		% \item $d \ge 1$, $b=1$, $a\in (0,2]$, $r=0$ or $r>1$;
		\item $1\le d \le 3$, $1<b<a\le 2$, $r>0$;
		\item $1 \le d \le 3$, $1 < b=a < 2$, $r > \dfrac{d+3}{2}-b$.
	\end{enumerate}
\end{remark}

\begin{figure}[htpb]
	\centering
	\begin{tikzpicture}[scale=2.2]
	\draw [->] (-0.3,0) -- (2.3,0) node [right] {$a$};
	\draw [->] (0,-0.3) -- (0,2.2) node [above] {$b$};
	\draw [densely dashed,gray] (1,0) -- (1,2) (2,0) -- (2,2) (-0.4,1) -- (2.5,1) (0,2) -- (2,2);
	\draw (1,0.05) -- (1,-0.05) node [below] {$1$};
	\draw (2,0.05) -- (2,-0.05) node [below] {$2$};
	\draw (0.05,1) -- (-0.05,1) node [left] {$1$};
	\draw (0.05,2) -- (-0.05,2) node [left] {$2$};
	\filldraw [fill=red!30, opacity=0.5] (0,0) rectangle (2,1);
	\filldraw [fill=blue!30, opacity=0.5] (1,1) -- (2,1) -- (2,2);
	\draw [thick] (1.03,1.03) -- (1.97,1.97);
	\draw [thick] ([shift=(0:0.05)]1,1) arc (0:90:0.05);
	\draw [thick] ([shift=(180:0.05)]2,2) arc (180:270:0.05);
	\node at (2.3,1.5) [rotate=90, anchor=center] {$1\le d\le 3$};
	\node at (2.3,0.45)[rotate=90, anchor=center] {$d\ge 1$};
	\draw [thick] (0.03,1) -- (2,1);
	\draw [thick] ([shift=(-90:0.05)]0,1) arc (-90:90:0.05);
	\node at (1.73,1.35) [rotate=45, anchor=center] {$r\ge 0$};
	\node at (1,0.45) {$r\ge 0$};
	% \node at (1,0.9) {$r=0$ or $r>1$};
	\node at (1.5,1.55) [rotate=45,anchor=south] {$r>\frac{d+3}{2}-b$};
	\end{tikzpicture}
	\caption{Illustration of the sufficient conditions (Remark \ref{R:Nonnegative}) for
		$G(t,\cdot)$ to be nonnegative.}
	\label{F:Nonnegative}
\end{figure}
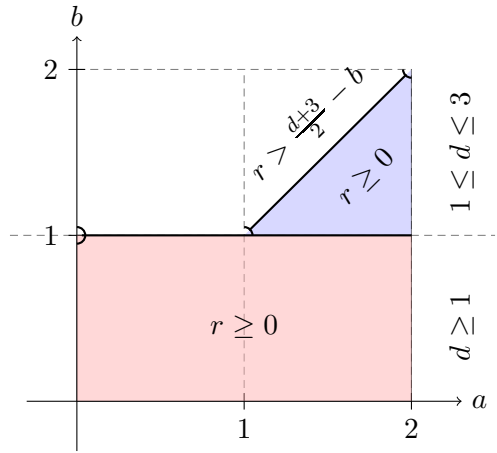

Regarding the noise, we formulate the following assumption in order to cover the Riesz kernel case,
the fractional noise and a mixture of them:
\begin{assumption}[Noise]
	\label{A:Noise}
	Let $k \in \{1, \cdots , d\}$ and partition the $d$-coordinates of $x=(x_1, \cdots, x_d)$ into $k$
	distinct groups of size $d_i$ so that $d_1 + \cdots + d_k = d$. Denote $x_{(i)} = (x_{i_1} ,
	\cdots , x_{i_{d_i}})$ to be the coordinates in the $i^{th}$ partition. Assume that the
	correlation function of the Gaussian noise is given by
	\begin{equation}
		\label{E:corfunction}
		\gamma(x) = \prod_{i=1}^k |x_{(i)}|^{-\alpha_i} \qquad \text{with $\alpha_i \in (0, d_i)$.}
	\end{equation}
	Define $\alpha := \sum_{i=1}^k \alpha_i$.
\end{assumption}

\begin{remark}[Spectral density and decomposition]
	\label{R:Noise}
	Recall that the {\it spectral density} of $\gamma$ from \eqref{E:corfunction}, which by definition is
	$\mathcal{F}\gamma$, takes the following form:
	\begin{equation}
	\label{E:specfunction}
	\mu (\ud \xi) = \varphi(\xi) \ud \xi \quad \text{with} \quad
	\varphi(\xi) = \prod_{i=1}^k C_{\alpha_i,d_i} |\xi_{(i)}|^{-(d_i-\alpha_i)}.
	\end{equation}
	Moreover, in the derivations below, we need to find a nonnegative and nonnegative definite $K$
	such that $\gamma = K*K$ where `$*$' denotes the spatial convolution. Indeed, one can choose
	\begin{align}
	\label{E:K}
	K(x) =\prod_{i=1}^k \beta_{\alpha_i,d_i} \left|x_{(i)}\right|^{-(d_i+\alpha_i)/2}.
	\end{align}
	The two constants in both \eqref{E:specfunction} and \eqref{E:K} are defined as
	\begin{equation}
	\label{E:noiseconstant}
	C_{\alpha,d}     = \pi^{-d/2}2^{-\alpha} \frac{\Gamma((d-\alpha)/2)}{\Gamma(\alpha/2)} \quad \text{and}\quad
	\beta_{\alpha,d} = \pi^{-d/4}\frac{\Gamma((d+\alpha)/4)}{\Gamma((d-\alpha)/4)}\sqrt{ \frac{\Gamma((d-\alpha)/2)}{\Gamma(\alpha/2)}}.
	\end{equation}
\end{remark}

\begin{example}[Noises]
	\label{Eg:NoiseEg}
	We have the following special cases: (1) Setting $k=1$ in \eqref{E:corfunction} and
	\eqref{E:specfunction} recovers the Riesz kernel case. In this case,
	\begin{equation}\label{E:riesznoise}
	\gamma(x)  = |x|^{-\alpha}, \quad
	\varphi(x) = C_{\alpha,d}|x|^{-(d-\alpha)} \quad \text{and} \quad
	K(x)       = \beta_{\alpha,d}|x|^{-(d+\alpha)/2}.
	\end{equation}
	\noindent (2) Setting $k=d$ in \eqref{E:corfunction} and \eqref{E:specfunction} recovers the
	time-independent fractional noise. The corresponding SHE with such noise was earlier studied by Hu
	\cite{Hu01Heat}. For this noise, we have that
	\begin{equation}\label{E:fracnoise}
	\gamma(x)    = \prod_{i=1}^d |x_i|^{-\alpha_i}, \quad
	\varphi(\xi) = \prod_{i=1}^d C_{\alpha_i,1} |\xi_i|^{-(1-\alpha_i)}\quad\text{and}\quad
	K(x)         = \prod_{i=1}^d \beta_{\alpha_i,1}|x_i|^{-(1+\alpha_i)/2}.
	\end{equation}
\end{example}
\bigskip

In a recent work by Balan {\it et al} \cite{BCC21}, the same equation as \eqref{E:SPDE}, but
exclusively for the {\it stochastic wave equation} (SWE), namely, the case when $a=b=\nu=2$ and
$r=0$, has been studied, where both the well-posedness and the exact moment asymptotics have been
obtained.  The corresponding {\it stochastic heat equation} (SHE), namely, the case when $a=2$,
$b=\nu=1$ and $r=0$, has been earlier studied by Hu \cite{Hu01Heat}, but only for the well-posedness
and exclusively for the fractional noise \eqref{E:fracnoise}.  The corresponding moment asymptotics
have been obtained by X. Chen \cite{Chen17AIHP} as a special case by setting $\alpha_0=0$. One may
check Remark 1.9 of Balan {\it et al} \cite{BCC21} for the explicit expressions in terms of notation
of the current paper. In this paper, by working on a more general class of SPDEs, we are able to
interpolate the asymptotics for both SWE and SHE; see Section \ref{SS:EgAsym} below for more
details. Moreover, we give the sharp conditions under which there exists only a local $L^2(\Omega)$
solution.

The moment asymptotics obtained by X. Chen, such as those in \cite{Chen17AIHP,Chen19AIHP}, rely
crucially on the Feynman-Kac representation of the moments of the solution. However, whenever $b\ne
1$, especially for the case when $b\in(1,2)$, we are not aware of any such Feynman-Kac formula for
the moments. Instead, in the recent work by Balan {\it et al} \cite{BCC21}, this difficulty has been
overcome by studying the Weiner chaos expansion of the solution. In this paper, we follow the same
strategy laid out by Balan {\it et al} ({\it ibid.}). The challenge comes from the much more
involved parametric form of the fundamental solution.  \bigskip

Now let us state the main results of this paper. The first main result deals with the well-posedness
of the SPDE \eqref{E:SPDE} (or \eqref{E:spdeint}) as stated in the following theorem. For this, we
need to introduce the following variational constant (see Section \ref{SS:Variation} for more
details):
\begin{align}
	\label{E:M1}
	\mathcal{M}_{a,d}\left(f\right) := \sup_{g \in \mathcal{F}_a} \left\{ \left\langle g^2 * g^2, f \right\rangle^{1/2}_{L^2(\R^d)} - \frac{1}{2}\: \mathcal{E}_a(g,g) \right\}.
\end{align}
We use the convention that $\mathcal{M}_{a}\left(f\right) := \mathcal{M}_{a,d}\left(f\right)$ when
the dimension is clear from the context, and $\mathcal{M}_{a}:=\mathcal{M}_{a}(\gamma)$, where
$\gamma$ is defined in \eqref{E:corfunction}.

\begin{theorem}[Solvability]
	\label{T:solvable}
	Assume that both Assumptions \ref{A:Nonnegative} and \ref{A:Noise} hold.

	\noindent (1) \eqref{E:SPDE} has a unique (global) solution $u(t,x)$ in $L^p(\Omega)$ for all $p \ge 2$,
	$t>0$, and $x \in \R^d$ provided that
	\begin{equation}
		\label{E:alpha}
		0< \alpha < \min\left( \frac{a}{b}\left[2(b+r)-1\right], 2a, d \right).
	\end{equation}

	\noindent (2) Otherwise, if
	\begin{align}
		\label{E:critical}
		r\in \left[0,1/2\right] \qquad \text{and} \qquad
		0<\alpha = \frac{a}{b}[2(b+r)-1] \le d,
	\end{align}
	then \eqref{E:SPDE} has a local solution in the sense that
	\begin{enumerate}
		\item[(2-i)] For any $p\ge 2$,\eqref{E:SPDE} has a unique solution $u(t,x)$ in $L^p(\Omega)$ for
			all $p \ge 2$ and $x \in \R^d$, but only for $t\in (0, T_p)$ where
			\begin{align}
				\label{E:Tp}
				T_p := \dfrac{\nu^{\alpha/a}}{2 \theta (p-1) \mathcal{M}_a^{(2a-\alpha)/a}}.
			\end{align}
		\item[(2-ii)] For any $t>T_2$, the series \eqref{E:2mom} below diverges, that is, the
			$L^2(\Omega)$-solution $u(t,x)$ to \eqref{E:SPDE} does not exist whenever $t>T_2$.
	\end{enumerate}
\end{theorem}

% To facilitate the statement of our results, all the above examples can be put under the following
% assumption. It is interesting to see whether this assumption could actually cover more examples than
% those stated above.
%
% \begin{remark}
% 	\label{R:spec}
% 	We may actually apply the techniques in this paper to any such noise whose correlation function and
% 	spectral density, $\gamma$ and $\mu$, satisfy
% 	\begin{enumerate}
% 		\item both $\mu$ and $\gamma$ are absolutely continuous with respect to Lebesgue measure
% 		\item for some $\alpha \in (0,d)$, $\gamma$ satisfies that for all $c>0$ and $x\in \R^d$ that $\gamma(cx) = c^{-\alpha}\gamma(x)$
% 		\item there exists a nonnegative function $K$ on $\R^d$ such that $\gamma = K*K$ where $"*"$ refers to the spatial convolution.
% 	\end{enumerate}
% 	However, there is an issue with applying the results in this paper to the more general noises
% 	satisfying \ref{R:spec}. Proving $\rho$, defined below, is finite is a challenge. In \cite{BCR09},
% 	$\rho$ is proven to be finite in the case where $\varphi$ is the Riesz kernel. Below in Lemma
% 	\ref{L:rhofinite}, we give a proof showing that $\rho$ is finite for the noise with spectral
% 	density defined above in \eqref{E:specfunction}.
% \end{remark}

% Next, we need to make some assumption on the parameters $a,b,r$ of \eqref{E:SPDE} as well as the
% dimension $d$. Under theses assumptions, the fundamental function $G$ will be nonnegative and smooth
% for $x\ne 0$; see Theorems 4.1 and 4.6 of \cite{CHHH17}.

The second main result of the paper is about the moment asymptotics. We use $\Norm{\cdot}_p$ to
denote the $L^p(\Omega)$ moments.

\begin{theorem}
	\label{T:asym}
	Under Assumptions \ref{A:Nonnegative} and \ref{A:Noise}, if condition \eqref{E:alpha} holds, then we have that
	\begin{equation}
		\label{E:momentasym}
		\begin{aligned}
			\lim_{t_p \to \infty} t_p^{-\beta} \log \Norm{u(t,x)}_p = & \left(\frac{1}{2}\right)\left(\frac{2a}{2a(b+r)- b\alpha} \right)^\beta \\
                                                                & \times \left(\theta\nu^{-\alpha/a} \mathcal{M}_a^{\frac{2a-\alpha}{a}}\right)^{\frac{a}{2a(b+r)-b\alpha-a}}\left(2(b+r)-\frac{b\alpha}{a}-1\right),
		\end{aligned}
	\end{equation}
	where
	\begin{align}
		\label{E:beta-tp}
		\beta := \dfrac{2(b+r)-\dfrac{b\alpha}{a}}{ 2(b+r)-\dfrac{b\alpha}{a} -1 } \qquad \text{and} \qquad
		t_p   := (p-1)^{1-1/\beta} \: t.
	\end{align}
\end{theorem}
\begin{proof}
	We prove the matching upper bound \eqref{E:momentasym-upper} and the lower bound
	\eqref{E:momentasym-lower} of \eqref{E:momentasym} at the end of Sections \ref{S:upperbd} and
	\ref{S:lowerbd} below, respectively, which together prove \eqref{E:momentasym}.
\end{proof}

As a direct consequence of \eqref{E:momentasym}, one can send either $t$ or $p$ to infinity as follows:

\begin{corollary}
	\label{C:Lim}
	Under both Assumptions \ref{A:Nonnegative} and \ref{A:Noise}, if condition \eqref{E:alpha} holds,
	then
	\begin{itemize}
		\item[(1)] For all $p\ge 2$ fixed, it holds that
	\begin{align}
		\label{E:momentasym-t}
		\begin{aligned}
			\lim_{t \to \infty} t^{-\beta} \log \E\left(|u(t,x)|^p\right) = &  p(p-1)^{\frac{1}{2(b+r)-\frac{b\alpha}{a} -1}} \left(\frac{1}{2}\right)\left(\frac{2a}{2a(b+r)- b\alpha} \right)^\beta \\
                                                                      & \times \left(\theta\nu^{-\alpha/a} \mathcal{M}_a^{\frac{2a-\alpha}{a}}\right)^{\frac{a}{2a(b+r)-b\alpha-a}}\left(2(b+r)-\frac{b\alpha}{a}-1\right);
		\end{aligned}
	\end{align}
	\item[(2)] For all $t>0$ fixed, it holds that
	\begin{align}
		\label{E:momentasym-p}
		\begin{aligned}
			\lim_{p \to \infty} p^{-\beta} \log \E\left(|u(t,x)|^p\right) = & t^\beta \left(\frac{1}{2}\right)\left(\frac{2a}{2a(b+r)- b\alpha} \right)^\beta \\
                                                                        & \times \left(\theta\nu^{-\alpha/a} \mathcal{M}_a^{\frac{2a-\alpha}{a}}\right)^{\frac{a}{2a(b+r)-b\alpha-a}}\left(2(b+r)-\frac{b\alpha}{a}-1\right).
		\end{aligned}
	\end{align}
	\end{itemize}
\end{corollary}

\bigskip

The paper is organized as follows. In Section \ref{S:SolvEx}, we first give some concrete examples,
where one can find many explicit formulas for either moment asymptotics in the case of global
solutions or the expressions for the critical time $T_p$ in the case of local solutions. Then in
Section \ref{S:Pre}, we present some preliminaries of the paper, including the {\it Skorohod}
integral, definition of the mild solution, and some asymptotics with corresponding variational
constants. We prove part (1) and part (2) of Theorem \ref{T:solvable} in Sections \ref{S:solvable}
and \ref{S:upperbd}, respectively.  The upper bound and lower bounds for \eqref{E:momentasym} are
established in Sections \ref{S:upperbd} and \ref{S:lowerbd}, respectively. Finally, in the appendix
--- Section \ref{S:Appendix}, we list a few proofs of results that are used in the paper.

\section{Examples on solvability and asymptotics}
\label{S:SolvEx}

In this section, we will give various examples to our main results. The cases with $b=1$ and $r=0$
are mostly known, which will be pointed out in the example below and will be used as test examples
for our results. To the best of our knowledge, all results in this paper for either $b\ne 1,2$ or
$r>0$ should be new.

\subsection{Examples on solvability}

In this part, we list some concrete examples regarding the solvability --- Theorem \ref{T:solvable}.

\begin{example}[SHE]% {{{
	\label{Eg:CriSHE}
	By setting $a=2$, $b=1$ and $r=0$ in \eqref{E:critical}, we obtain the following condition for SHE
	under which there only exits a local solution:
	\begin{align}
		\label{E:CrSHE}
		\alpha = 2 \le d .
	\end{align}
	Clearly, the fundamental solutions in this case are nonnegative for all $d\ge 1$.  Hence, the
	picture is slightly more complicated since we need to check all possible dimensions $d\ge 1$.  We
	illustrate possible cases in Figure \ref{F:SHE}. In particular, let us explain a few cases:
	\begin{enumerate}[(a)]
		\item When $d=2$, condition \eqref{E:CrSHE} says that the $2$-dimensional SHE driven by white
			noise has only a local $L^p(\Omega)$ solution. By applying \eqref{E:Tp} to this case, the critical time $T_p$ becomes
			\begin{align}
				\label{E:Tp-SHE2}
				T_p = \dfrac{\nu}{2\theta (p-1) \mathcal{M}_{2,2}(\delta_0)}, \qquad p\ge 2.
			\end{align}
			Note that in part 2) of Theorem 4.1 of Hu \cite{Hu02Chaos}, some lower and upper bounds for $
			T_2$ were obtained. More precisely, by setting additionally that $\theta=1$ and $\nu=1$, Hu
			({\it ibid.}) proved that when  $t<2$, an  $L^2(\Omega)$ solution exists but when $t>2\pi$,
			the second moment of the solution blows up. It is an interesting exercise to show that
			\begin{align*}
				2\le T_2 = \frac{1}{2 \mathcal{M}_{2,2}(\delta_0)} \le 2\pi, \qquad \text{where $d=2$.}
			\end{align*}
			This case is covered as a special time-independent case (i.e., $H_0=1$) by Chen {\it et al} \cite[Theorem 3.4
			and Remark 3.13]{CDOT20}.
		\item When $d\ge 3$, SHE driven by white noise no longer has any $L^2(\Omega)$-solution. Local
			solutions exist only when $\alpha=2$ and the noise is not white. The critical time $T_p$ takes
			the same expression as \eqref{E:Tp-SHE2} but one needs to replace $\delta_0$ by $\gamma$.
	\end{enumerate}
\end{example}% }}}

\begin{remark}% {{{
	\label{R:Hairer}
	Note that we use the Skorohod integral to interpret the multiplication of the solution with the
	noise in \eqref{E:SPDE}. Multiplication interpreted in this way is traditionally called the {\it
	Wick product} which is consistent to the {\it It\^o} or {\it Walsh} integral (see, e.g.,
	\cite{Minicourse09}) when the noise is white in time. One can also interpret this product as the
	usual product. In order to handle the singularities caused by this multiplication, one needs to
	carry out certain renormalization processes. In fact, for the standard SHE with white noise in
	$\R^d$ (i.e., $a=2$, $b=1$, $\alpha=d$ and $r=0$), Hairer and Labb\'e constructed pathwise
	solutions using the regularity structure for both cases $d=2,3$ in \cite{HairerLabbe15SHE2} and
	\cite{HairerLabbe18SHE3}, respectively. The relation between these two types of solution is left
	for the future work.
\end{remark}% }}}

\begin{figure}[htpb]% {{{ F:SHE
	\centering
	\begin{tikzpicture}[scale=0.9]
		\draw [thick,dashed] (2,0) -- (2,5.5) node [above] {$r=0$};
		\draw [->] (-0.3,0) -- (5.5,0) node [right] {$\alpha$};
		\draw [->] (0,-0.3) -- (0,5.5) node [above] {$d$};
		\foreach \x in {1,...,5}{
			\draw (\x,0.1)--++(0,-0.2) node [below] {$\x$};
			\draw (0.1,\x)--++(-0.2,0) node [left] {$\x$};
			\draw (\x,\x) + (-0.15,-0.15) rectangle +(0.15,0.15);
			\draw [thick] ([shift=(-90:0.1)]0,\x) arc (-90:90:0.1);
			\draw [thin, dotted] (\x,0) -- (\x, 5.4);
		}
		\filldraw (0, 1) + (0.1, -0.05) rectangle +(1,0.05);%\filldraw (1,1) circle (0.05);
		\filldraw (0, 2) + (0.1, -0.05) rectangle +(1.9,0.05);
		\filldraw (0, 3) + (0.1, -0.05) rectangle +(1.9,0.05);
		\filldraw (0, 4) + (0.1, -0.05) rectangle +(1.9,0.05);
		\filldraw (0, 5) + (0.1, -0.05) rectangle +(1.9,0.05);
		\draw (2,2) circle (0.11);
		\draw (2,3) circle (0.11);
		\draw (2,4) circle (0.11);
		\draw (2,5) circle (0.11);
		\draw (2, 3) + (0.1, -0.05) rectangle +(1,0.05);
		\draw (2, 4) + (0.1, -0.05) rectangle +(2,0.05);
		\draw (2, 5) + (0.1, -0.05) rectangle +(3,0.05);
	\end{tikzpicture}
	\hspace{2em}
	\begin{tikzpicture}[scale=0.9]
		\def\x{3}
		\draw (-0.4,-1.5+\x) rectangle +(6,3);
		\filldraw (0, 1+\x) + (0.1, -0.05) rectangle +(1,0.05) node [right] {Global $L^p(\Omega)$-solution};
		\draw (0.5,0+\x) circle (0.1); \node at (3,0+\x) {Local $L^p(\Omega)$-solution};
		\draw (0, -1+\x) + (0.1, -0.05) rectangle +(1,0.05) node [right] {No $L^2(\Omega)$-solution};
		\draw (0.5, 1) + (-0.15,-0.15) rectangle +(0.15,0.15); \node at (2.5,1) {White noise};
		\draw[white] (0,-1.5) circle (0.01);
	\end{tikzpicture}

	\caption{Solvability for the stochastic heat equation (i.e., $a=2$, $b=1$ and $r=0$) with $p\ge 2$.}
	\label{F:SHE}
\end{figure}
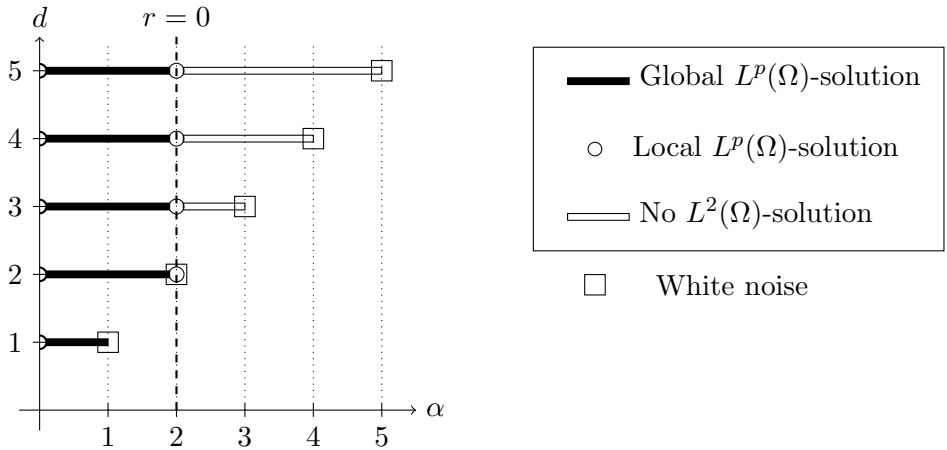% }}}

\begin{example}[SWE]% {{{
	\label{Eg:CriSWE}
	By setting $a=2$ and formally setting $b=2$ in \eqref{E:critical}, we obtain the following
	condition for the stochastic wave equation under which there only exits a local solution:
	\begin{align}
		\label{E:CrSWE}
		\alpha=3 + 2 r\le d \quad \text{and} \quad r\in[0,1/2].
	\end{align}
	First of all, results in Balan {\it et al} \cite{BCC21} require $d\le 3$. On the other hand,
	Assumption \ref{A:Nonnegative} and all known sufficient conditions for the nonnegativity of the
	fundamental solution (see Remark \ref{R:Nonnegative}) also require $d\le 3$ in case of $b\in
	(1,2)$. With this restriction, conditions \eqref{E:CrSWE} reduce to
	\begin{align*}
		\alpha=3=d \quad \text{and} \quad r=0,
	\end{align*}
	which says that at dimension $d=3$, when $\W$ is a white noise, there exits only a local
	$L^p(\Omega)$ solution for all $p\ge 2$. See Figure \ref{F:SWE} for an illustration. Moreover, one
	can check easily that the expression for the critical time $T_p$ in \eqref{E:Tp} in this case
	reduces to
	\begin{align}
		\label{E:TpCrSWE}
		T_p=	\frac{\nu^{3 / 2}}{2\theta (p-1) \sqrt{\mathcal{M}_{2,3}(\delta_0)}},  \quad p\ge 2,
	\end{align}
	which is identical to (1.12) ({\it ibid.}) when setting $\nu=2$. \bigskip
\end{example}% }}}

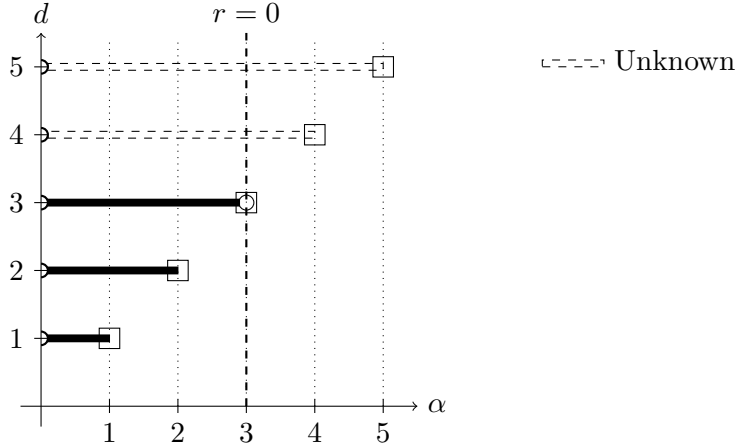
\begin{figure}[htb]% {{{ F:SWE
	\centering
	\begin{tikzpicture}[scale=0.9]
		\draw [thick,dashed] (3,0) -- (3,5.5) node [above] {$r=0$};
		\draw [->] (-0.3,0) -- (5.5,0) node [right] {$\alpha$};
		\draw [->] (0,-0.3) -- (0,5.5) node [above] {$d$};
		\foreach \x in {1,...,5}{
			\draw (\x,0.1)--++(0,-0.2) node [below] {$\x$};
			\draw (0.1,\x)--++(-0.2,0) node [left] {$\x$};
			\draw (\x,\x) + (-0.15,-0.15) rectangle +(0.15,0.15);
			\draw [thick] ([shift=(-90:0.1)]0,\x) arc (-90:90:0.1);
			\draw [thin, dotted] (\x,0) -- (\x, 5.4);
		}
		\filldraw (0, 1) + (0.1, -0.05) rectangle +(1,0.05);%\filldraw (1,1) circle (0.05);
		\filldraw (0, 2) + (0.1, -0.05) rectangle +(2,0.05);%\filldraw (2,2) circle (0.05);
		\filldraw (0, 3) + (0.1, -0.05) rectangle +(2.9,0.05);%\filldraw (3,3) circle (0.05);
		% \filldraw (0, 4) + (0.1, -0.05) rectangle +(2.9,0.05);
		% \filldraw (0, 5) + (0.1, -0.05) rectangle +(2.9,0.05);
		% \draw (2,2) circle (0.11);
		\draw (3,3) circle (0.11);
		% \draw (3,4) circle (0.11);
		% \draw (3,5) circle (0.11);
		% \draw (2, 3) + (0.1, -0.05) rectangle +(1,0.05);
		% \draw (2, 4) + (0.1, -0.05) rectangle +(2,0.05);
		\draw[dashed] (0, 5) + (0.1, -0.05) rectangle +(5,0.05);
		\draw[dashed] (0, 4) + (0.1, -0.05) rectangle +(4,0.05);
	\end{tikzpicture}
	\hspace{2em}
	\begin{tikzpicture}[scale=0.9]
		\def\x{3.2}
		% \draw (-0.4,-1.5+\x) rectangle +(6,3);
		\draw[dashed] (0, 1+\x) + (0.1, -0.05) rectangle +(1,0.05) node [right] {Unknown};
		% \draw (0.5,0+\x) circle (0.1); \node at (3,0+\x) {Local $L^p(\Omega)$-solution};
		% \draw (0, -1+\x) + (0.1, -0.05) rectangle +(1,0.05) node [right] {No $L^2(\Omega)$-solution};
		% \draw (0.5, 1) + (-0.15,-0.15) rectangle +(0.15,0.15); \node at (2.5,1) {White noise};
		\draw[white] (0,-1.5) circle (0.01);
	\end{tikzpicture}

	\caption{Solvability for the stochastic wave equation (i.e., $a=b=2$ and $r=0$). See Figure
	\ref{F:SHE} for an additional legend.}
	\label{F:SWE}
\end{figure}% }}}

\begin{example}[Fractional SPDEs with $r=\Ceil{b}-b$ and $a=2$]% {{{
	\label{Eg:FracSHE}
	For the fractional SDPEs with $b\ne 1$, many known works focus on the case when $r=\Ceil{b}-b$,
	where $\Ceil{b}$ is the ceiling function; see, e,g., \cite{Chen17Tran,MijenaNane15SPA}. To
	facilitate the discussions here, we will only focus on the case when $a=2$. In particular, by
	setting $r=\Ceil{b}-b$ and $a=2$, conditions in \eqref{E:critical} become
	\begin{align}
		\label{E:CrFrSHE}
		\begin{cases}
			\displaystyle \alpha =\frac{2}{b} \le d \quad  \text{and} \quad b\in [1/2,1], \bigskip\\
			\displaystyle \alpha =\frac{6}{b} \le d \quad  \text{and} \quad b\in [3/2,2).
		\end{cases}
	\end{align}
	When $b=1$, we have $r=0$ and the fundamental solution is the standard heat kernel. Hence,
	Assumption \ref{A:Nonnegative} is satisfied for all $d\ge 1$. When $b<1$, sufficient conditions in
	Remark \ref{R:Nonnegative} guarantees Assumption \ref{A:Nonnegative} for all $d\ge 1$. However,
	when $b>1$ and  $a=2$, from Remark  \ref{R:Nonnegative} we see that the fundamental solution is
	nonnegative only for $d\le 3$. The solvability for this case is illustrated in Figure
	\ref{F:FracSHE} and the critical time $T_p$ in case of local solution (hence, only for the case
	when $b\in[1/2,1]$) is equal to
	\begin{align}
		\label{E:TpFracSHE}
		T_p = \dfrac{\nu^{\alpha/2}}{2 \theta (p-1) \mathcal{M}_{2,d}^{2-\alpha/2}}.
	\end{align}
\end{example}% }}}

\begin{figure}[htpb]% {{{ F:FracSHE
	\centering
	\begin{tikzpicture}[scale=0.9]
		\draw [thick,dashed] (4,0) -- (4,5.5) node [above] {$r=1/2$};
		\draw [thick,dashed] (4,0) -- (4,6) node [above] {$b=3/2$};
		\draw [->] (-0.3,0) -- (5.5,0) node [right] {$\alpha$};
		\draw [->] (0,-0.3) -- (0,5.5) node [above] {$d$};
		\foreach \x in {1,...,5}{
			\draw (\x,0.1)--++(0,-0.2) node [below] {$\x$};
			\draw (0.1,\x)--++(-0.2,0) node [left] {$\x$};
			\draw (\x,\x) + (-0.15,-0.15) rectangle +(0.15,0.15);
			\draw [thick] ([shift=(-90:0.1)]0,\x) arc (-90:90:0.1);
			\draw [thin, dotted] (\x,0) -- (\x, 5.4);
		}
		\filldraw (0, 1) + (0.1, -0.05) rectangle +(1,0.05);%\filldraw (1,1) circle (0.05);
		\filldraw (0, 2) + (0.1, -0.05) rectangle +(2,0.05);
		\filldraw (0, 3) + (0.1, -0.05) rectangle +(3,0.05);
		\draw[thick,dashed] (0, 4) + (0.1, -0.05) rectangle +(4,0.05);
		\draw[thick,dashed] (0, 5) + (0.1, -0.05) rectangle +(5,0.05);
		% \draw (2,2) circle (0.11);
		% \draw (2,3) circle (0.11);
		% \draw[thick,dotted] (4,4) circle (0.11);
		% \draw (4,5) circle (0.11);
		% \draw (2, 3) + (0.1, -0.05) rectangle +(1,0.05);
		% \draw (2, 4) + (0.1, -0.05) rectangle +(2,0.05);
		% \draw (2, 5) + (0.1, -0.05) rectangle +(3,0.05);
	\end{tikzpicture}
	\qquad \qquad
	\begin{tikzpicture}[scale=0.9]
		\draw [thick,dashed] (3.5,0) -- (3.5,6) node [above] {$b=12/7$};
		\draw [thick,dashed] (3.5,0) -- (3.5,5.5) node [above] {$r=2/7$};
		\draw [->] (-0.3,0) -- (5.5,0) node [right] {$\alpha$};
		\draw [->] (0,-0.3) -- (0,5.5) node [above] {$d$};
		\foreach \x in {1,...,5}{
			\draw (\x,0.1)--++(0,-0.2) node [below] {$\x$};
			\draw (0.1,\x)--++(-0.2,0) node [left] {$\x$};
			\draw (\x,\x) + (-0.15,-0.15) rectangle +(0.15,0.15);
			\draw [thick] ([shift=(-90:0.1)]0,\x) arc (-90:90:0.1);
			\draw [thin, dotted] (\x,0) -- (\x, 5.4);
		}
		\filldraw (0, 1) + (0.1, -0.05) rectangle +(1,0.05);%\filldraw (1,1) circle (0.05);
		\filldraw (0, 2) + (0.1, -0.05) rectangle +(2,0.05);%\filldraw (2,2) circle (0.05);
		\filldraw (0, 3) + (0.1, -0.05) rectangle +(3,0.05);
		\draw[thick,dashed] (0, 4) + (0.1, -0.05) rectangle +(4,0.05);
		\draw[thick,dashed] (0, 5) + (0.1, -0.05) rectangle +(5,0.05);
		% \draw (2.5,2) circle (0.11);
		% \draw (2.5,3) circle (0.11);
		% \draw (2.5,4) circle (0.11);
		% \draw (2.5,5) circle (0.11);
		% \draw (2.5, 3) + (0.1, -0.05) rectangle +(0.5,0.05);
		% \draw (2.5, 4) + (0.1, -0.05) rectangle +(1.5,0.05);
		% \draw (2.5, 5) + (0.1, -0.05) rectangle +(2.5,0.05);
	\end{tikzpicture}
	\bigskip

	\begin{tikzpicture}[scale=0.9]
		\draw [thick,dashed] (2,0) -- (2,5.5) node [above] {$r=0$};
		\draw [thick,dashed] (2,0) -- (2,6) node [above] {$b=1$};
		\draw [->] (-0.3,0) -- (5.5,0) node [right] {$\alpha$};
		\draw [->] (0,-0.3) -- (0,5.5) node [above] {$d$};
		\foreach \x in {1,...,5}{
			\draw (\x,0.1)--++(0,-0.2) node [below] {$\x$};
			\draw (0.1,\x)--++(-0.2,0) node [left] {$\x$};
			\draw (\x,\x) + (-0.15,-0.15) rectangle +(0.15,0.15);
			\draw [thick] ([shift=(-90:0.1)]0,\x) arc (-90:90:0.1);
			\draw [thin, dotted] (\x,0) -- (\x, 5.4);
		}
		\filldraw (0, 1) + (0.1, -0.05) rectangle +(1,0.05);%\filldraw (1,1) circle (0.05);
		\filldraw (0, 2) + (0.1, -0.05) rectangle +(1.9,0.05);
		\filldraw (0, 3) + (0.1, -0.05) rectangle +(1.9,0.05);
		\filldraw (0, 4) + (0.1, -0.05) rectangle +(1.9,0.05);
		\filldraw (0, 5) + (0.1, -0.05) rectangle +(1.9,0.05);
		\draw (2,2) circle (0.11);
		\draw (2,3) circle (0.11);
		\draw (2,4) circle (0.11);
		\draw (2,5) circle (0.11);
		\draw (2, 3) + (0.1, -0.05) rectangle +(1,0.05);
		\draw (2, 4) + (0.1, -0.05) rectangle +(2,0.05);
		\draw (2, 5) + (0.1, -0.05) rectangle +(3,0.05);
	\end{tikzpicture}
	\qquad \qquad
	\begin{tikzpicture}[scale=0.9]
		\draw [thick,dashed] (2.5,0) -- (2.5,6) node [above] {$b=4/5$};
		\draw [thick,dashed] (2.5,0) -- (2.5,5.5) node [above] {$r=1/5$};
		\draw [->] (-0.3,0) -- (5.5,0) node [right] {$\alpha$};
		\draw [->] (0,-0.3) -- (0,5.5) node [above] {$d$};
		\foreach \x in {1,...,5}{
			\draw (\x,0.1)--++(0,-0.2) node [below] {$\x$};
			\draw (0.1,\x)--++(-0.2,0) node [left] {$\x$};
			\draw (\x,\x) + (-0.15,-0.15) rectangle +(0.15,0.15);
			\draw [thick] ([shift=(-90:0.1)]0,\x) arc (-90:90:0.1);
			\draw [thin, dotted] (\x,0) -- (\x, 5.4);
		}
		\filldraw (0, 1) + (0.1, -0.05) rectangle +(1,0.05);%\filldraw (1,1) circle (0.05);
		\filldraw (0, 2) + (0.1, -0.05) rectangle +(2,0.05);%\filldraw (2,2) circle (0.05);
		\filldraw (0, 3) + (0.1, -0.05) rectangle +(2.4,0.05);
		\filldraw (0, 4) + (0.1, -0.05) rectangle +(2.4,0.05);
		\filldraw (0, 5) + (0.1, -0.05) rectangle +(2.4,0.05);
		% \draw (2.5,2) circle (0.11);
		\draw (2.5,3) circle (0.11);
		\draw (2.5,4) circle (0.11);
		\draw (2.5,5) circle (0.11);
		\draw (2.5, 3) + (0.1, -0.05) rectangle +(0.5,0.05);
		\draw (2.5, 4) + (0.1, -0.05) rectangle +(1.5,0.05);
		\draw (2.5, 5) + (0.1, -0.05) rectangle +(2.5,0.05);
	\end{tikzpicture}
	\bigskip

	\begin{tikzpicture}[scale=0.9]
		\draw [thick,dashed] (3,0) -- (3,6) node [above] {$b=2/3$};
		\draw [thick,dashed] (3,0) -- (3,5.5) node [above] {$r=1/3$};
		\draw [->] (-0.3,0) -- (5.5,0) node [right] {$\alpha$};
		\draw [->] (0,-0.3) -- (0,5.5) node [above] {$d$};
		\foreach \x in {1,...,5}{
			\draw (\x,0.1)--++(0,-0.2) node [below] {$\x$};
			\draw (0.1,\x)--++(-0.2,0) node [left] {$\x$};
			\draw (\x,\x) + (-0.15,-0.15) rectangle +(0.15,0.15);
			\draw [thick] ([shift=(-90:0.1)]0,\x) arc (-90:90:0.1);
			\draw [thin, dotted] (\x,0) -- (\x, 5.4);
		}
		\filldraw (0, 1) + (0.1, -0.05) rectangle +(1,0.05);%\filldraw (1,1) circle (0.05);
		\filldraw (0, 2) + (0.1, -0.05) rectangle +(2,0.05);%\filldraw (2,2) circle (0.05);
		\filldraw (0, 3) + (0.1, -0.05) rectangle +(2.9,0.05);
		\filldraw (0, 4) + (0.1, -0.05) rectangle +(2.9,0.05);
		\filldraw (0, 5) + (0.1, -0.05) rectangle +(2.9,0.05);
		% \draw (2,2) circle (0.11);
		\draw (3,3) circle (0.11);
		\draw (3,4) circle (0.11);
		\draw (3,5) circle (0.11);
		% \draw (3, 3) + (0.1, -0.05) rectangle +(1,0.05);
		\draw (3, 4) + (0.1, -0.05) rectangle +(1,0.05);
		\draw (3, 5) + (0.1, -0.05) rectangle +(2,0.05);
	\end{tikzpicture}
	\qquad \qquad
	\begin{tikzpicture}[scale=0.9]
		\draw [thick,dashed] (4,0) -- (4,6) node [above] {$b=1/2$};
		\draw [thick,dashed] (4,0) -- (4,5.5) node [above] {$r=1/2$};
		\draw [->] (-0.3,0) -- (5.5,0) node [right] {$\alpha$};
		\draw [->] (0,-0.3) -- (0,5.5) node [above] {$d$};
		\foreach \x in {1,...,5}{
			\draw (\x,0.1)--++(0,-0.2) node [below] {$\x$};
			\draw (0.1,\x)--++(-0.2,0) node [left] {$\x$};
			\draw (\x,\x) + (-0.15,-0.15) rectangle +(0.15,0.15);
			\draw [thick] ([shift=(-90:0.1)]0,\x) arc (-90:90:0.1);
			\draw [thin, dotted] (\x,0) -- (\x, 5.4);
		}
		\filldraw (0, 1) + (0.1, -0.05) rectangle +(1,0.05);%\filldraw (1,1) circle (0.05);
		\filldraw (0, 2) + (0.1, -0.05) rectangle +(2,0.05);%\filldraw (2,2) circle (0.05);
		\filldraw (0, 3) + (0.1, -0.05) rectangle +(3,0.05);%\filldraw (3,3) circle (0.05);
		\filldraw (0, 4) + (0.1, -0.05) rectangle +(3.9,0.05);
		\filldraw (0, 5) + (0.1, -0.05) rectangle +(3.9,0.05);
		% \draw (2,2) circle (0.11);
		% \draw (2,3) circle (0.11);
		\draw (4,4) circle (0.11);
		\draw (4,5) circle (0.11);
		% \draw (2, 3) + (0.1, -0.05) rectangle +(1,0.05);
		% \draw (2, 4) + (0.1, -0.05) rectangle +(2,0.05);
		\draw (4, 5) + (0.1, -0.05) rectangle +(1,0.05);
	\end{tikzpicture}

	\caption{Solvability for the fractional SPDEs in case of $a=2$ and $r=\Ceil{b}-b$. See Figures
	\ref{F:SHE} and \ref{F:SWE} for the legend.}
	\label{F:FracSHE}
\end{figure}% }}}

\begin{example}[Fractional SPDEs with $r=0$ and $a=2$]% {{{
	\label{Eg:FracSHE0}
	In this example, we study the special case of the fractional SPDEs when $r=0$.  The choice of
	$r=0$ has been used in, e.g., \cite{CHHH17}.  We will only consider the case $a=2$ for simplicity.
	Now by setting $r=0$ and $a=2$ and restricting $b\le 1$, conditions in \eqref{E:critical} become
	\begin{align}
		\label{E:CrFrSHE}
		\alpha =4- \frac{2}{b} \le d \quad  \text{and} \quad b\in (0,1].
	\end{align}
	As discussed in Example \ref{Eg:FracSHE}, Assumption \ref{A:Nonnegative} is satisfied for all
	$d\ge 1$ when $b\le 1$ but only for  $d\le 3$ when  $b>1$. The solvability for this case is
	illustrated in Figure \ref{F:FracSHE0} with $T_p$ given in \eqref{E:TpFracSHE}. In particular, for
	the example in the second figure in Figure \ref{F:FracSHE0}, namely, when $b=2/3$ and
	$\alpha=d=1$, the white noise driven SHE has a local solution with
	\begin{align}
		T_p= \frac{2^{5/2}\sqrt{\nu}}{3(p-1)\theta}, \qquad \text{for all $p\ge 2$,}
	\end{align}
	where we have applied \eqref{E:TpFracSHE} and the relation \eqref{E:Mdelta1d}.
\end{example}% }}}

More examples regarding the solvability can be studied in a similar way, which are left to the
interested readers.

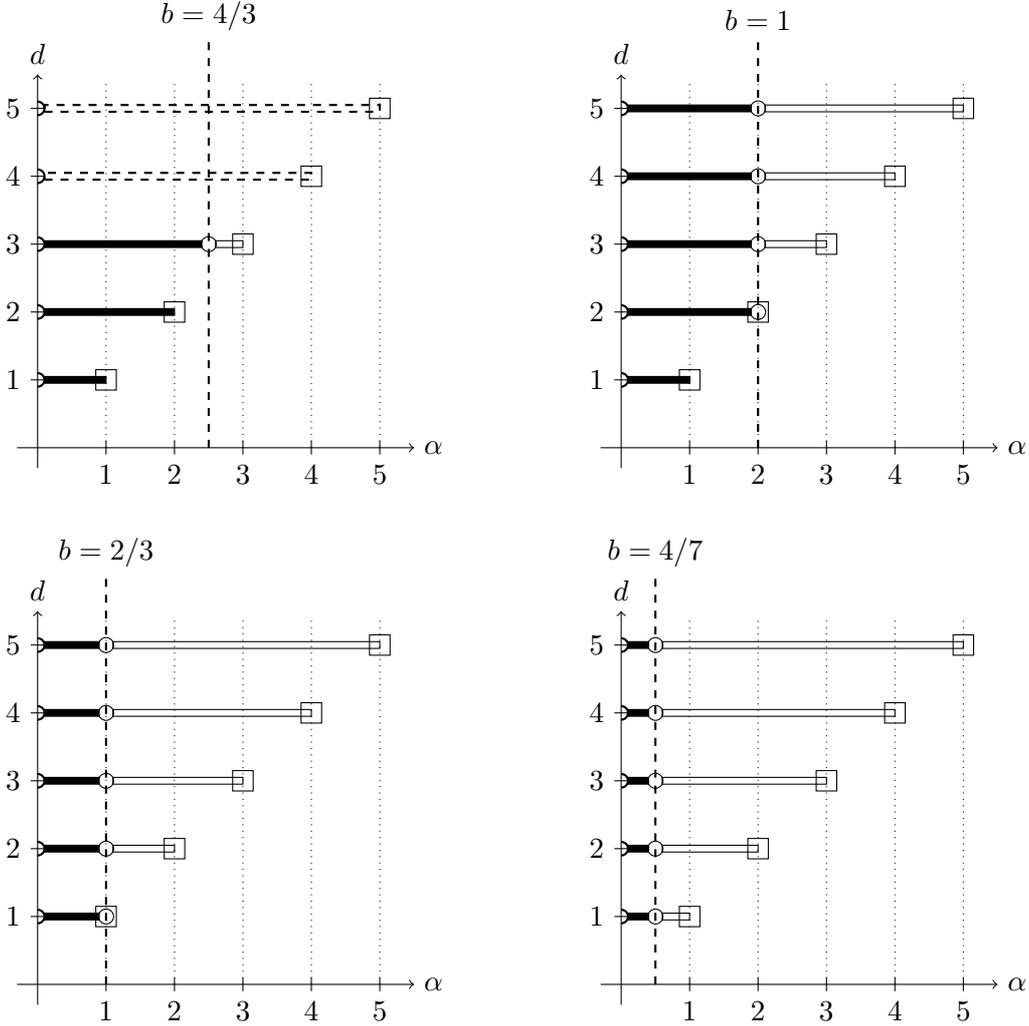
\begin{figure}[htpb]% {{{ F:FracSHE0
	\centering
	\begin{tikzpicture}[scale=0.9]
		% \draw [thick,dashed] (2,0) -- (2,5.5) node [above] {$r=0$};
		\draw [thick,dashed] (2.5,0) -- (2.5,6) node [above] {$b=4/3$};
		\draw [->] (-0.3,0) -- (5.5,0) node [right] {$\alpha$};
		\draw [->] (0,-0.3) -- (0,5.5) node [above] {$d$};
		\foreach \x in {1,...,5}{
			\draw (\x,0.1)--++(0,-0.2) node [below] {$\x$};
			\draw (0.1,\x)--++(-0.2,0) node [left] {$\x$};
			\draw (\x,\x) + (-0.15,-0.15) rectangle +(0.15,0.15);
			\draw [thick] ([shift=(-90:0.1)]0,\x) arc (-90:90:0.1);
			\draw [thin, dotted] (\x,0) -- (\x, 5.4);
		}
		\filldraw (0, 1) + (0.1, -0.05) rectangle +(1,0.05);%\filldraw (1,1) circle (0.05);
		\filldraw (0, 2) + (0.1, -0.05) rectangle +(2,0.05);
		\filldraw (0, 3) + (0.1, -0.05) rectangle +(2.4,0.05);
		\draw[thick,dashed] (0, 4) + (0.1, -0.05) rectangle +(4,0.05);
		\draw[thick,dashed] (0, 5) + (0.1, -0.05) rectangle +(5,0.05);
		\draw (2.5,3) circle (0.11);
		% \draw (2.5,4) circle (0.11);
		% \draw (2.5,5) circle (0.11);
		\draw (2.5,3) + (0.1, -0.05) rectangle +(0.5,0.05);
		% \draw (2.5,4) + (0.1, -0.05) rectangle +(1.5,0.05);
		% \draw (2.5,5) + (0.1, -0.05) rectangle +(2.5,0.05);
	\end{tikzpicture}
	\qquad \qquad
	\begin{tikzpicture}[scale=0.9]
		% \draw [thick,dashed] (2,0) -- (2,5.5) node [above] {$r=0$};
		\draw [thick,dashed] (2,0) -- (2,6) node [above] {$b=1$};
		\draw [->] (-0.3,0) -- (5.5,0) node [right] {$\alpha$};
		\draw [->] (0,-0.3) -- (0,5.5) node [above] {$d$};
		\foreach \x in {1,...,5}{
			\draw (\x,0.1)--++(0,-0.2) node [below] {$\x$};
			\draw (0.1,\x)--++(-0.2,0) node [left] {$\x$};
			\draw (\x,\x) + (-0.15,-0.15) rectangle +(0.15,0.15);
			\draw [thick] ([shift=(-90:0.1)]0,\x) arc (-90:90:0.1);
			\draw [thin, dotted] (\x,0) -- (\x, 5.4);
		}
		\filldraw (0, 1) + (0.1, -0.05) rectangle +(1,0.05);%\filldraw (1,1) circle (0.05);
		\filldraw (0, 2) + (0.1, -0.05) rectangle +(1.9,0.05);
		\filldraw (0, 3) + (0.1, -0.05) rectangle +(1.9,0.05);
		\filldraw (0, 4) + (0.1, -0.05) rectangle +(1.9,0.05);
		\filldraw (0, 5) + (0.1, -0.05) rectangle +(1.9,0.05);
		\draw (2,2) circle (0.11);
		\draw (2,3) circle (0.11);
		\draw (2,4) circle (0.11);
		\draw (2,5) circle (0.11);
		\draw (2, 3) + (0.1, -0.05) rectangle +(1,0.05);
		\draw (2, 4) + (0.1, -0.05) rectangle +(2,0.05);
		\draw (2, 5) + (0.1, -0.05) rectangle +(3,0.05);
	\end{tikzpicture}

	\bigskip
	\begin{tikzpicture}[scale=0.9]
		\draw [thick,dashed] (1,0) -- (1,6) node [above] {$b=2/3$};
		% \draw [thick,dashed] (0.5,0) -- (0.5,5.5) node [above] {$r=0$};
		\draw [->] (-0.3,0) -- (5.5,0) node [right] {$\alpha$};
		\draw [->] (0,-0.3) -- (0,5.5) node [above] {$d$};
		\foreach \x in {1,...,5}{
			\draw (\x,0.1)--++(0,-0.2) node [below] {$\x$};
			\draw (0.1,\x)--++(-0.2,0) node [left] {$\x$};
			\draw (\x,\x) + (-0.15,-0.15) rectangle +(0.15,0.15);
			\draw [thick] ([shift=(-90:0.1)]0,\x) arc (-90:90:0.1);
			\draw [thin, dotted] (\x,0) -- (\x, 5.4);
		}
		\filldraw (0, 1) + (0.1, -0.05) rectangle +(0.9,0.05);%\filldraw (1,1) circle (0.05);
		\filldraw (0, 2) + (0.1, -0.05) rectangle +(0.9,0.05);%\filldraw (2,2) circle (0.05);
		\filldraw (0, 3) + (0.1, -0.05) rectangle +(0.9,0.05);
		\filldraw (0, 4) + (0.1, -0.05) rectangle +(0.9,0.05);
		\filldraw (0, 5) + (0.1, -0.05) rectangle +(0.9,0.05);
		% \draw (2.5,2) circle (0.11);
		\draw (1,1) circle (0.11);
		\draw (1,2) circle (0.11);
		\draw (1,3) circle (0.11);
		\draw (1,4) circle (0.11);
		\draw (1,5) circle (0.11);
		% \draw (1, 1) + (0.1, -0.05) rectangle +(0.5,0.05);
		\draw (1, 2) + (0.1, -0.05) rectangle +(1,0.05);
		\draw (1, 3) + (0.1, -0.05) rectangle +(2,0.05);
		\draw (1, 4) + (0.1, -0.05) rectangle +(3,0.05);
		\draw (1, 5) + (0.1, -0.05) rectangle +(4,0.05);
	\end{tikzpicture}
	\qquad \qquad
	\begin{tikzpicture}[scale=0.9]
		\draw [thick,dashed] (0.5,0) -- (0.5,6) node [above] {$b=4/7$};
		% \draw [thick,dashed] (0.5,0) -- (0.5,5.5) node [above] {$r=0$};
		\draw [->] (-0.3,0) -- (5.5,0) node [right] {$\alpha$};
		\draw [->] (0,-0.3) -- (0,5.5) node [above] {$d$};
		\foreach \x in {1,...,5}{
			\draw (\x,0.1)--++(0,-0.2) node [below] {$\x$};
			\draw (0.1,\x)--++(-0.2,0) node [left] {$\x$};
			\draw (\x,\x) + (-0.15,-0.15) rectangle +(0.15,0.15);
			\draw [thick] ([shift=(-90:0.1)]0,\x) arc (-90:90:0.1);
			\draw [thin, dotted] (\x,0) -- (\x, 5.4);
		}
		\filldraw (0, 1) + (0.1, -0.05) rectangle +(0.4,0.05);%\filldraw (1,1) circle (0.05);
		\filldraw (0, 2) + (0.1, -0.05) rectangle +(0.4,0.05);%\filldraw (2,2) circle (0.05);
		\filldraw (0, 3) + (0.1, -0.05) rectangle +(0.4,0.05);
		\filldraw (0, 4) + (0.1, -0.05) rectangle +(0.4,0.05);
		\filldraw (0, 5) + (0.1, -0.05) rectangle +(0.4,0.05);
		% \draw (2.5,2) circle (0.11);
		\draw (0.5,1) circle (0.11);
		\draw (0.5,2) circle (0.11);
		\draw (0.5,3) circle (0.11);
		\draw (0.5,4) circle (0.11);
		\draw (0.5,5) circle (0.11);
		\draw (0.5, 1) + (0.1, -0.05) rectangle +(0.5,0.05);
		\draw (0.5, 2) + (0.1, -0.05) rectangle +(1.5,0.05);
		\draw (0.5, 3) + (0.1, -0.05) rectangle +(2.5,0.05);
		\draw (0.5, 4) + (0.1, -0.05) rectangle +(3.5,0.05);
		\draw (0.5, 5) + (0.1, -0.05) rectangle +(4.5,0.05);
	\end{tikzpicture}
	\caption{Solvability for the fractional SPDEs in case of $a=2$ and $r=0$. See Figures \ref{F:SHE} and \ref{F:SWE} for the legend.}
	\label{F:FracSHE0}
\end{figure}% }}}

\begin{example}[SHE with fractional Laplace]% {{{
	\label{Eg:StableSHE}
	The stochastic heat equation with fractional Laplace (i.e., the case when $b=1$,  $r=0$ and  $a\in
	(0,2]$) has been widely studied in the literature, but possibly with different noises. In this
	case, the fundamental solutions are transition densities for the alpha-stable jump processes,
	which are necessarily to be nonnegative. This is also consistent with the sufficient conditions
	for nonnegativity in Remark \ref{R:Nonnegative}. By setting $b=1$ and $r=0$ in \eqref{E:critical},
	we have the following condition:
	\begin{align*}
		\alpha = a \le d
	\end{align*}
	The solvability for this case is illustrated in Figure \ref{F:Stable}. In particular, when
	$\alpha=a=d=2$, there is only a local solution under white noise. In this case, using \eqref{E:Tp}
	and \eqref{E:Mdelta1d}, we see that
	\begin{align}
		\label{E:StableSHE-Tp}
		% T_p = \frac{2^{4/3}\nu}{3^{2/3}\theta (p-1)}, \qquad \text{for $p\ge 2$}.
		T_p = \frac{\nu}{2\theta (p-1)\mathcal{M}_{2,2}(\delta_0)}, \qquad \text{for $p\ge 2$}.
	\end{align}
\end{example}% }}}

\begin{figure}[htpb]% {{{ F:Stable
	\centering
	\begin{tikzpicture}[scale=0.9]
		% \draw [thick,dashed] (2,0) -- (2,5.5) node [above] {$r=0$};
		\draw [thick,dashed] (0.5,0) -- (0.5,6) node [above] {$a=1/2$};
		\draw [->] (-0.3,0) -- (5.5,0) node [right] {$\alpha$};
		\draw [->] (0,-0.3) -- (0,5.5) node [above] {$d$};
		\foreach \x in {1,...,5}{
			\draw (\x,0.1)--++(0,-0.2) node [below] {$\x$};
			\draw (0.1,\x)--++(-0.2,0) node [left] {$\x$};
			\draw (\x,\x) + (-0.15,-0.15) rectangle +(0.15,0.15);
			\draw [thick] ([shift=(-90:0.1)]0,\x) arc (-90:90:0.1);
			\draw [thin, dotted] (\x,0) -- (\x, 5.4);
		}
		\filldraw (0, 1) + (0.1, -0.05) rectangle +(0.4,0.05);
		\filldraw (0, 2) + (0.1, -0.05) rectangle +(0.4,0.05);
		\filldraw (0, 3) + (0.1, -0.05) rectangle +(0.4,0.05);
		\filldraw (0, 4) + (0.1, -0.05) rectangle +(0.4,0.05);
		\filldraw (0, 5) + (0.1, -0.05) rectangle +(0.4,0.05);
		\draw (0.5,1) circle (0.11);
		\draw (0.5,2) circle (0.11);
		\draw (0.5,3) circle (0.11);
		\draw (0.5,4) circle (0.11);
		\draw (0.5,5) circle (0.11);
		\draw (0.5,1) + (0.1, -0.05) rectangle +(0.5,0.05);
		\draw (0.5,2) + (0.1, -0.05) rectangle +(1.5,0.05);
		\draw (0.5,3) + (0.1, -0.05) rectangle +(2.5,0.05);
		\draw (0.5,4) + (0.1, -0.05) rectangle +(3.5,0.05);
		\draw (0.5,5) + (0.1, -0.05) rectangle +(4.5,0.05);
	\end{tikzpicture}
	\qquad \qquad
	\begin{tikzpicture}[scale=0.9]
		% \draw [thick,dashed] (2,0) -- (2,5.5) node [above] {$r=0$};
		\draw [thick,dashed] (1,0) -- (1,6) node [above] {$a=1$};
		\draw [->] (-0.3,0) -- (5.5,0) node [right] {$\alpha$};
		\draw [->] (0,-0.3) -- (0,5.5) node [above] {$d$};
		\foreach \x in {1,...,5}{
			\draw (\x,0.1)--++(0,-0.2) node [below] {$\x$};
			\draw (0.1,\x)--++(-0.2,0) node [left] {$\x$};
			\draw (\x,\x) + (-0.15,-0.15) rectangle +(0.15,0.15);
			\draw [thick] ([shift=(-90:0.1)]0,\x) arc (-90:90:0.1);
			\draw [thin, dotted] (\x,0) -- (\x, 5.4);
		}
		\filldraw (0, 1) + (0.1, -0.05) rectangle +(0.9,0.05);%\filldraw (1,1) circle (0.05);
		\filldraw (0, 2) + (0.1, -0.05) rectangle +(0.9,0.05);
		\filldraw (0, 3) + (0.1, -0.05) rectangle +(0.9,0.05);
		\filldraw (0, 4) + (0.1, -0.05) rectangle +(0.9,0.05);
		\filldraw (0, 5) + (0.1, -0.05) rectangle +(0.9,0.05);
		\draw (1,1) circle (0.11);
		\draw (1,2) circle (0.11);
		\draw (1,3) circle (0.11);
		\draw (1,4) circle (0.11);
		\draw (1,5) circle (0.11);
		\draw (1, 2) + (0.1, -0.05) rectangle +(1,0.05);
		\draw (1, 3) + (0.1, -0.05) rectangle +(2,0.05);
		\draw (1, 4) + (0.1, -0.05) rectangle +(3,0.05);
		\draw (1, 5) + (0.1, -0.05) rectangle +(4,0.05);
	\end{tikzpicture}

	\bigskip
	\begin{tikzpicture}[scale=0.9]
		\draw [thick,dashed] (1.5,0) -- (1.5,6) node [above] {$a=3/2$};
		% \draw [thick,dashed] (0.5,0) -- (0.5,5.5) node [above] {$r=0$};
		\draw [->] (-0.3,0) -- (5.5,0) node [right] {$\alpha$};
		\draw [->] (0,-0.3) -- (0,5.5) node [above] {$d$};
		\foreach \x in {1,...,5}{
			\draw (\x,0.1)--++(0,-0.2) node [below] {$\x$};
			\draw (0.1,\x)--++(-0.2,0) node [left] {$\x$};
			\draw (\x,\x) + (-0.15,-0.15) rectangle +(0.15,0.15);
			\draw [thick] ([shift=(-90:0.1)]0,\x) arc (-90:90:0.1);
			\draw [thin, dotted] (\x,0) -- (\x, 5.4);
		}
		\filldraw (0, 1) + (0.1, -0.05) rectangle +(1,0.05);%\filldraw (1,1) circle (0.05);
		\filldraw (0, 2) + (0.1, -0.05) rectangle +(1.4,0.05);%\filldraw (2,2) circle (0.05);
		\filldraw (0, 3) + (0.1, -0.05) rectangle +(1.4,0.05);
		\filldraw (0, 4) + (0.1, -0.05) rectangle +(1.4,0.05);
		\filldraw (0, 5) + (0.1, -0.05) rectangle +(1.4,0.05);
		\draw (1.5,2) circle (0.11);
		\draw (1.5,3) circle (0.11);
		\draw (1.5,4) circle (0.11);
		\draw (1.5,5) circle (0.11);
		\draw (1.5, 2) + (0.1, -0.05) rectangle +(0.5,0.05);
		\draw (1.5, 3) + (0.1, -0.05) rectangle +(1.5,0.05);
		\draw (1.5, 4) + (0.1, -0.05) rectangle +(2.5,0.05);
		\draw (1.5, 5) + (0.1, -0.05) rectangle +(3.5,0.05);
	\end{tikzpicture}
	\qquad \qquad
	\begin{tikzpicture}[scale=0.9]
		\draw [thick,dashed] (2,0) -- (2,6) node [above] {$a=2$};
		% \draw [thick,dashed] (0.5,0) -- (0.5,5.5) node [above] {$r=0$};
		\draw [->] (-0.3,0) -- (5.5,0) node [right] {$\alpha$};
		\draw [->] (0,-0.3) -- (0,5.5) node [above] {$d$};
		\foreach \x in {1,...,5}{
			\draw (\x,0.1)--++(0,-0.2) node [below] {$\x$};
			\draw (0.1,\x)--++(-0.2,0) node [left] {$\x$};
			\draw (\x,\x) + (-0.15,-0.15) rectangle +(0.15,0.15);
			\draw [thick] ([shift=(-90:0.1)]0,\x) arc (-90:90:0.1);
			\draw [thin, dotted] (\x,0) -- (\x, 5.4);
		}
		\filldraw (0, 1) + (0.1, -0.05) rectangle +(1,0.05);%\filldraw (1,1) circle (0.05);
		\filldraw (0, 2) + (0.1, -0.05) rectangle +(1.9,0.05);%\filldraw (2,2) circle (0.05);
		\filldraw (0, 3) + (0.1, -0.05) rectangle +(1.9,0.05);
		\filldraw (0, 4) + (0.1, -0.05) rectangle +(1.9,0.05);
		\filldraw (0, 5) + (0.1, -0.05) rectangle +(1.9,0.05);
		\draw (2,2) circle (0.11);
		\draw (2,3) circle (0.11);
		\draw (2,4) circle (0.11);
		\draw (2,5) circle (0.11);
		\draw (2, 3) + (0.1, -0.05) rectangle +(1,0.05);
		\draw (2, 4) + (0.1, -0.05) rectangle +(2,0.05);
		\draw (2, 5) + (0.1, -0.05) rectangle +(3,0.05);
	\end{tikzpicture}
	\caption{Solvability for the stochastic heat equation with fractional Laplace, i.e, the case when
		$b=1$ and $r=0$. See Figures \ref{F:SHE} and \ref{F:SWE} for the legend.}
	\label{F:Stable}
\end{figure}
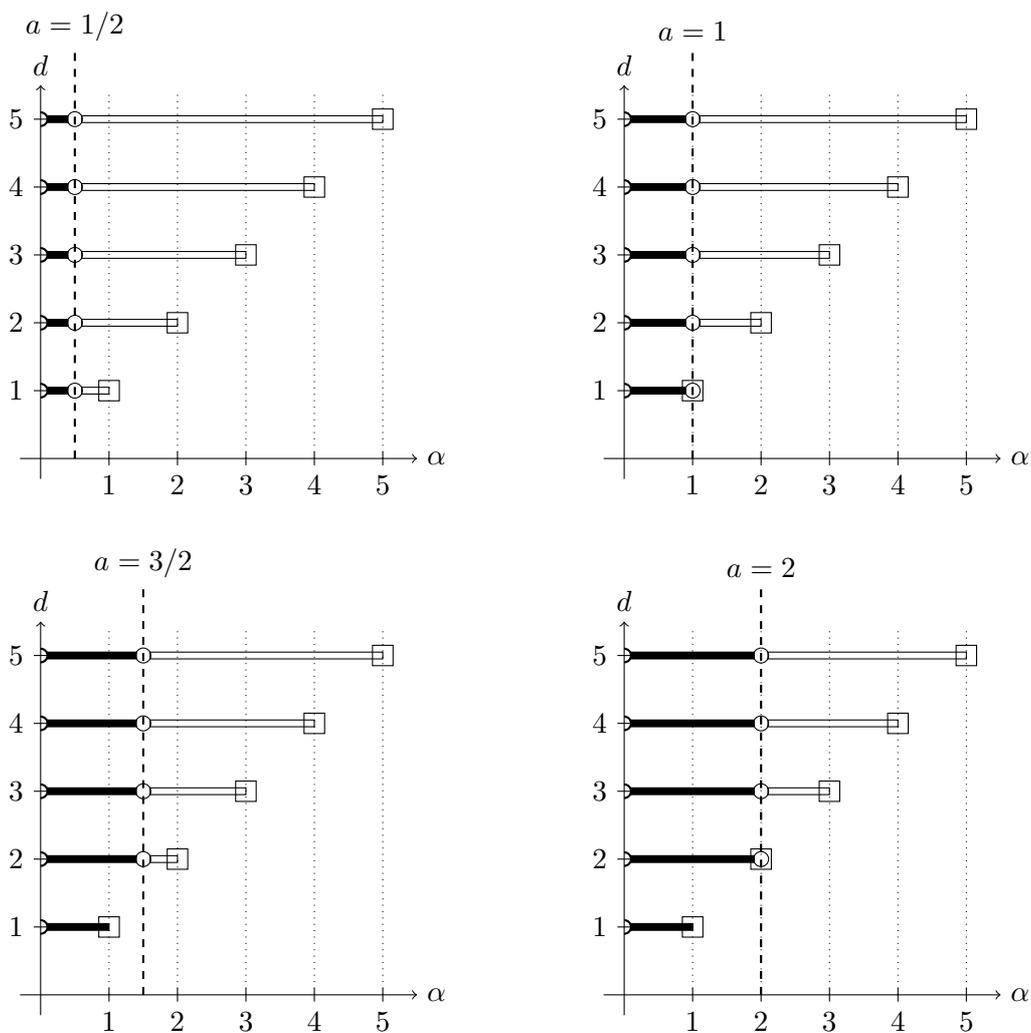% }}}

\subsection{Examples on asymptotics}
\label{SS:EgAsym}

In this part, we list several examples for the asymptotics when global solutions exist. In
particular, we will show that the asymptotics in \eqref{E:momentasym} interpolates the corresponding
results for both stochastic wave and heat equations.

\begin{example}[Asymptotics for SWE]% {{{
	\label{Eg:AsymSWE}
	Even though our results requires $b$ to be strictly less than $2$, but by
	formally setting
	\begin{align*}
		a = b = \nu = 2 \quad \text{and} \quad r = 0,
	\end{align*}
	we have that
	\begin{align*}
		\beta = \frac{4-\alpha}{3-\alpha} \quad \text{and} \quad
		t_p   = (p-1)^{1/(4-\alpha)} t,
	\end{align*}
	and results in \eqref{E:momentasym}, \eqref{E:momentasym-t}, and \eqref{E:momentasym-p} recover
	the corresponding results for the stochastic wave equation, namely,  (1.9), (1.10), and (1.11) of
	\cite{BCC21}, respectively.  Due to the importance of white noise and for the future references,
	here we list two special cases regarding white noise:\\

	\noindent(1) The SWE with white noise in $\R$: By further setting $d=\alpha=1$, we see that
	\begin{align}
		\label{E:AsymSWE-1d}
			% =p(p-1)^{1/2}\frac{2^{3/2}\sqrt{\theta} \mathcal{M}_2^{3/4}(\delta_0)}{ 2^{3/2}\nu^{1/4}}
			\lim_{t \to \infty} \frac{\log \E\left[|u(t,x)|^p\right]}{t^{3/2}} =  \frac{p(p-1)^{1/2}\sqrt{\theta}}{3 (2\nu)^{1/4}}
			\quad \text{and} \quad
			\lim_{p \to \infty} \frac{\log \E\left[|u(t,x)|^p\right]}{p^{3/2}} = \frac{t^{3/2}\sqrt{\theta}}{3 (2\nu)^{1/4}},
	\end{align}
	where we have applied \eqref{E:Mdelta1d}.\\

	\noindent(2) The SWE with white noise in $\R^2$: Similarly, by setting $d=\alpha=2$, we see that
	\begin{align}
		\label{E:AsymSWE-2d}
			% =p(p-1)^{1/2}\frac{2^{3/2}\sqrt{\theta} \mathcal{M}_2^{3/4}(\delta_0)}{ 2^{3/2}\nu^{1/4}}
			\lim_{t \to \infty} \frac{\log \E\left[|u(t,x)|^p\right]}{t^2} = \frac{p(p-1)\theta\mathcal{M}_{2,2}(\delta_0)}{2\nu}
			\quad \text{and} \quad
			\lim_{p \to \infty} \frac{\log \E\left[|u(t,x)|^p\right]}{p^2} = \frac{t^2\theta\mathcal{M}_{2,2}(\delta_0)}{2\nu}.
	\end{align}
\end{example}% }}}

\begin{example}[Asymptotics for SHE]% {{{
	\label{Eg:AsymSHE}
	As for the stochastic heat equation case, by setting
	\begin{align*}
		a = 2,\quad b= \nu = 1 \quad \text{and} \quad r = 0,
	\end{align*}
	we have that
	\begin{align*}
		\beta = \frac{4-\alpha}{2-\alpha} \quad \text{and} \quad
		t_p=(p-1)^{2/(4-\alpha)}t,
	\end{align*}
	and results in \eqref{E:momentasym} and \eqref{E:momentasym-t} recover the corresponding
	conjectured results for SHE, namely, (1.16) and (1.17) of Balan {\it et al} \cite{BCC21},
	respectively, which are equivalent to Theorem 1.1 and 1.2 of X. Chen \cite{Chen17AIHP} when
	setting $\alpha_0=0$ and using \cite[Lemma A.2]{CHSX15AIHP} to rewrite the constant
	$\mathcal{E}$ in \cite{Chen17AIHP} in terms of $\mathcal{M}_a$. Due to the importance of the white
	noise case, we list the corresponding asymptotics here. When $\alpha=d=1$,
	\begin{align}
		\label{E:AsymSHE1d}
		\lim_{t \to \infty} \frac{\log \E\left(|u(t,x)|^p\right)}{t^3} = \frac{p(p-1)^2\theta^2}{24\nu}
		\quad \text{and} \quad
		\lim_{p \to \infty} \frac{\log \E\left(|u(t,x)|^p\right)}{p^3} = \frac{t^3\theta^2}{24\nu},
	\end{align}
	where we have applied \eqref{E:Mdelta1d}. Note that some upper and lower bounds for the first
	limit in \eqref{E:AsymSHE1d} in case of $p=2$ were earlier obtained by Hu \cite[part 1) of
	Theorem 4.1]{Hu02Chaos}.
\end{example}% }}}

\begin{example}[Asymptotics for SHE with fractional Laplace]% {{{
	\label{Eg:AsymStable}
	In this example we restrict ourselves to the case when $b=1$, $a\in (0,2]$, $\alpha < d$, and $r=0$, which is the 1-dimensional SHE with fractional Laplace. With this set up we have
	\begin{align*}
	\beta = \frac{2a-\alpha}{a-\alpha} \qquad \text{and} \qquad t_p = (p-1)^{\frac{a}{2a-\alpha}} t,
	\end{align*}
	and by Corollary \ref{C:Lim},
	\begin{align}
		\label{E:AsymStable-t}
		\lim_{t \to \infty} \frac{\log \E\left(|u(t,x)|^p\right)}{t^{\frac{2a-\alpha}{a-\alpha} }} =
		p (p-1)^{\frac{a}{a-\alpha}} \left( \frac{1}{2} \right)\left( \frac{2a}{2a-\alpha} \right)^{\frac{2a-\alpha}{a-\alpha}} \left[ \theta \nu^{-\alpha/a} \mathcal{M}_{a,d}^{\frac{2a-\alpha}{a}} \right]^{ \frac{a}{a-\alpha} } \left(\frac{a-\alpha}{a}\right)
	\end{align}
	and
	\begin{align}
		\label{E:AsymStable-p}
		\lim_{p \to \infty} \frac{\log \E\left(|u(t,x)|^p\right)}{p^{\frac{2a-\alpha}{a-\alpha} }} =
		t^{\frac{2a-\alpha}{a-\alpha}}
		\left( \frac{1}{2} \right)\left( \frac{2a}{2a-\alpha} \right)^{\frac{2a-\alpha}{a-\alpha}} \left[ \theta \nu^{-\alpha/a} \mathcal{M}_{a,d}^{\frac{2a-\alpha}{a}} \right]^{ \frac{a}{a-\alpha} } \left(\frac{a-\alpha}{a}\right).
	\end{align}
	this setup has been studied in \cite{CHSS18} for the case of a time-dependent noise where the covariance function is given by
	\[
		\mathbb{E}[\dot{W}(r,x)\dot{W}(s,y)] = |r-s|^{-\alpha_0}\gamma(x-y)
	\]
	and $\gamma(x)$ is defined to be either of the following:
	\begin{align}
		\label{E:gammaHu}
		\gamma(x) := \begin{cases}
									|x|^{-\alpha}                  & \text{where } \alpha \in (0,d)\: \text{ or} \\
									\prod_{j=1}^d |x_j|^{\alpha_j} & \text{where } \alpha_j \in (0,1).
									\end{cases}
	\end{align}
	They proved that for $\alpha < d$ and let $p \ge 2$,
	\begin{equation}\label{E:AsymVar}
		\lim_{t\to\infty} t^{-\frac{2a - \alpha - \alpha\alpha_0}{a-\alpha}}\log\mathbb{E}[|u(t,x)|^p] = p(p-1)^{\frac{a}{a-\alpha}}\textbf{M}(a,\alpha_0,d,\gamma),
	\end{equation}
	where the variational constant is given by
	\begin{align*}
		\textbf{M}(a,\alpha_0,d,\gamma) = \sup_{g \in \mathcal{A}_a} \bigg\{ \frac{1}{2} \int_0^1\int_0^1 \int_{\R^{2d}} & \frac{\gamma(x)}{|r-s|^{\alpha_0}}g^2(s,x)g^2(r,y)  \ud x \ud y \ud r \ud s
		\\& - (2\pi)^{-d} \int_0^1 \int_{\R^d} |x|^a |\mathcal{F}g(s,\xi)|^2 \ud \xi \ud s \bigg\}
	\end{align*}
	with\footnote{Note that the Fourier transform is defined differently in \cite{CHSS18}.}
	\[
		\mathcal{A}_d := \left\{ g(s,x) : \int_{\R^d} g^2(s,x) \ud x = 1, \forall s \in [0,1] \text{ and } (2\pi)^{-d} \int_0^1 \int_{\R^d} |x|^a |\mathcal{F}g(s,\xi)|^2 \ud \xi \ud s < \infty \right\}.
	\]
	By setting $\alpha_0=0$ and letting $g(s,x)=g(x) \in \mathcal{F}_a$ be independent in $s$, which is the time-independent setup, then Equation \eqref{E:AsymStable-t} and Lemma \ref{L:EvsM} together recover \eqref{E:AsymVar}. Indeed, by observing \eqref{E:VarConstE} and \eqref{E:EvsM}, we see that
	\begin{align}
		\textbf{M}(a,\alpha_0,d,\gamma) & = \mathbf{E}_{a,d}\left(\frac{1}{2}\gamma,2\right) \nonumber \\
                                    & = 2^{\frac{-a}{a-\alpha}}\left(\frac{a-\alpha}{a}\right)\left( \frac{2a}{2a-\alpha} \right)^{\frac{2a-\alpha}{a-\alpha}}\mathcal{M}_{a,d}\left(\gamma,1\right)^{\frac{2a-\alpha}{a-\alpha}} \label{E:MbfAndMmathcalRelation}
	\end{align}
	and by rewriting \eqref{E:AsymVar} with \eqref{E:MbfAndMmathcalRelation} yields \eqref{E:AsymStable-t}.
	Finally, we note that condition \eqref{E:gammaHu} can be relaxed to allow white noise in one
	dimensional case, namely, $\alpha=d=1$. In this case, one can simply replace $\alpha$ and $d$ in both
	\eqref{E:AsymStable-t} and \eqref{E:AsymFracSHE-p} by $1$ and in addition replace $\mathcal{M}_{a,d}$ by
	$\mathcal{M}_{a ,1}(\delta_0)$.
\end{example}% }}}

\begin{example}[Asymptotics for SPDEs with $r=\Ceil{b}-b$ and white noise]% {{{
	\label{Eg:AsymFracSHE}
	In this example, we consider the case when $a=2$, $d=\alpha=1$ (white noise), and $r=\Ceil{b}-b$.
	As seen in Example \ref{Eg:FracSHE}, there exits a global solution. In this case,
	\begin{align*}
		\beta = \frac{4\Ceil{b}-b}{4\Ceil{b}-b-2} \qquad \text{and} \qquad
		t_p   = (p-1)^{\frac{2}{4\Ceil{b}-b}}\: t,
	\end{align*}
	and by \eqref{E:Mdelta1d} and Corollary \ref{C:Lim},
	\begin{align}
		\label{E:AsymFracSHE-t}
		\begin{aligned}
			\lim_{t \to \infty} \frac{\log \E\left(|u(t,x)|^p\right)}{t^{\frac{4\Ceil{b}-b}{4\Ceil{b}-b-2} }} =
			& p (p-1)^{\frac{2}{4\Ceil{b}-b-2}}\\
			& \times \left(\frac{9\theta^2}{8\nu}\right)^{\frac{1}{4 \Ceil{b}-b-2}} (4 \Ceil{b}-b-2) \left(4 \Ceil{b}-b\right)^{-\frac{4 \Ceil{b}-b}{4 \Ceil{b}-b-2}},
		\end{aligned}
	\end{align}
	and
	\begin{align}
		\label{E:AsymFracSHE-p}
		\lim_{p \to \infty} \frac{\log \E\left(|u(t,x)|^p\right)}{p^{\frac{4\Ceil{b}-b}{4\Ceil{b}-b-2} }} =
		t^{\frac{4\Ceil{b}-b}{4\Ceil{b}-b-2}}
			\left(\frac{9\theta^2}{8\nu}\right)^{\frac{1}{4 \Ceil{b}-b-2}} (4 \Ceil{b}-b-2) \left(4 \Ceil{b}-b\right)^{-\frac{4 \Ceil{b}-b}{4 \Ceil{b}-b-2}}.
	\end{align}
\end{example}% }}}

\begin{example}[Asymptotics for SPDEs with $r=0$ and white noise]% {{{
	\label{Eg:AymsFracSHE0}
	From Example \ref{Eg:AymsFracSHE0}, we see that when $a=2$, $r=0$, $d=\alpha=1$ (white noise), the
	global solution exists when  $b\in (2 /3, 2)$. In this case, we have that
	\begin{align*}
		\beta = \frac{3b}{3b-2} \qquad \text{and} \qquad
		t_p   = (p-1)^{2 / (3b)}\: t,
	\end{align*}
	and by \eqref{E:Mdelta1d} and Corollary \ref{C:Lim},
	\begin{gather}
		\label{E:AsymFracSHE-t}
			\lim_{t \to \infty} \frac{\log \E\left(|u(t,x)|^p\right)}{t^{3b/(3b-2)}} =
			p (p-1)^{\frac{2}{3b}} \left(b-\frac{2}{3}\right) b^{-\frac{3 b}{3 b-2}} \left(\frac{\theta^2}{8\nu}\right)^{\frac{1}{3 b-2}}
			\quad \text{and}\\
		\label{E:AsymFracSHE-p}
		\lim_{p \to \infty} \frac{\log \E\left(|u(t,x)|^p\right)}{p^{3b/(3b-2)}} =
		t^{\frac{3b}{3b-2}}
		\left(b-\frac{2}{3}\right) b^{-\frac{3 b}{3 b-2}} \left(\frac{\theta^2}{8\nu}\right)^{\frac{1}{3 b-2}}.
	\end{gather}
\end{example}% }}}

\section{Some preliminaries}
\label{S:Pre} \subsection{Skorohod integral and mild solution}
\label{SS:Malliavin}

We start with a nonnegative and nonnegative definite tempered measure $\Gamma$ with density $\gamma$
in the sense that $\Gamma(\ud x) = \gamma(x) \ud x$ and
\begin{align*}
	\int_{\R^d} \Gamma(\ud x) (\phi * \widetilde{\phi})(x) \ge 0 \quad \text{for all $\phi \in \mathscr{S}(\R^d)$}
\end{align*}
where $\widetilde{\phi}(x): = \phi(-x)$.
% In addition we require the existence of an $r>0$ such that \begin{align*} \int_{\R^d}\Gamma(\ud x)
% \frac{1}{ (1 + |x|^2 )^r} < \infty.  \end{align*}
According to the Bochner theorem, there exists a nonnegative and nonnegative definite measure $\mu$,
often referred as the {\it spectral measure} on $\R^d$ whose Fourier transform (in the weak sense)
is $\Gamma$, namely, that for any $\phi \in \mathcal{D}(\R^d)$ (the space of test functions),
\begin{align*}
	\int_{\R^d}\Gamma(\ud x) \phi(x) = \int_{\R^d} \mu(\ud \xi) \mathcal{F}\phi (\xi).
\end{align*}
Since $\mu$ is nonnegative definite, the following functional
\begin{align}
	\label{E:Cov}
	C(\phi,\psi) = \int_{\R} \mathcal{F}\phi(\xi) \overline{ \mathcal{F}\psi(\xi) } \mu(\ud \xi),
	\qquad \text{for all $\phi,\psi\in\mathcal{D}(\R^d)$}
\end{align}
is nonnegative-definite and thus one can associate it with a zero-mean Gaussian processes, $W:= \left\{
W(\phi) : \phi \in \mathcal{D}(\R^d) \right\}$, with the covariance functional of $W$
given by \eqref{E:Cov}. In other words,
\begin{align*}
	\mathbb{E}\left(W(\phi)W(\psi)\right) = \int_{\R} \mathcal{F}\phi(\xi) \overline{ \mathcal{F}\psi(\xi) } \mu(\ud \xi) =: \langle \phi, \psi \rangle_{\mathcal{H}}.
\end{align*}
Let $\mathcal{H}$ be the completion of $\mathcal{D}(\R^d)$ with respect to $\langle \cdot, \cdot
\rangle_{\mathcal{H}}$ and thus we see $\phi \mapsto W(\phi)$ is an isometry from
$\mathcal{D}(\R^d)$ to $L^2(\Omega)$, that is, $\E\left(W(\phi)^2\right)	=
\Norm{\phi}_{\mathcal{H}}^2$ for $\phi\in \mathcal{D}(\R^d)$. One can extend this isometry from
$\mathcal{D}(\R^d)$ to $\mathcal{H}$. We refer the interested readers to \cite{Minicourse09} and
references therein.  \bigskip

We denote $\delta$ the {\it Skorohod integral} with respect to $W$ and denote its domain by
$\text{Dom}(\delta)$. $u$ is called {\it Skorohod integrable} if $u\in \text{Dom}(\delta)$, in which
case we write $\delta(u) = \int_{\R^d} u(x) W(\delta x)$ and by isometry,
$\E\left(\delta(u)^2\right)=\E(\Norm{u}_{\mathcal{H}}^2)$.

\begin{definition}[Mild, local and global solutions]
	\label{D:Sol}
	(1) For $T\in (0,\infty]$, a random field $u=\left\{u(t,x) : t \in (0,T),\: x\in\R^d\right\}$ is
	called {\it a mild solution} to the equation \eqref{E:SPDE} if for all $x\in\R^d$ and $s,t$ fixed
	with $0<s\le t <T$, $y\to G(t-s,x-y)u(s,y)$ is Skorohod integrable and the following stochastic
	integral equation holds almost surely
	\begin{equation}
		\label{E:IntEq}
		u(t,x) = 1 + \sqrt{\theta}\int_0^t \left( \int_{\R^d} G(t-s,x-y)u(s,y) W(\delta y) \right)\ud s.
	\end{equation}

	\noindent (2) Let $u(t,x)$ be a mild solution to \eqref{E:SPDE} (or \eqref{E:IntEq}) and fix $p\ge
	1$. We call $u(t,x)$ a {\it global $L^p(\Omega)$-solution}, or simply an {\it
	$L^p(\Omega)$-solution} if
	\begin{align}
		\label{E:Lp}
		\Norm{u(t,x)}_p < \infty \qquad \text{for all $t>0$ and $x \in \R^d$.}
	\end{align}

	\noindent (3) If there exist $0<T_1\le T_2<\infty$ such that $\Norm{u(t,x)}_p$ is finite for all
	$t\in(0,T_1)$ and $x\in\R^d$, but $\Norm{u(t,x)}_p$ does not exist whenever $t>T_2$, the mild
	solution $u(t,x)$ in this case is called a {\it local $L^p(\Omega)$-solution}.
\end{definition}

Note that through construction of the Skorohod integral $\delta$, a mild solution is necessarily to
be an $L^2(\Omega)$-solution.  For more details, one may check, e.g., Nualart \cite[Chapter
3]{Nualart18Book}.  \bigskip

% \begin{definition}
% 	\label{D:CriDim}
% 	The {\it critical dimension} for \eqref{E:SPDE}, denoted as $d_c$, refers to the dimension at
% 	which one can find a viable condition on $\alpha$ --- the so-called {\it critical condition}, such
% 	that under that this critical condition on $\alpha$, there exits a local, but not global,
% 	$L^2(\Omega)$-solution to \eqref{E:SPDE} (or \eqref{E:IntEq}).
% \end{definition}

Through the standard {\it Picard iteration} scheme, the solution can be expressed by the following
{\it Weiner chaos expansion}:
\begin{equation}
	\label{E:solwexp}
	u(t,x) = 1 + \sum_{k=1}^\infty \theta^{k/2}I_k(f_k(\cdot,x,t)),
\end{equation}
where $I_k : \mathcal{H}^{\otimes k}\left((\R^d)^k\right) \to \mathcal{H}_k$ is the $k^{th}$ order
Skorohod integral and $\mathcal{H}_k$ is the $k^{th}$ Wiener chaos space and the kernels
$f_k(\cdot,x,t)$, obtained through the iteration, are equal to
\begin{align*}
	% \label{E:wker}
	f_n(x_1,...,x_k;x,t) & = \int_0^t \int_0^{t_n}\cdots \int_0^{t_2} G(t-t_n,x-x_n)\cdots G(t_2-t_1,x_2-x_1) \ud t_1\cdots \ud t_n \\
                       & = \int_0^t \int_0^{t_n}\cdots \int_0^{t_2} G(t_1,x-x_n)\cdots G(t_n-t_{n-1},x_2-x_1) \ud t_1\cdots \ud t_n. \nonumber
\end{align*}
For ease of notation, throughout this article, we may write the above integrals as
\begin{align*}
	f_n(x_1,...,x_n;x,t) & = \int_{[0,t]^n_<} G(t-t_n,x-x_n)\cdots G(t_2-t_1,x_2-x_1) \ud \vec{t} \\
                       & = \int_{[0,t]^n_<}G(t_1,x-x_n)\cdots G(t_n-t_{n-1},x_2-x_1) \ud \vec{t},
\end{align*}
where $[0,t]_<^n:=\{(t_1,\cdots,t_n)\in [0,t]^n:\: t_1<\cdots < t_n\}$. As usual, we use
$\widetilde{f_n}(\cdot, x, t)$ to denote the symmetrization of $f_n(\cdot,x,t)$:
	\begin{align*}
		\widetilde{f_n}(\cdot; x, t) &= \frac{1}{n!} \sum_{\rho \in \Sigma_n} f_n\left(x_{\rho(1)},\cdots, x_{\rho(n)}\right) \\
																 &= \frac{1}{n!} \sum_{\rho \in \Sigma_n} \int_{[0,t]^n_<} G\left(t-t_n,x-x_{\rho(n)}\right)\cdots G\left(t_2-t_1,x_{\rho(2)}-x_{\rho(1)}\right) \ud \vec{t},
	\end{align*}
where $\Sigma_n$ is the set of all permutations of $\{1,\cdots,n\} $.
% In particular, for $x=0$ we have
% \begin{align}
% 	\label{E:fperm0}
% 	\widetilde{f_n}(\cdot; 0, t) = \frac{1}{n!} \sum_{\rho \in \Sigma_n} \int_{[0,t]^n_<} G(t_1,x_{\rho(1)})\cdots G(t_n-t_{n-1},x_{\rho(n)}-x_{\rho(n-1)}) \ud \vec{t}.
% \end{align}
By setting $t_{n+1}=t$, the Fourier transform of the kernels, $f_n$, is given by
\begin{align}
	\label{E:fweker}
	\mathcal{F}f_n(\cdot;x,t)(\xi_1,...,\xi_n) = e^{-i \left(\sum_{j=1}^n \xi_j \right) \cdot x} \int_{[0,t]^n_<} \prod_{k=1}^n \overline{\mathcal{F}G(t_{k+1}-t_k,\cdot)\left(\sum_{j=1}^k \xi_j\right) } 	\ud \vec{t}.
\end{align}

Recall the notation above that $G(t,x) = Y_{a,b,r,d}(t,x)$. The following scaling properties for both $\mathcal{F}G(t,\cdot)(\xi)$ and
$\Norm{\widetilde{f_n}(\cdot,t,x)}_{\mathcal{H}^{\otimes n}}$ play an important role in the paper.

\begin{lemma}
	\label{L:Scale}
	For any $c,t>0$, $n\ge 1$,  $\xi,\xi_1,\cdots,\xi_n\in\R^d$, the following scaling properties hold:
	\begin{gather}
		\label{E:greenscale}
		\mathcal{F}G(t,\cdot)(c\xi) = c^{-\frac{a}{b}(b+r -1)}\mathcal{F}G\left(c^{\frac{a}{b}}t,\cdot\right)( \xi)  \quad \text{and} \quad
		\mathcal{F}G(ct,\cdot)(\xi) = c^{b+r -1}\mathcal{F}G(t,\cdot)(c^{b/a} \xi),\\
		\label{E:Ffnscale}
		\mathcal{F}\widetilde{f}_n(\cdot,0,ct)(\xi_1,\cdots,\xi_n) = c^{n(b+r)} \mathcal{F}\widetilde{f}_n(\cdot,0,t)(c^{b/a}\xi_1, \cdots, c^{b/a}\xi_n), \\
		\label{E:normscale}
		\Norm{\widetilde{f_n}(\cdot,0,t)}_{\mathcal{H}^{\otimes n}}^2 = t^{[2(b + r ) -b \alpha / a]n} \Norm{\widetilde{f_n}(\cdot,0,1)}_{\mathcal{H}^{\otimes n}}^2.
	\end{gather}
\end{lemma}
\begin{proof}
	The scaling properties in \eqref{E:greenscale} are direct consequences of the explicit expression
	of $\mathcal{F}G(t,\cdot)(\xi)$ as in \cite[(4.8)]{CHN19}. Property \eqref{E:Ffnscale} is an easy
	exercise of change of variables on \eqref{E:fweker}. Property \eqref{E:normscale} is a direct consequence of
	\eqref{E:greenscale}, \eqref{E:Ffnscale}, and the scaling property of the spectral measure $\mu$.
	We leave the details for the interested readers.
\end{proof}

Finally, let us recall the following standard result about the existence and uniqueness of the
solution to \eqref{E:SPDE} (or \eqref{E:IntEq}) when it can be written as the Weiner chaos expansion
\eqref{E:solwexp}.

\begin{theorem}\label{T:2momseries}
	Fix any $T \in (0,\infty]$. Suppose that $f_n(\cdot,x,t) \in \mathcal{H}^{\otimes n}$ for any $t
	\in (0,T)$, $x \in R^d$ and $n \ge 1$. Then \eqref{E:SPDE} (or \eqref{E:IntEq}) has a unique
	$L^2(\Omega)$-solution on $(0,T) \times \R^d$ if and only if the series \eqref{E:solwexp}
	converges in $L^2(\Omega)$ for any $(t,x) \in (0,T) \times \R^d$, which is equivalent to the
	convergence of the series \eqref{E:2mom}. In this case, the solution is given by \eqref{E:solwexp}
	with the second moment given by
	\begin{equation}
		\label{E:2mom}
		\mathbb{E}\left(u(t,x)^2\right) = \sum_{n \ge 0} \theta^n n!  \Norm{\widetilde{f_n}(\cdot,x,t)}_{\mathcal{H}^{\otimes n}}^2 \quad \text{for all $(t,x) \in (0,T) \times \R^d$}.
	\end{equation}
\end{theorem}

\subsection{Some asymptotics and variational constants}
\label{SS:Variation}

Recall that the correlation function $\gamma$ satisfies Assumption \ref{A:Noise} and that the
corresponding spectral measure is $\mu$; see Remark \ref{R:Noise}. Define
\begin{gather}
	\label{E:rho}
	\rho_{\nu,a}(\gamma) = \sup_{\Norm{f}_{L^2(\R^d)} =1} \int_{\R^d} \left[ \int_{\R^d} \frac{f(x+y)f(y)}{\sqrt{1+\frac{\nu}{2}|x+y|^a } \sqrt{1+\frac{\nu}{2}|y|^a} } \ud y \right]^2 \mu(\ud x)
\end{gather}
and
\begin{gather}
	\label{E:Lambda}
	\begin{aligned}
		\mathcal{M}_{a,d}(\gamma,\theta) & := \sup_{g \in \mathcal{F}_a} \left\{ \left( \iint_{\R^{2d}} g^2(x)g^2(y) \gamma(x+y)\ud x\ud y \right)^{1/2} - \frac{\theta}{2}\: \mathcal{E}_a(g,g) \right\} \\
                                     & = \sup_{g \in \mathcal{F}_a} \left\{ \left\langle g^2 * g^2, \gamma \right\rangle_{L^2(\R^d)}^{1/2} - \frac{\theta}{2}\: \mathcal{E}_a(g,g) \right\},
	\end{aligned}
\end{gather}
where
\begin{gather}
	\label{E:Egg}
	\mathcal{E}_a(g,g) := (2\pi)^{-d} \int_{\R^d} |\xi|^a |\mathcal{F}g(\xi)|^2 \ud \xi \quad \text{and} \\
	\mathcal{F}_a      := \left\{f\in L^2(\R^d): \: \Norm{f}_{L^2(\R^d)}=1,\: \mathcal{E}_a(f,f)<\infty \right\}.
	\label{E:Fa}
\end{gather}
We often omit the dimension $d$ in  $ \mathcal{M}_{a,d}$ when it is clear from context.  We use the
convention that $\mathcal{M}_a(f):=\mathcal{M}_a(f,1)$ to be consistent with notation \eqref{E:M1}.
By a similar argument as the proof of \cite[Lemma 2.3]{BCC21}, one can show that
\begin{equation}
	\label{E:ScaleM}
	\mathcal{M}_a(\Theta \gamma , \theta) =
	\Theta^{\frac{a}{2a-\alpha}}\theta^{-\frac{\alpha}{2a-\alpha}} \mathcal{M}_a( \gamma , 1),
	\qquad \text{for all $\theta$ and $\Theta>0 $.}
\end{equation}

For the Riesz kernel case (see Example \ref{Eg:NoiseEg}), Bass, Chen and Rosen \cite{BCR09}
established that when $a \in (0,2]$ and $\nu = 2$,
\begin{equation}
	\label{E:oldlim}
	\lim_{n \to \infty} \frac{1}{n} \log \left[ \frac{1}{(n!)^2} \int_{(\R^d)^n} \left( \sum_{\sigma \in \Sigma_n} \prod_{k=1}^n \frac{1}{1+ |\sum_{j=k}^n \xi_{\sigma(j)}|^a}\right)^2 \mu (\ud \vec{\xi})\right]
	= \log\left( \rho_{2,a}\left(|\cdot|^{-\alpha}\right) \right),
\end{equation}
and \footnote{In Theorem 1.5 or eq. (1.20) of Bass {\it et al} \cite{BCR09}, the factor
$(2\pi)^{-d}$ should not be present.}
\begin{equation}
	\label{E:oldrhovarrep}
	\rho_{2,a} \left(|\cdot|^{-\alpha}\right) = \mathcal{M}_a^{2-(\alpha/a)}\left(|\cdot|^{-\alpha},2\right).
\end{equation}
We first apply some scaling arguments to accommodate the parameter $\nu$ in both \eqref{E:oldlim}
and \eqref{E:oldrhovarrep},
the proof of which can be found in Appendix:

\begin{lemma}[The Riesz kernel case]
	\label{L:rhovar}
	If $\gamma(x) = |x|^{-\alpha}$ for some $\alpha \in (0,d)$, then for any $\nu >0 $ and $a \in
	(0,2]$,
	\begin{equation}
	\label{E:rhoLam}
		\rho_{\nu, a}\left(|\cdot|^{-\alpha}\right) = \left(\frac{\nu}{2}\right)^{-\alpha/a} \mathcal{M}_a^{2-(\alpha/a)}\left(|\cdot|^{-\alpha},2\right)
		= \nu^{-\alpha / a}\mathcal{M}_a^{2-(\alpha / a)}(|\cdot|^{-\alpha}),
	\end{equation}
	and
	\begin{equation}
		\label{E:LogRhoRiez}
		\lim_{n \to \infty} \frac{1}{n} \log \left[ \frac{1}{(n!)^2} \int_{(\R^d)^n} \left( \sum_{\sigma \in \Sigma_n} \prod_{k=1}^n \frac{1}{1+ \frac{\nu}{2}|\sum_{j=k}^n \xi_{\sigma(j)}|^a}\right)^2 \mu (\ud \vec{\xi})\right]
		= \log\left( \rho_{\nu, a}\left(|\cdot|^{-\alpha}\right) \right).
	\end{equation}
\end{lemma}

More generally we have the following theorem:

\begin{theorem}
	\label{T:Variation}
	For the correlation function $\gamma$	that satisfies Assumption \ref{A:Noise}, both
	\eqref{E:rhoLam} and \eqref{E:LogRhoRiez} hold with $|\cdot|^{-\alpha}$ and $\mu$ replaced by $\gamma$
	and  $\mu$ as in \eqref{E:specfunction}, respectively. More precisely, it holds that
	\begin{equation}
		\label{E:rhoMrep}
		\rho_{\nu,a}(\gamma) = \nu^{-\alpha/a} \mathcal{M}_a^{2-(\alpha / a)}(\gamma)<\infty,
	\end{equation}
	and
	\begin{equation}
		\label{E:LogRho}
		\lim_{n \to \infty} \frac{1}{n} \log \left[ \frac{1}{(n!)^2} \int_{(\R^d)^n} \left( \sum_{\sigma \in \Sigma_n} \prod_{k=1}^n \frac{1}{1+ \frac{\nu}{2}|\sum_{j=k}^n \xi_{\sigma(j)}|^a}\right)^2 \mu (\ud \vec{\xi})\right]
		= \log\left( \rho_{\nu, a}\left(\gamma\right) \right).
	\end{equation}
\end{theorem}

\begin{remark}
	\label{R:Constant}
	It is often very difficult to obtain the exact value for the variational constant
	$\mathcal{M}_{a,d}(\gamma)$. To the best of our knowledge, only in case of $a=2$ and $\alpha=d=1$
	(white noise), one can compute explicitly that
	\begin{align}
		\label{E:Mdelta1d}
		\mathcal{M}_{2,1}(\delta_0)= (3/4)(1/6)^{1/3},
	\end{align}
	which is a consequence of Chen and Li \cite[Lemma 7.2]{ChenLi04} with $p=2$. When $d\ge 2$, the
	value of $\mathcal{M}_{2,d}(\delta_0)$ can be expressed using the best constant for the classical
	{\it Gagliardo-Nirenberg} inequality; see Remark 3.13 of Chen {\it et al} \cite{CDOT20} for more
	details.
\end{remark}

% Next we show that $\rho$ is finite. We do this by showing $\rho_{2,a}$ is finite and then noting
% that this implies that $\rho$ is finite by Lemma \ref{L:rhovar}. We stated in the introduction
% that Bass, Chen and Rosen proved in \cite{BCR09} that $\rho$ is finite for the case of the Riesz
% kernel. We mimic his proof but were he uses the Hardy-Littlewood-Sobolev theorem, we make use of the
% weak Young's inequality instead.
\begin{proof}[Sketch of the proof of Theorem \ref{T:Variation}]
	The proof of this theorem follows essentially the identical proof as Bass {\it et al}
	\cite{BCR09}, which is exclusively for the Riesz kernel. One simplification is that we only need
	to handle the case $p=2$ thanks to the hypercontractivity property. For our slight extension to
	the noise given in Assumption \ref{A:Noise}, there is no need to repeat their paper. Instead we
	will only point out the differences and necessary changes. For your convenience, the
	correspondence of parameters between Bass {\it et al} \cite{BCR09} and the current paper is listed
	in the following Table \ref{Tb:Notation}.

	\begin{table}[htpb]
		\centering
		\caption{Notation correspondence.}
		\label{Tb:Notation}
		\renewcommand{\arraystretch}{1.2}
		\begin{tabular}{c}
	\renewcommand{\arraystretch}{1.2}
	\begin{tabular}{|c|c|c|c|c|c|c|c|}
		\hline
                                  & \multicolumn{2}{c|}{Laplace} & \multicolumn{3}{c|}{Noise} & Moment   & Variational Const.  \\ \hline
		Bass {\it et al} \cite{BCR09} & $\beta$                      & $2$                        & $\sigma$ & $|\cdot|^{-\sigma}$  & $\varphi_{d-\sigma}$ & $p$ & $\Lambda_\sigma$                                \\
		Current paper                 & $a$                          & $\nu$                      & $\alpha$ & $\gamma(\cdot)$      & $\varphi$            & $2$ & $\mathcal{M}_a\left(|\cdot|^{-\alpha},2\right)$ \\ \hline
	\end{tabular}
		\end{tabular}
	\end{table}

	Theorem  \ref{T:Variation} is proven by showing the following claims: for $\gamma$ given in
	Assumption  \ref{A:Noise},
	\begin{enumerate}[(i)]
		\item $\rho_{\nu,a}(\gamma)<\infty$;
		\item $\liminf_{n \to \infty} \frac{1}{n} \log \left[ \frac{1}{(n!)^2} \int_{(\R^d)^n} \left( \sum_{\sigma \in \Sigma_n} \prod_{k=1}^n \frac{1}{1+ \frac{\nu}{2}|\sum_{j=k}^n \xi_{\sigma(j)}|^a}\right)^2 \mu (\ud \vec{\xi})\right] \ge \log\left( \rho_{\nu, a}\left(\gamma\right) \right)$;
		\item $\limsup_{n \to \infty} \frac{1}{n} \log \left[ \frac{1}{(n!)^2} \int_{(\R^d)^n} \left( \sum_{\sigma \in \Sigma_n} \prod_{k=1}^n \frac{1}{1+ \frac{\nu}{2}|\sum_{j=k}^n \xi_{\sigma(j)}|^a}\right)^2 \mu (\ud \vec{\xi})\right] \le \log\left( \rho_{\nu, a}\left(\gamma\right) \right)$;
		\item $\mathcal{M}_{a}(\gamma)<\infty$;
		\item $\rho_{\nu,a}(\gamma) = \nu^{-\alpha/a} \mathcal{M}_a^{2-(\alpha / a)}(\gamma)$.
	\end{enumerate}

	Part (i) which corresponds to Lemma 1.6 ({\it ibid.}) is established by Lemma \ref{L:rhofinite}
	below.

	Following exactly the same arguments as those in Section 3 ({\it ibid.}) with $\varphi_{d-\sigma}$
	({\it ibid.}) replaced by our $\varphi$ as in \eqref{E:specfunction}, one can prove part (ii) for
	$\nu=2$. Then an application of the scaling property as the proof of Lemma \ref{L:rhovar} shows
	the general case $\nu >0$.

	The proof of the upper bound, namely part (iii), is more challenging. This part corresponds to
	Sections 5 and 6 ({\it ibid.}). By examining these two sections carefully, we need to make some
	changes in Section 5 ({\it ibid.}), where as the arguments in Section 6 ({\it ibid.}) follow unchanged.
	For Section 5 ({\it ibid.}), we need to use the following decomposition of $\varphi$ as opposed to (5.4)
	({\it ibid.}):
	\begin{align*}
		\varphi(\xi) = \prod_{i=1}^k C_{\alpha_i,d_i} |\xi_{(i)}|^{-(d_i-\alpha_i)}
		= \prod_{i=1}^k C_{\alpha_i,d_i} (\mathcal{P}_i*\mathcal{P}_i) \left(\xi_{(i)}\right),
	\end{align*}
	with
	\begin{align*}
		\mathcal{P}_i\left(\xi_{(i)}\right)=\beta_{d_i-\alpha_i,d_i}\left|\xi_{(i)}\right|^{-(d_i-\alpha_i / 2)};
	\end{align*}
	see \eqref{E:noiseconstant} for the constants. Or equivalently,
	\begin{align*}
		\varphi(\xi)      = \left(\mathcal{P} * \mathcal{P}\right)(\xi) \quad \text{with} \quad
		\mathcal{P}(\xi) := \prod_{i=1 }^{k} \sqrt{C_{\alpha_i,d_i}}\: \beta_{d_i-\alpha_i,d_i} \left|\xi_{(i)}\right|^{-(d_i-\alpha_i / 2)}.
	\end{align*}
	Now (5.5) ({\it ibid.}) should be written as
	\begin{align*}
		\mathcal{P}_{\beta,\epsilon}(\xi) = \hat{h}(\epsilon \xi) \prod_{i=1 }^{k} \frac{\sqrt{C_{\alpha_i,d_i}}\: \beta_{d_i-\alpha_i,d_i} }{\beta+\left|\xi_{(i)}\right|^{-(d_i-\alpha_i / 2)}},
		\quad \text{for all $\beta,\epsilon \ge 0$,}
	\end{align*}
	where $h(\cdot)$ is defined in (5.2) ({\it ibid.}).  So $\mathcal{P}(\xi) =
	\mathcal{P}_{0,0}(\xi)$ and
	\begin{align*}
		\left(\mathcal{P}_{\beta,\epsilon}*\mathcal{P}_{\beta,\epsilon}\right)(\xi) \le
		\left(\mathcal{P}_{\beta,0}*\mathcal{P}_{\beta,0}\right)(\xi) \le
		\left(\mathcal{P}_{0,0}*\mathcal{P}_{0,0}\right)(\xi) = \varphi(\xi);
	\end{align*}
	see (5.6) ({\it ibid.}). With these changes, one can update accordingly the proof of Lemma 5.1
	({\it ibid.}) without any difficulty. Then the rest of Section 5 ({\it ibid.}) follows
	unchanged. In this way, we establish part (iii) for $\nu=2$. Finally, a scaling argument as in
	part (ii) proves part (iii) for all $\nu>0$.

	Parts (iv) and (v) correspond to Section 7 ({\it ibid.}). In particular, part (iv) corresponds to
	Lemma 7.1 ({\it ibid.}). Note that we only need to study the case $p=2$. By
	\eqref{E:fracrhofinite} with $\varphi(x-y)$ replaced by  $\gamma(x-y)$, we see that
	\begin{align}
		\label{E:(7.1)}
		\begin{aligned}
			\int_{(\R^d)^2} g^2(x)g^2(-y) \gamma(x-y) \ud x \ud y & \le  C \Norm{\tilde{g}^2}_{L^{2d/(2d-\alpha)}(\R^d)} \Norm{g^2}_{L^{2d/(2d-\alpha)}(\R^d)} \\
																														& =    C \Norm{g}_{L^{4d/(2d-\alpha)}(\R^d)}^4,
		\end{aligned}
	\end{align}
	where $\tilde{g}^2(x) = g(-x)$. Thus equation (7.1)  ({\it ibid.}) can be applied in our setting . The rest of the proof of Lemma 7.1 ({\it
	ibid.}) remains unchanged.

	It remains to update the proof of Theorem 1.5 in Section 7 ({\it ibid.}). For this, one needs only
	to update the four appearances of $1/(|\cdot|^\sigma)$ in (7.15), (7.22) and (7.23) ({\it ibid.}) to
	$\gamma(\cdot)$. Note that the factor $(2\pi)^{-d(p+1)}$ in the first equation  of (7.22) ({\it
	ibid.}) should be $(2\pi)^{-dp}$. With this, we complete the sketch proof of Theorem
	\ref{T:Variation}.
\end{proof}

Note that the proof of \cite[Lemma 1.6]{BCR09} relies on inequality (1.27) on p.  630 ({\it ibid.}),
which was a consequence of {\it Sobolev's inequality}. For the more general noises studied in this
paper, we can no longer apply this inequality. Instead, we prove the following lemma using the {\it
weak Young's inequality} (see, e.g., \cite[p.107]{LiebLoss01}) as a generalization of Lemma 1.6
({\it ibid.}).  Even though we only need the case $p=2$, the following lemma is proven for all $p\ge
2$.

\begin{lemma}
	\label{L:rhofinite}
	For any $f,g,h$ with $h\ge0$, for $\varphi$ given as in \eqref{E:specfunction} (see also Assumption
	\ref{A:Noise}), and for all $p\ge 2$, it holds that
	\begin{align*}
		\left( \int_{\R^d} \left[ \int_{\R^d} \dfrac{|f(x+y)g(y)|}{\sqrt{h(x+y)}\sqrt{h(y)}} \ud y \right]^p \varphi(x) \ud x \right)^{1/p}
		\le C \Norm{f}_{L^2(\R^d)} \Norm{g}_{L^2(\R^d)} \Norm{h^{-1}}_{L^{pd/\alpha}(\R^d)}.
	\end{align*}
	% where $\alpha = \sum_{=1}^k \alpha_i$ with $k \in \{1, \cdots ,d\}$, $\alpha_i \in (0,d_i)$ and $\varphi(x) = \prod_{i =1}^k |x_{(i)}|^{-(d_i-\alpha_i)}$ where $x_{(i)} = (x_{i_1},\cdots , x_{i_{d_i}} )$ is the $i^{th}$ group of $x=(x_1, \cdots, x_d)$.
\end{lemma}
\begin{proof}
	By observing the proof of Lemma 1.6 of \cite{BCR09}, we only need to prove that
	\begin{equation}
		\label{E:fracrhofinite}
		\int_{\R^d}\int_{\R^d} F(y)G(x)\varphi\left(x-y\right) \ud y \ud x
		\le C \Norm{F}_{L^{2d/(d+\alpha)}(\R^d)} \Norm{G}_{L^{2d/(d+\alpha)}(\R^d)},
	\end{equation}
	where
	\begin{align*}
		F(x) = \frac{|f(x)|}{(h(x))^{p/2}} \qquad \text{and} \qquad
		G(x) = \frac{|g(x)|}{(h(x))^{p/2}}.
	\end{align*}
	Note that when $\varphi(x)=C |x|^{-(d-\alpha)}$, \eqref{E:fracrhofinite} is nothing but (1.27)
	({\it ibid.}). Here we need to handle more general $\varphi$ as given in \eqref{E:specfunction}.
	To prove \eqref{E:fracrhofinite}, we need to apply the {\it weak version of Young's inequality}
	(see, e.g., \cite[eq. (7) on p. 107]{LiebLoss01}), which says that for all  $p,q,r >1$ with $1/p +
	1/q + 1/r =2$, it holds that
	\begin{align}
		\label{E:WeakYoung}
		\left|	\int_{\R^d}\int_{\R^d}a(x)b(x-y)c(y) \ud x \ud y \right|
		\le K_{p,q,r,d} \Norm{a}_{L^p(\R^d)}\Norm{b}_{q,w}\Norm{c}_{L^r(\R^d)},
	\end{align}
	where
	\begin{align*}
		\Norm{b}_{q,w} := \sup_{A} |A|^{-1/q'} \int_A |f(x)| \ud x, \qquad
		\text{with $1/q + 1/(q')=1$},
	\end{align*}
	and $A$ is an arbitrary Borel set of finite measure $|A|<\infty$.
	Now we apply \eqref{E:WeakYoung} with
	\begin{align*}
		a = F, \quad
		c = G, \quad
		b = \varphi, \quad
		p = r=2d/(d+\alpha), \quad \text{and} \quad
		q = \frac{d}{d-\sum_{j=1}^d\alpha_j }
      = d/(d-\alpha).
	\end{align*}
	By the proof of \cite[Lemma 1.6]{BCC21}, it suffices to prove that $\Norm{\varphi}_{q,w}$ is
	finite with $q = d/(d-\alpha)$ and $1/q' = 1 - 1/q = \alpha/d$.  \bigskip

	Recall that according to Assumption \ref{A:Noise}, the $d$ coordinates are partitioned into $k$
	groups.  Define $A_R := A_1 \times \cdots \times A_k$ where $A_i = B_{R,d_i}(0)$ is the ball in
	$\R^{d_i}$ centered at the origin with radius $R$.  With this we have that
	\begin{equation}
		\label{E:int}
		\int_{A_i} |x_{(i)}|^{-(d_i-\alpha_i)} \ud x_{(i)}
		= |\mathbb{S}^{d_i - 1}|\: \dfrac{R^{\alpha_i}}{\alpha_i},
	\end{equation}
	where we have used polar coordinated to calculate the integral and
	$\left|\mathbb{S}^{d_i-1}\right| = 2\pi^{d_i / 2}/\Gamma(d_i/2)$ is the surface area of the unit
	sphere in $\R^{d_i}$ (clearly, when $d_i=1$,  $|S^0|=2$). Moreover, by the formula for the volume
	of balls in $\R^{d_i}$, we see that
	\begin{equation}
		\label{E:leb}
		|A_i| = \frac{\pi^{d_i /2}}{\Gamma\left(1+ \frac{d_i}{2}\right)} R^{d_i}
          = |\mathbb{S}^{d_i - 1}|\dfrac{R^{d_i}}{d_i}.
	\end{equation}
	Recall that $1/q' = \alpha/d$. Then a combination of \eqref{E:int}, \eqref{E:leb} and
	\eqref{E:specfunction} shows that
	\begin{align*}
		|A_R|^{-1/q'}	\int_{A_R} \varphi(x) \ud x
    & = \prod_{i=1}^k C_{\alpha_i,d_i} |A_i|^{-1/q'} \int_{A_i} |x_{(i)}|^{-(d_i-\alpha_i)} \ud x_{(i)}                       \\
    & = \prod_{i=1}^k C_{\alpha_i,d_i} \alpha_i^{-1}|S^{d_i-1}|^{1-\alpha/d} R^{\alpha_i- \frac{d_i}{d}\alpha} d_i^{\alpha/d} \\
    & = \prod_{i=1}^k C_{\alpha_i,d_i} \alpha_i^{-1}|S^{d_i-1}|^{1-\alpha/d} d_i^{\alpha/d} =: K,
	\end{align*}
	where the constants $C_{\alpha_i,d_i}$ are defined in \eqref{E:noiseconstant} and the final
	constant $K$ does not depend on $R$. Finally, by symmetry of $\varphi$, we have that
	\begin{align}
		\label{E:WeakNorm}
		\Norm{\varphi}_{q,w} = \sup_{R>0}	|A_R|^{-1/q'}	\int_{A_R} \varphi(x) \ud x = K < \infty.
	\end{align}
	Hence, $\varphi \in L_{q,w}(\R^d)$ with $q = d/(d-\alpha)$. This completes the proof of Lemma \ref{L:rhofinite}.
\end{proof}

\begin{remark}
	Note that when there is only one partition (i.e., $k=1$), or equivalently when $\gamma$ itself is
	the Riesz kernel, by \cite[(6) on p.  107]{LiebLoss01}, we see that
	\begin{align*}
		\Norm{|\cdot|^{-(d-\alpha)}}_{\frac{d}{d-\alpha},w} = \alpha^{-1} |S^{d-1}|^{1-\alpha/d} d^{\alpha/d},
	\end{align*}
	which is consistent with the norm we find in \eqref{E:WeakNorm} up to a constant $C_{\alpha,d}$.
\end{remark}

In order to compare our results with known results, let us introduce another commonly used
variational constant
\begin{align}
	\label{E:VarConstE}
	\begin{aligned}
		\mathbf{E}_{a,d}(\gamma,\theta) & := \sup_{g \in \mathcal{F}_a} \left\{ \iint_{\R^{2d}} g^2(x)g^2(y) \gamma(x+y)\ud x\ud y - \frac{\theta}{2}\: \mathcal{E}_a(g,g) \right\} \\
                                    & = \sup_{g \in \mathcal{F}_a} \left\{ \left\langle g^2 * g^2, \gamma \right\rangle_{L^2(\R^d)} - \frac{\theta}{2}\: \mathcal{E}_a(g,g) \right\}.
	\end{aligned}
\end{align}
By using the same techniques used to derive \eqref{E:ScaleM}, one can show that for any $\Theta >0$ and $\theta>0$ that
\begin{equation}\label{E:ScaleE}
		\mathbf{E}_{a,d}(\Theta\gamma,\theta) =\Theta^{\frac{a}{a-\alpha}} \theta^{-\frac{\alpha}{a-\alpha}}  \mathbf{E}_{a,d}(\gamma,1).
\end{equation}
The relation between $\mathbf{E}_{a,d}(\gamma,\theta)$ and $\mathcal{M}_{a,d}(\gamma,\theta)$ can
be established in a similar way as \cite[Lemma A.2]{CHSX15AIHP}, which is stated in the following
lemma:

\begin{lemma}
	\label{L:EvsM}
	Under Assumption \ref{A:Noise}, the following three expressions hold:
	\begin{gather}
		\label{E:Erelation}
		\mathbf{E}_{a,d}(\Theta\gamma,\theta)  = \Theta^{\frac{a}{a-\alpha}} \theta^{-\frac{\alpha}{a-\alpha}}\left(\frac{2\alpha}{a}\right)^{\alpha/(a-\alpha)}\frac{a-\alpha}{a}\sigma(a, d,\alpha)^{a/(a-\alpha)},        \\
		\label{E:Mrelation}
		\mathcal{M}_{a,d}(\Theta\gamma,\theta) = \Theta^{\frac{a}{2a-\alpha}}\theta^{-\frac{\alpha}{2a-\alpha}} \left( \frac{\alpha}{a}\right)^{\alpha/(2a-\alpha)} \frac{2a-\alpha}{2a}\sigma(a, d,\alpha)^{a/(2a-\alpha)}, \\
		\label{E:EvsM}
		\mathbf{E}_{a,d}(\Theta\gamma,\theta)  = \left( \frac{a-\alpha}{a}\right) 2^{\alpha/(a-\alpha)}\left( \frac{2a-\alpha}{2a} \right)^{-(2a-\alpha)/(a-2)} \mathcal{M}_{a,d}(\Theta\gamma,\theta)^{\frac{2a-\alpha}{a-\alpha}}.
	\end{gather}
\end{lemma}

We need to prove two lemmas in order to prove Lemma \ref{L:EvsM}.

\begin{lemma}
	For any $f \in L^2(\R^d)$ with $\mathcal{E}_a(f,f) < \infty$, it holds that
	\begin{equation}
		\label{E:SobIneq}
		\int_{\R^d} f^2(x)\gamma(x) \ud x \le C \Norm{f}_2^{2-(2\alpha/a)} \mathcal{E}_a(f,f)^{\alpha/a},
	\end{equation}
	where the constant $C$ only depends on $a, d$ and $\alpha$. Denote the best constant in
	\eqref{E:SobIneq} by $\sigma(a,d,\alpha)$.
\end{lemma}

\begin{proof}
	The proof of this result follows the scheme laid out in the proof of  \cite[Lemma A.3]{ChenAOP12}. By the same techniques presented in Lemma \ref{T:Variation}, one can show the following quantity is finite:
	\begin{align*}
		\Lambda & := \sup_{f\in\mathcal{F}_a}  \left\{  \int_{\R^d} g^2(x)\gamma(x) \ud x - \frac{1}{2} (2\pi)^{-d}\int_{\R^d} |x|^a|\mathcal{F} g(x)|^2 \ud x  \right\} \\
            & =    \sup_{f\in\mathcal{F}_a}  \left\{  \int_{\R^d} g^2(x)\gamma(x) \ud x - \frac{1}{2} \mathcal{E}_a(f,f)\right\} < \infty.
	\end{align*}
	Fix an arbitrary $f \in \mathcal{F}_a$. Clearly, $\Norm{f}_2 = 1$ and $\mathcal{E}_a(f,f) <
	\infty$. Let $C_f$ be the constant such that
	\[
		\int_{\R^d} f^2(x)\gamma(x)\ud x = C_f  \mathcal{E}_a(f,f)^{\alpha / a}.
	\]
	Now for $g(x):= t^{d/2} f(tx)$, it is easy to see that $\Norm{g}_2 = 1$ and
	\begin{align*}
		\mathcal{E}_a(g,g)                = t^a\mathcal{E}_a(f,f) \quad \text{and} \quad
		\int_{\R^d} g^2(x)\gamma(x) \ud x = t^\alpha \int_{\R^d} f^2(x)\gamma(x)\ud x.
	\end{align*}
	From this we can deduce that
	\[
	\int_{\R^d} g^2(x)\gamma(x)\ud x = C_f \mathcal{E}_a(g,g)^{\alpha/a}.
	\]
	Next we note that
	\begin{align*}
		\Lambda & \ge \int_{\R^d} g^2(x)\gamma(x)\ud x - \frac{1}{2}\mathcal{E}_a(g,g)               \\
						& = t^\alpha \int_{\R^d} f^2(x)\gamma(x) \ud x - \frac{1}{2}t^a \mathcal{E}_a(f,f)   \\
						& = C_f t^{\alpha} \mathcal{E}_a(f,f)^{\alpha/a} - \frac{1}{2}t^a \mathcal{E}_a(f,f) \\
						& = C_f \left(t\mathcal{E}_a(f,f)^{1/a}\right)^\alpha - \frac{1}{2} \left( t \mathcal{E}_a(f,f)^{1/a}\right)^a.
	\end{align*}
	Since $t>0$, then $ t\mathcal{E}_a(f,f)^{1/a}$ runs through all of $\R$ and thus we have that
	\[
	\Lambda \ge \sup_{x>0} \left\{ C_fx^\alpha - \frac{1}{2}x^a \right\} = \frac{a-\alpha}{a}C_f^{a/(a-2)}\left(\frac{2\alpha}{a}\right)^{\alpha/(a-\alpha)}.
	\]
	Note that this reduces to the equation present in the proof of Lemma A.3 \cite{ChenAOP12} when $a=2$. By taking the $\sup$ over all $f\in \mathcal{F}_a$ we see that
	\[
	\infty > \Lambda \ge \frac{a-\alpha}{a}\sigma(d,\alpha)^{a/(a-2)}\left(\frac{2\alpha}{a}\right)^{\alpha/(a-\alpha)}
	\]
	where
	\[
	\sup_{f\in \mathcal{F}_a} C_f = \sigma(d,\alpha)
	\]
	and finally we conclude that for any $f\in \mathcal{F}_a$
	\begin{equation}\label{E:SobIneqNorm1}
	\int_{\R^d} \frac{f^2(x)}{|x|^\alpha} \ud x \le \sigma(d,\alpha) \mathcal{E}_a(f,f)^{\alpha/a}.
	\end{equation}
	For arbitrary $f\in L^2(\R^2)$ with $\mathcal{E}_a(f,f) < \infty$ we apply the \eqref{E:SobIneqNorm1} to $f/\Norm{f}_2$ and see that
	\begin{align*}
	\int_{\R^d} \frac{f^2(x)}{|x|^\alpha} \ud x \le \sigma(d,\alpha) \Norm{f}_2^{2-(2\alpha/a)} \mathcal{E}_a(f,f)^{\alpha/a}
	\end{align*}
	which again reduces to the equation A.4 \cite{ChenAOP12} when $a=2$.
\end{proof}

	\begin{lemma}
For any $f \in \mathcal{F}_a$ we have
		\begin{equation}\label{E:sobineq1}
			\int_{\R^{2d}} \gamma(x-y)f^2(x)f^2(y) \ud x \ud y \le \sigma(d,\alpha) \mathcal{E}_a(f,f)^{\alpha/a}.
		\end{equation}
	\end{lemma}

	\begin{proof}
Suppose that $f \in L^2(\R^d)$ and suppose that $\mathcal{E}_a(f,f) < \infty$ and let $y\in \R^d$ be arbitrary. Recall the translation property of the Fourier transform
	\[
		|\mathcal{F}f(x)|  = |\mathcal{F}f(x-y)|.
	\]
Then by making a substitution and applying \eqref{E:SobIneq} we see that
	\[
		\int_{\R^d} \frac{f^2(x)}{|x-y|^\alpha} \ud x \le \sigma(d,\alpha) \Norm{f}_2^{2-(2\alpha/a)}\mathcal{E}_a(f,f)^{\alpha/a}
	\]
and in return,
	\[
		\sup_{y\in \R^d} \int_{\R^d} \frac{f^2(x)}{|x-y|^\alpha} \le \sigma(d,\alpha) 		\Norm{f}_2^{2-(2\alpha/a)}\mathcal{E}_a(f,f)^{\alpha/a}.
	\]
Next, notice that
	\begin{align*}
		\int_{\R^{2d}} \frac{f^2(x)f^2(y)}{|x-y|^\alpha} \ud x \ud y &= \int_{\R^d} \ud x f^2(x) \int_{\R^d} \ud y \frac{f^(y)}{|x-y|^\alpha}
		\\& \le \sigma(d,\alpha) \Norm{f}_2^{4-(2\alpha/a)}\mathcal{E}_a(f,f)^{\alpha/a}.
	\end{align*}
Now for two functions $f$ and $g$ and  by using the property that $\gamma(x) = C(\gamma) (K*K)(x)$, we see that
	\begin{align*}
		\int_{\R^{2d}} & \gamma(x-y)f^2(x)g^2(y) \ud x \ud y = C(\gamma) \int_{\R^d}\left[ \int_{\R^d}K(z-x)f^2(x) \ud x \right]\left[ \int_{\R^d}K(z-y)g^2(y) \ud y \right]\ud z
		\\&\le C(\gamma) \left\{\int_{\R^d}\left[ \int_{\R^d}K(z-x)f^2(x) \ud x \right]^2 \ud z \right\}^{1/2}\left\{\left[ \int_{\R^d}K(z-y)g^2(y) \ud y \right]^2\ud z\right\}^{1/2}
		\\&= \left\{ \int_{\R^{2d}} \gamma(x-y)f^2(x)f^2(y) \ud x \ud y \right\}^{1/2} \left\{ \int_{\R^{2d}} \gamma(x-y)f^2(x)f^2(y) \ud x \ud y \right\}^{1/2}
		\\&\le \sigma(d,\alpha) \Norm{f}_2^{2-(\alpha/a)}\mathcal{E}_a(f,f)^{\alpha/(2a)} \Norm{g}_2^{2-(1\alpha/a)}\mathcal{E}_a(g,g)^{\alpha/(2a)}
	\end{align*}
and when $f \in \mathcal{F}_a$ we see that this reduces to
	\[
		\int_{\R^{2d}} \gamma(x-y)f^2(x)f^2(y) \ud x \ud y \le \sigma(d,\alpha) \mathcal{E}_a(f,f)^{\alpha/a}.
	\]
We note that this reduces to equation (A.1) \cite{15AIHP} for the time independent case when $a=2$.
	\end{proof}

	\begin{proof}[Proof of Lemma \ref{L:EvsM}:]
We only prove the case of $\Theta=\theta=1$, the general case can be proven by applying the scaling properties \eqref{E:ScaleM} and \eqref{E:ScaleE}.

We have that
	\begin{align}
		E_{a,d}(\gamma,1) &\le \sup_{g\in \mathcal{F}_a} \left\{ \sigma(d,\alpha) \mathcal{E}_a(f,f)^{\alpha/a} - \frac{1}{2}\: (\mathcal{E}_a(g,g)^{1/a})^a \right\} \nonumber
		\\& \le \sup_{x>0} \left\{ \sigma(d,\alpha) x^\alpha - \frac{1}{2}x^a \right\} \nonumber
		\\& = \left(\frac{2\alpha}{a}\right)^{\alpha/(a-\alpha)}\frac{a-\alpha}{a}\sigma(d,\alpha)^{a/(a-\alpha)} \label{E:Eupperbound}
	\end{align}
and similarly
	\begin{align}
		\mathcal{M}_{a,d}(\gamma,1) & \le \sup_{x > 0} \left\{ \sigma(d,\alpha)^{1/2} x^{\alpha/2} - \frac{1}{2} x^{a} \right\} \nonumber
		\\&= \left( \frac{\alpha}{a}\right)^{\alpha/(2a-\alpha)} \frac{2a-\alpha}{2a}\sigma(d,\alpha)^{a/(2a-\alpha)}. \label{E:Mupperbound}
	\end{align}
Now choose $0 < \epsilon < \sigma(d,\alpha)$ and let $f\in \mathcal{F}_a$ be such that
	\[
		\int_{\R^{2d}} \gamma(x-y)f^2(x)f^2(y) \ud x \ud y \ge (\sigma(d,\alpha)-\epsilon) \mathcal{E}_a(f,f)^{\alpha/a}.
	\]
Now define
	\[
		g(x) = t^{d/2}f(tx).
	\]
Notice that
	\begin{align*}
		E_{a,d}(\gamma,1) &\ge \int_{\R^{2d}} g^2(x)g^2(y) \gamma(x-y)\ud x\ud y - \frac{1}{2}\: \mathcal{E}_a(g,g)
		\\&= t^\alpha \int_{\R^{2d}} \gamma(x-y) f^2(x)f^2(y) - 		\frac{1}{2} t^a \mathcal{E}_a(f,f)
		\\& \ge (\sigma(d,\alpha)-\epsilon) t^{\alpha} \mathcal{E}_a(f,f)^{\alpha/a} - \frac{1}{2} t^a \mathcal{E}_a(f,f)
	\end{align*}
and this is true for all $t>0$ so we can say that
	\begin{align*}
		E_{a,d}(\gamma,1) &\ge \sup_{x>0} \left\{ (\sigma(d,\alpha)-\epsilon) x^\alpha - \frac{1}{2} x^a \right\} \nonumber
\\& = \left(\frac{2\alpha}{a}\right)^{\alpha/(a-\alpha)}\frac{a-\alpha}{a}(\sigma(d,\alpha)-\epsilon)^{a/(a-\alpha)}.
	\end{align*}
and be letting $\epsilon \to 0$ gives us
	\begin{equation}\label{E:Elowerbound}
		E_{a,d}(\gamma,1) \ge\left(\frac{2\alpha}{a}\right)^{\alpha/(a-\alpha)}\frac{a-\alpha}{a}\sigma(d,\alpha)^{a/(a-\alpha)}.
	\end{equation}
and if we combine this with \eqref{E:Eupperbound} then we see that
	\begin{equation}\label{E:Eequality1}
		E_{a,d}(\gamma,1) = \left(\frac{2\alpha}{a}\right)^{\alpha/(a-\alpha)}\frac{a-\alpha}{a}\sigma(d,\alpha)^{a/(a-\alpha)}.
	\end{equation}
Similarly we show that
	\begin{equation}\label{E:Meqquality1}
		\mathcal{M}_{a,d}(\gamma,1) = \left( \frac{\alpha}{a}\right)^{\alpha/(2a-\alpha)} \frac{2a-\alpha}{2a}\sigma(d,\alpha)^{a/(2a-\alpha)}.
	\end{equation}
Lastly, by combining \eqref{E:Eequality1} and \eqref{E:Meqquality1}, we see that
	\begin{equation} \label{E:EandMrelation1}
		E_{a,d}(\gamma,1) = \left( \frac{a-\alpha}{a}\right)(2)^{\alpha/(a-\alpha)}\left( \frac{2a-\alpha}{2a} \right)^{-(2a-\alpha)/(a-2)}\mathcal{M}_{a,d}(\gamma,1)^{(2a-\alpha)/(a-2)} .
	\end{equation}
Equations \eqref{E:Meqquality1}, \eqref{E:Eequality1} and \eqref{E:EandMrelation1} recover equations (A.2), (A.3) and (A.4) \cite{15AIHP} respectively when $a=2$.
	\end{proof}
% \section{Existence and uniqueness of the solution --- the proof of Theorem \ref{T:solvable}}
\section{Existence and uniqueness of the solution}
\label{S:solvable}

In this section, we will prove Theorem \ref{T:solvable}. The proof will need the following lemma:

\begin{lemma}[Lemma 3.5 of \cite{BCC21}]
	\label{L:doubleExp}
	If $H:[0,\infty) \to [0,\infty)$ is a non-decreasing function, then
	\begin{equation}
		\label{E:doubleExp}
		2\int_0^\infty e^{-2t}H^2(t) \ud t \le \left( \int_0^\infty e^{-t} H(t) \ud t \right)^2.
	\end{equation}
\end{lemma}

The proof of Theorem \ref{T:solvable} follows the same strategy as \cite[Section 3]{BCC21} with minor changes such as
\begin{align}
	\label{E:Changes}
	\frac{1}{1+|\xi|^2} \qquad \text{replaced by} \qquad \frac{1}{1+\frac{\nu}{2}|\xi|^a}.
\end{align}
Nevertheless, for completeness, here we streamline and reorganize this proof as follows.

\begin{proof}[Proof of Theorem \ref{T:solvable}]
	We first introduce some notation. Let $L(x)$ be the Laplace transform of $G(\cdot,x)$ evaluated at
	one and calculate its Fourier transform $\mathcal{F}L(\xi)$ as follows:
	\begin{equation}
		\label{E:LandFL}
		L(x)              = \int_0^\infty e^{-t}G(t,x) \ud t \quad \text{and} \quad
		\mathcal{F}L(\xi) = \int_0^\infty e^{-t}\mathcal{F}G(t,\cdot)(\xi) \ud t
											= \dfrac{1}{1 + \frac{\nu}{2}|\xi|^a};
	\end{equation}
	see the proof of Theorem 4.1 of \cite{CHN19} for the last equality.  Similarly, let $L_n(\vec{y})$
	to be the Laplace transform of $\widetilde{f_n}(\vec{y},0,\cdot)$ evaluated at one, namely,
	\begin{align*}
		L_n(\vec{y}) & := n! \int_0^\infty e^{-t}\widetilde{f_n}(\vec{y};0,t) \ud t \\
								 & = \sum_{\sigma \in \Sigma_n} \int_0^\infty e^{-t} \int_{[0,t]^n<} \prod_{k=1}^n G(s_k - s_{k-1}, y_{\sigma(k)}-y_{\sigma(k-1)}) \ud \vec{s} \ud t
	\end{align*}
	with the convention that $s_0 = 0$ and $y_{\sigma(0)} = 0$.  By the relation of convolution and the
	Laplace transform (or through a change of variables), we see that
	\begin{equation}
			L_n(\vec{y}) = \sum_{\sigma \in \Sigma_n} L\left(y_{\sigma(1)}\right) L\left(y_{\sigma(2)}-y_{\sigma(1)}\right) \cdots L\left(y_{\sigma(n)}-y_{\sigma(n-1)}\right).
	\end{equation}
	Hence, from \eqref{E:LandFL},
	\begin{equation}
		\label{E:FLfn}
		\mathcal{F}L_n(\vec{\xi}) = \sum_{\sigma \in \Sigma_n} \prod_{k=1}^{n} \dfrac{1}{1+ \frac{\nu}{2}\left|\sum_{j=k}^n \xi_j \right|^a}.
	\end{equation}
	Moreover, define
	\begin{align*}
		H_n(t, \vec{x}) & = n! \int_{(\R^d)^n} \prod_{k=1}^n K(x_k - y_k) \widetilde{f_n}(\vec{y};0,t) \ud \vec{y} \\
										& = \sum_{\sigma \in \Sigma_n} \int_{[0,t]^n <} \int_{(\R^d)^n} \prod_{k=1}^n K(x_k-y_k) 	\prod_{k=1}^n G(s_k - s_{k-1}, y_{\sigma(k)} - y_{\sigma(k-1)}) \ud \vec{y} \ud \vec{s}.
	\end{align*}
	Under the nonnegativity assumption  --- Assumption \ref{A:Nonnegative}, we see that for any $\vec{x}
	\in\R^{nd}$ fixed, the function $t\to H_n\left(t,\vec{x} \right)$ is a non-decreasing function for
	$t\ge 0$. For this function, we are about to apply Lemma \ref{L:doubleExp}.

	{\bigskip\bf\noindent Step 1.~} We first compute the corresponding part to the right-hand side of
	\eqref{E:doubleExp}. By Fubini's theorem,
	\begin{align*}
			& \int_0^\infty e^{-t} H_n(t,\vec{x}) \ud t                                                                                                                                                                   \\
			& = \sum_{\sigma \in \Sigma_n} \int_0^\infty e^{-t} \int_{[0,t]^n <} \int_{(\R^d)^n} \prod_{k=1}^n K(x_k-y_k) 	\prod_{k=1}^n G(s_k - s_{k-1}, y_{\sigma(k)} - y_{\sigma(k-1)}) \ud \vec{y} \ud \vec{s} \ud t \\
			& = \int_{\R^{dn}} \prod_{k=1}^n K(x_k-y_k) L_n\left(\vec{y}\right) \ud \vec{y}.
	\end{align*}
	Then an application of the Plancherel's theorem and the fact that $K*K=\gamma$ shows that
	\begin{equation}
		\label{E:Hneq1}
		\int_{(\R^d)^n} \left[ \int_0^\infty e^{-t} H_n(t,\vec{x}) \ud t \right]^2 \ud \vec{x}
		= \int_{(\R^d)^n} | \mathcal{F}L_n(\vec{\xi}) |^2 \mu(\ud \vec{\xi});
	\end{equation}
	one may check the proof of Lemma 3.3 of \cite{BCC21} for more details.

	{\bigskip\bf\noindent Step 2.~} Now we compute the corresponding part to the left-hand side of
	\eqref{E:doubleExp}.  First, using the fact that $K*K=\gamma$, we see that
	\begin{equation}
		\label{E:inteq1}
		\Norm{\widetilde{f_n}(\cdot,0;t)}_{\mathcal{H}^{\otimes n}}^2 = \frac{1}{(n!)^2}\int_{(\R^d)^n} H_n^2(t,\vec{x}) \ud \vec{x};
	\end{equation}
	one may check the proof of Lemma 3.4 of \cite{BCC21} for more details.  By the scaling property
	for $\mathcal{F}\widetilde{f}_n$ in \eqref{E:Ffnscale}, one can show that
	\begin{align}
		\label{E:fHn}
		\int_0^\infty e^{-t} \Norm{ \widetilde{f}_n(\cdot,0,t) }_{\mathcal{H}^{\otimes n}}^2 \ud t
		 & = \frac{2^{n(2(b+r) - b\alpha/a)}}{(n!)^2}  \int_0^\infty 2e^{-2t} \int_{\R^{nd}} H_n(t,\vec{x})^2 \ud \vec{x} \ud t;
	\end{align}
	see Appendix for the proof.

	{\bigskip\bf\noindent Step 3.~} Now we can apply Fubini's theorem and Lemma \ref{L:doubleExp} to
	the function $t\to H_n(t,\vec{x})$ to see that
	\begin{align}
		\label{E:doubleH}
		\int_0^\infty 2e^{-2t} \int_{\R^{nd}} H_n(t,\vec{x})^2 \ud \vec{x} \ud t
		\le \int_{\R^{nd}} \left[ \int_0^\infty e^{-t} H_n(t,\vec{x}) \ud t\right ]^2 \ud \vec{x}.
	\end{align}
	Therefore, combining \eqref{E:Hneq1}, \eqref{E:inteq1}, and \eqref{E:doubleH} shows that
	\begin{align}
		\label{E:fTn}
		\begin{aligned}
			\int_0^\infty e^{-t} \Norm{ \widetilde{f}_n(\cdot,0,t) }_{\mathcal{H}^{\otimes n}}^2 \ud t
			& \le \frac{2^{n(2(b+r) - b\alpha/a)}}{(n!)^2} \int_{(\R^d)^n} | \mathcal{F}L_n(\vec{\xi}) |^2 \mu(\ud \vec{\xi}) \\
			& =: \frac{2^{n(2(b+r) - b\alpha/a)}}{(n!)^2} T_n(\nu,a),
		\end{aligned}
	\end{align}
	where
	\begin{equation}
		\label{E:Tn}
		T_n(\nu,a) = \int_{(\R^d)^n} \left[\sum_{\sigma \in \Sigma_n} \prod_{k=1}^n \frac{1}{1+\frac{\nu}{2}\left|\sum_{j=k}^n \xi_{\sigma(j)} \right|^a}\right]^2 \mu(\ud \vec{\xi}).
	\end{equation}
	By the same arguments as those of Lemma 3.6 of \cite{BCC21} with the replacement
	\eqref{E:Changes}, we see that
	\begin{equation}
		\label{E:muassumption}
		T_n(\nu, a) \le (n!)^2 C_\mu^n(\nu, a)
		\qquad \text{with} \quad
		C_\mu(\nu,a) := \int_{(\R^d)} \left(\frac{1}{1+\frac{\nu}{2}|\xi|^a}\right)^2 \mu(\ud \xi)
		< \infty.
	\end{equation}
	Here, the finiteness of $C_\mu(\nu,a)$ entails one part of the conditions in \eqref{E:alpha}:
	\begin{align}
		\label{E:alpha2}
		\int_{(\R^d)} \left(\frac{1}{1+|\xi|^a}\right)^2 \mu(\ud \xi)
		= C \int_0^\infty \frac{\rho^{\alpha-1}}{(1+\rho^{a})^2} \ud \rho<\infty
		\quad \Longleftrightarrow \quad
		2a - \alpha +1 > 1.
	\end{align}
	Combining \eqref{E:fTn} and \eqref{E:muassumption} gives that
	\begin{align}
		\label{E:LNfnineq}
		\int_0^\infty e^{-t} \Norm{\widetilde{f_n}(\cdot,0;t)}_{ \mathcal{H}^{ \otimes n } }^n \ud t
		\le 2^{n\left[2(b+r)-\frac{b\alpha}{a}\right]} C_\mu^n(\nu,a).
	\end{align}

	{\bigskip\bf\noindent Step 4.~}
	From the scaling property \eqref{E:normscale} we see that
	\begin{align}
		\nonumber
		\int_0^\infty e^{-t} \Norm{ \widetilde{f_n}\left(\cdot,0,t\right) }_{ \mathcal{H}^{\otimes n} }^2 \ud t
    & = \Norm{ \widetilde{f_n}(\cdot,0,1) }_{ \mathcal{H}^{\otimes n} }^2 \int_0^\infty e^{-t} t^{ \left[ 2(b+r)-\frac{b \alpha}{a} \right] n } \ud t \\
    & = \Norm{ \widetilde{f_n}(\cdot,0,1) }_{ \mathcal{H}^{\otimes n} }^2 \Gamma\left( \left[ 2(b+r)- \frac{b \alpha}{a} \right] n +1 \right),
		\label{E:fnOne}
	\end{align}
	which entails another part of the conditions in \eqref{E:alpha}:
	\begin{equation}
		\label{E:alpha3}
		2(b+r)-\frac{b\alpha}{a} >0.
	\end{equation}
	From \eqref{E:LNfnineq} and \eqref{E:fnOne}, we deduce that
	\begin{align*}
		\Norm{ \widetilde{f_n}(\cdot,0,1) }_{ \mathcal{H}^{\otimes n} }^2
    & = \dfrac{1}{ \Gamma\left( \left[ 2(b+r)- \frac{b \alpha}{a} \right] n +1 \right) }  \int_0^\infty e^{-t} \Norm{ \widetilde{f_n}(\cdot,0,t) }_{ \mathcal{H}^{\otimes n} }^2 \ud t \\
		% & \le \dfrac{ 2^{n \left[ 2(b+r)- \frac{b \alpha}{a} \right]} }{ \Gamma\left( \left[ 2(b+r)- \frac{b \alpha}{a} \right] n +1 \right) } T_n                                        \\
    & \le \dfrac{ \left(2^{ \left[ 2(b+r)- \frac{b \alpha}{a} \right]}C_\mu(\nu, a) \right)^n }{\Gamma\left( \left[ 2(b+r)- \frac{b \alpha}{a} \right] n +1 \right) }
		\le C \dfrac{ \left(2^{ \left[ 2(b+r)- \frac{b \alpha}{a} \right]}C_\mu(\nu,a) \right)^n }{ (n!)^{ 2(b+r) - b\alpha/a }},
	\end{align*}
	where the constant $C$ depends only on the value of $2(b+r)-b\alpha/a$.  Therefore, by
	\eqref{E:2mom}, \eqref{E:normscale}, and the above inequality,
	\begin{align}
		\nonumber
		\Norm{u(t,x)}_2^2 & = \sum_{n \ge 0} \theta^n n! \: t^{[2(b + r ) -b \alpha / a]n}  \Norm{ \widetilde{f_n}(\cdot,0,1) }_{ \mathcal{H}^{\otimes n} }^2 \\
                      & \le  C \sum_{n \ge 0} \theta^n  t^{[2(b + r ) -b \alpha / a]n} \dfrac{ \left(2^{ \left[ 2(b+r)- \frac{b \alpha}{a} \right]}C_\mu(\nu, a) \right)^n }{ (n!)^{ 2(b+r) - 1 - b\alpha/a }},
		\label{E:2nomBd}
	\end{align}
	which is finite provided that (see \eqref{E:alpha})
	\begin{equation}
		\label{E:alpha4}
		2(b+r) - 1 - b\alpha/a > 0.
	\end{equation}
	Finally, an application of the Minkowski inequality and the hypercontractivity (see \cite[Theorem
	B.1]{BCC21} or \cite{Le16} for the case of SHE) shows that for all $p\ge 2$,
	\begin{align}
		\nonumber
		\Norm{u(t,x)}_p & \le \sum_{n\ge 0} \theta^{n/2} (p-1)^{n/2} \sqrt{n!} \Norm{\widetilde{f_n}(\cdot,0,t) }_{ \mathcal{H}^{\otimes n} } \\
                    & \le  \sqrt{C} \sum_{n \ge 0} \theta^{n/2} (p-1)^{n/2}  t^{[2(b + r ) -b \alpha /a]n/2} \dfrac{ \left(2^{ \left[ 2(b+r)- \frac{b \alpha}{a} \right]}C_\mu(\nu, a) \right)^{n/2} }{ (n!)^{\frac{1}{2}\left[ 2(b+r) - 1 - b\alpha/a \right]}}.
		\label{E:pmomBd}
	\end{align}
	Therefore, under condition \eqref{E:alpha}, \eqref{E:SPDE} has an unique $L^p(\Omega)$-solution
	$u(t,x)$ for all $p\ge 2$,  $t>0$ and $x\in\R^d$.  This proves part (1) of Theorem
	\ref{T:solvable}.

	{\bigskip\bf\noindent Step 5.~} The proof of part (2) of Theorem \ref{T:solvable} will be
	postponed to part (ii) of Lemma \ref{L:2ineq} below.
\end{proof}

% \section{Upper bound of the asymptotics --- the proof of \texorpdfstring{\eqref{E:momentasym-upper}}{plain_text}}
\section{Upper bound of the asymptotics}
\label{S:upperbd}

In this section, we will give the proof of part 2 of Theorem \ref{T:solvable} and establish the upper bound of \eqref{E:momentasym} (under Assumptions
\ref{A:Nonnegative} and \ref{A:Noise}, and condition \eqref{E:alpha}),  namely,
\begin{equation}
	\label{E:momentasym-upper}
	\begin{aligned}
		\limsup_{t_p \to \infty} t_p^{-\beta} \log \Norm{u(t,x)}_p \le & \left(\frac{1}{2}\right)\left(\frac{2a}{2a(b+r)- b\alpha} \right)^\beta \\
                                                                   & \times \left(\theta\nu^{-\alpha/a} \mathcal{M}_a^{\frac{2a-\alpha}{a}}\right)^{\frac{a}{2a(b+r)-b\alpha-a}}\left(2(b+r)-\frac{b\alpha}{a}-1\right).
	\end{aligned}
\end{equation}

As for the upper bound, we will first establish the corresponding result for $p=2$ in Lemma \ref{L:p=2} and then apply the
hypercontractivity property given by Theorem B.1 in \cite{BCC21} to obtain the general case for
$p\ge 2$. To prove the next lemma, we will need the following equality
	\begin{align}
		\label{E:Stirling}
		\lim_{n \to \infty} \frac{1}{n}\log\left( \frac{\Gamma(an+1)}{(n!)^a} \right) = a \log(a), \quad
		\text{for all $a>0$,}
\end{align}
which is a direct consequence of Stirling's formula.

\begin{lemma}
	\label{L:2ineq}
	Assume Assumptions \ref{A:Nonnegative} and \ref{A:Noise} hold. \\
	Under condition \eqref{E:alpha},
	\begin{align}
		\label{E:2ineq-a}
		&\lim_{n \to \infty}  \frac{1}{n} \log \left( \Norm{ \widetilde{f_n}(\cdot,0,1) }_{ \mathcal{H}^{\otimes n} }^2 \Gamma\left( \left[ 2(b+r)- \frac{b \alpha}{a} \right] n +1 \right) \right)
		= \log\left(2^{2(b+r)-\frac{b\alpha}{a}}\rho\right) \quad \text{and}\\
		\label{E:2ineq-b}
		& \lim_{n \to \infty}  \frac{1}{n} \log \left( (n!)^{2(b+r)- \frac{b \alpha}{a} }\Norm{ \widetilde{f_n}(\cdot,0,1) }_{ \mathcal{H}^{\otimes n} }^2  \right)
		=  \log\left( \left(\frac{2}{2(b+r)-\frac{b\alpha}{a}}\right)^{2(b+r)-\frac{b\alpha}{a}}\right) + \log\rho.
	\end{align}

\end{lemma}
\begin{proof}
	The proof follows the same arguments as those of \cite[Lemma 4.3]{BCC21}.
	Nevertheless, we sketch the proof here for completeness. Recall the definition of $T_n(\nu,a)$
	defined in \eqref{E:Tn}.  From \eqref{E:LogRho}, we see that
	\begin{align}
		\label{E:LimTn}
		% \lim_{n \to \infty} \frac{1}{n} \log \left[ \frac{1}{(n!)^2} \int_{(\R^d)^n} \left( \sum_{\sigma \in \Sigma_n} \prod_{k=1}^n \frac{1}{1+ \frac{\nu}{2}|\sum_{j=k}^n \xi_{\sigma(j)}|^a}\right)^2 \mu (\ud \vec{\xi})\right]
		\lim_{n \to \infty} \frac{1}{n} \log \left[ \frac{T_n(\nu,a)}{(n!)^2} \right]
		= \log\left(\rho_{\nu,a}\right).
	\end{align}
	As a consequence of \eqref{E:fTn} and \eqref{E:fnOne} in the proof of Theorem \ref{T:solvable}, we
	see that
	% \begin{align*}
	% 	\int_0^\infty e^{-t} \Norm{ \widetilde{f_n}(\cdot,0,t) }_{ \mathcal{H}^{\otimes n} }^2 \ud t = \Norm{ \widetilde{f_n}(\cdot,0,1) }_{ \mathcal{H}^{\otimes n} }^2 \Gamma\left( \left[ 2(b+r)- \frac{b \alpha}{a} \right] n +1 \right)
	% \end{align*}
	% and then \eqref{E:LNfnineq} above shows that
	\begin{align*}
		\Norm{ \widetilde{f_n}(\cdot,0,1) }_{ \mathcal{H}^{\otimes n} }^2 \Gamma\left( \left[ 2(b+r)- \frac{b \alpha}{a} \right] n +1 \right) \le \dfrac{2^{n\left[2(b+r)-\frac{b\alpha}{a}\right]}}{(n!)^2} T_n(\nu,a).
	\end{align*}
	Combining this and \eqref{E:LogRho} we see that
	\begin{align*}
		\limsup_{n \to \infty} \frac{1}{n} & \log\left[ \Norm{ \widetilde{f_n}(\cdot,0,1) }_{ \mathcal{H}^{\otimes n} }^2 \Gamma\left( \left[ 2(b+r)- \frac{b \alpha}{a} \right] n +1 \right)  \right] \\
                                       & \le \log\left( 2^{\left[2(b+r)-\frac{b\alpha}{a}\right]} \right) + \lim_{n \to \infty} \frac{1}{n} \log\left( \frac{T_n(\nu,a)}{(n!)^2} \right)           \\
                                       & = \log\left(  2^{\left[2(b+r)-\frac{b\alpha}{a}\right]} \right)  + \log(\rho_{\nu,a}),
	\end{align*}
	which proves the upper bound for \eqref{E:2ineq-a}.  \bigskip

	Now we prove the lower bound for \eqref{E:2ineq-a}. Let $\tau$ and $\widetilde{\tau}$ be
	independent exponential random variables with mean one. In the following, we will compute
	$\mathbb{E}[J_n(\tau,\widetilde{\tau})]$ in two ways, where
	\begin{align*}
		J_n(t,t') := \int_{\R^{nd}} H_n(t,x)H_n(t',x) \ud x, \qquad t,t'>0.
	\end{align*}
	Notice that using the above notation, \eqref{E:inteq1} can be rewritten as
	\begin{align*}
		\Norm{\widetilde{f_n}(\cdot,0;t)}_{\mathcal{H}^{\otimes n}}^2 = \frac{1}{(n!)^2} \int_{(\R^{d})^n} H_n(t,\vec{x})^2\ud \vec{x} = \frac{1}{(n!)^2} J_n(t,t).
	\end{align*}
	On the one hand, the Cauchy-Schwartz inequality implies that
	\begin{align*}
		J_n(t,t') \le J_n(t,t)^{1/2}J_n(t',t')^{1/2} =  t^{[2(b + r ) -b \alpha / a]n/2} ( t' )^{[2(b + r ) -b \alpha / a]n/2} J_n(1,1).
	\end{align*}
	where we have used the scaling property of $J_n(t,t)$ inherited from that of
	$\Norm{\widetilde{f_n}(\cdot,0;t)}_{\mathcal{H}^{\otimes n}}$ as in \eqref{E:normscale}.
	Hence,
	\begin{align*}
		\mathbb{E}[J_n(\tau,\widetilde{\tau})] & \le \mathbb{E}\left[  {\tau}^{[2(b + r ) -b \alpha / a]n/2} \right] \mathbb{E}\left[  {\widetilde{\tau}}^{[2(b + r ) -b \alpha / a]n/2} \right]J_n(1,1) \\
                                           % & = \Gamma\left( \frac{[2(b + r ) -b \alpha / a]n}{2} \right)^2 J_n(1,1)                                                                                  \\
                                           & = \Gamma\left( \frac{[2(b + r ) -b \alpha / a]n}{2} \right)^2 (n!)^2	\Norm{\widetilde{f_n}(\cdot,0;1) }_{ \mathcal{H}^{\otimes n} }^2.
	\end{align*}
	On the other hand, by \eqref{E:Hneq1} and \eqref{E:Tn}, we see that
	\begin{align*}
		\mathbb{E}[J_n(\tau,\widetilde{\tau})] & = \int_0^\infty \int_0^\infty e^{-t}e^{-\widetilde{t}}J_n(t,\widetilde{t}) \ud t \ud \widetilde{t}
                                             = \int_{\R^{dn}}\left[ \int_0^\infty e^{-t}H_n(t,x) \ud t \right]^2 \ud x
                                             = T_n\left(\nu,a\right).
                                           % & = \int_{\R^{nd}} |\mathcal{F}L_n(\xi_1,...,\xi_n)|^2 \mu(\ud \xi_1) \cdots \mu(\ud \xi_n).
	\end{align*}
	Therefore,
	\begin{align}
		\label{E:2lowbd}
		\frac{T_n(\nu,a)}{(n!)^2} \le \Gamma\left( \frac{[2(b + r ) -b \alpha / a]n}{2} \right)^2 	\Norm{ \widetilde{f_n}(\cdot,0;1) }_{ \mathcal{H}^{\otimes n} }^2 .
	\end{align}
	Now, an application of Stirling's formula shows that, as $n\to \infty$,
	\begin{align}
		\label{E:GammaSq}
		\Gamma\left( \frac{[2(b + r ) -b \alpha / a]n}{2} \right)^2  \sim \Gamma\left( [2(b + r ) -b \alpha / a]n +1 \right)2^{- [2(b + r ) -b \alpha / a]n}C_n,
	\end{align}
	where $C_n = 2^{-1}[2(b + r ) -b \alpha / a]^{1/2}(2 \pi n )^{1/2}$. Then an application of
	\eqref{E:LimTn}, \eqref{E:2lowbd} and \eqref{E:GammaSq} proves the lower bound of
	\eqref{E:2ineq-a}. Lastly, \eqref{E:2ineq-b} follows from \eqref{E:2ineq-a} and the limit
	\eqref{E:Stirling}. This proves Lemma \ref{L:2ineq}.
	\end{proof}

	Now are are ready to prove part (2) of Theorem \ref{T:solvable}.

	\begin{proof}[Proof of part (2) of Theorem \ref{T:solvable} ] The critical case happens when the exponent of $n!$ in \eqref{E:pmomBd}
	vanishes, namely,
	\begin{align*}
		\alpha = \frac{a}{b}\left[2(b+r)-1\right].
	\end{align*}
	Among the three inequalities in \eqref{E:alpha}, we also need to make sure that the minimum is
	achieved by $\frac{a}{b}\left[2(b+r)-1\right]$, for which, we need to additionally require
	$\frac{a}{b}\left[2(b+r)-1\right]\le d$ and
	\begin{align*}
		\frac{a}{b}\left[2(b+r)-1\right] \le 2a \quad \Longleftrightarrow \quad r\in[0,1/2].
	\end{align*}
	Putting these conditions together gives the conditions stated in \eqref{E:critical}.

	We start by proving part \textit{(2-i)}. Let $u_{\lambda}$ for $\lambda >0$ be the solution of
	the SPDE \eqref{E:SPDE} with $\theta$ replaced with $\lambda$ and $u=u_\theta$. By the
	hyperactivity property (see \cite[Lemma B.1]{BCC21}), we have that
	\begin{align}
		\label{E:HyContr}
		\Norm{u(t,x)}_p \le \Norm{u_{(p-1)\theta}(t,x)}_2, 	\quad \text{for all $p\ge 2$, $t>0$ and $x\in\R^d.$}
	\end{align}
	Now by recalling Theorem \ref{T:2momseries} and by applying $\alpha = \frac{a}{b}\left[ 2(b+r)
	- 1 \right]$ and the scaling property \eqref{E:normscale}, we see that
	\begin{align*}
	\Norm{u_{(p-1)\theta}(t,x)}_2^2 & =	\sum_{n \ge 0} \left[\theta(p-1)\right]^n n! \Norm{\tilde{f}_n(\cdot,0;t)}_{ \mathcal{H}^{\otimes n }}^2  \\
                                  & = \sum_{n \ge 0} \left[t\theta(p-1)\right]^n n! \Norm{\tilde{f}_n(\cdot,0;1)}_{ \mathcal{H}^{\otimes n }}^2 \\
                                  & =: \sum_{n \ge 0} \left[t\theta(p-1)\right]^n R_n,
	\end{align*}
	with $R_n = n! \Norm{\tilde{f}_n(\cdot,0;1)}_{ \mathcal{H}^{\otimes n }}^2$. By the
	Cauchy-Hadamard theorem, this series converges for $\theta t (p-1) < \limsup_{n \to \infty}
	|R_n|^{-1/n}$. However, by \eqref{E:2ineq-b} and Theorem \ref{T:Variation}, we see that
	\begin{align*}
		\lim_{n \to \infty} \frac{1}{n} \log(|R_n|) = \log\left( 2 \rho \right) =  \log\left( 2 \nu^{-\alpha/a} \mathcal{M}_a^{(2a-\alpha)/a} \right).
	\end{align*}
	Therefore, $\limsup_{n \to \infty} |R_n|^{-1/n} = (2 \nu^{-\alpha/a} \mathcal{M}_a^{(2a-\alpha)/a}
	)^{-1}$ and $\Norm{u(t,x)}_p$ converges for
	\begin{align*}
		t  < \dfrac{1}{2 \theta (p-1) \nu^{-\alpha/a} \mathcal{M}_a^{(2a-\alpha)/a} } =: T_p; \qquad \text{see \eqref{E:Tp}.}
	\end{align*}

	To show part \textit{(2-ii)}, we use the Cauchy-Hadamard theorem and the same techniques above to see that the radius of convergence of the series
	\begin{align*}
		\Norm{u(t,x)}_2^2 = \sum_{n \ge 0} \theta^n n! \Norm{\tilde{f}_n(\cdot,0;t)}_{ \mathcal{H}^{\otimes n }}^2,
	\end{align*}
	is precisely $T_2$. This completes the proof of Lemma \ref{L:2ineq}.
\end{proof}

\begin{lemma}
	\label{L:p=2}
	Assume Assumptions \ref{A:Nonnegative} and \ref{A:Noise} hold. Under condition \eqref{E:alpha}, we
	have that
	\begin{equation}
	\lim_{t \to \infty}  \frac{1}{t^\beta} \log \mathbb{E}(|u(t,x)|^2) \nonumber = \left( \frac{2a}{2a(b+r)- b\alpha} \right)^\beta (\theta\rho)^{\frac{a}{2a(b+r)-b\alpha-a}}\left(2(b+r)-\frac{b\alpha}{a}-1\right).
	\end{equation}
\end{lemma}
\begin{proof}
	By part (1) of Theorem \ref{T:solvable}, there is an $L^2(\Omega)$ solution $u(t,x)$.  By the
	scaling property \eqref{E:normscale},
	\begin{align*}
		\mathbb{E}(|u(t,x)|^2) & = \sum_{n\ge 0} 	\theta^n(n!)\Norm{\widetilde{f}_n(\cdot,0,t)}_{\mathcal{H}^{\otimes n}}^2 = \sum_{n \ge 0}\theta^n(n!)t^{(2(b+r)-b\alpha/a)n}\Norm{\widetilde{f}_n(\cdot,0,1)}_{\mathcal{H}^{\otimes n}}^2 \\
                           & = \sum_{n \ge 0} z_n R_nt^{(2(b+r)-b\alpha/a)n}
	\end{align*}
	where
	\begin{align*}
		R_n = (n!)^{2(b+r)-b\alpha/a}\Norm{\widetilde{f}_n(\cdot,0,1)}_{\mathcal{H}^{\otimes n}}^2 \quad \text{and} \quad
		z_n = \frac{\theta^n}{(n!)^{2(b+r)-(b\alpha/a)-1}}.
	\end{align*}
	Notice that \eqref{E:2ineq-b} above says that
	\begin{align*}
		\frac{1}{n}\log(R_n) \to  \log\left( \left(\frac{2}{2(b+r)-\frac{b\alpha}{a}}\right)^{2(b+r)-\frac{b\alpha}{a}}  \rho \right )  \quad \text{as } n\to \infty.
	\end{align*}
	Now define $R$ to be
	\begin{align*}
		R=\left(\frac{2}{2(b+r)-\frac{b\alpha}{a}}\right)^{2(b+r)-\frac{b\alpha}{a}} \rho.
	\end{align*}
	We want to find a $\beta$ and $A$ so that
	\begin{align*}
		\lim_{t \to \infty}\frac{1}{t^\beta} \log \sum_{n \ge 0} z_n R_n \left(t^{(2(b+r)-b\alpha/a)}\right)^n = A.
	\end{align*}
	Indeed, by the following limit (see \cite[Lemma A.3]{BCC21}),
	\begin{align*}
		\lim_{t\to\infty} t^{-1/\gamma}\log\sum_{n\ge 0} (n!)^{-\gamma} t^n = \gamma, \qquad \text{for all $\gamma>0$,}
	\end{align*}
	we see that
	\begin{align*}
		\lim_{t\to \infty} \left[\frac{1}{(\theta
		R)t^{2(b+r)-b\alpha/a}}\right]^{\frac{1}{2(b+r)-(b\alpha/a)-1}} \log \sum_{n \ge 0} \frac{\left[ (\theta R) t^{2(b+r)-b\alpha/a}\right]^n}{(n!)^{2(b+r)-(b\alpha/a)-1}} = 2(b+r) - (b\alpha/a) -1,
	\end{align*}
	which, by an easy algebraic manipulation, is equivalent to
	\begin{align*}
		\lim_{t\to \infty} & \left[\frac{1}{t^{2(b+r)-b\alpha/a}}\right]^{\frac{1}{2(b+r)-(b\alpha/a)-1}} \log \sum_{n \ge 0} \frac{\left[ (\theta R) t^{2(b+r)-b\alpha/a}\right]^n}{(n!)^{2(b+r)-(b\alpha/a)-1}} \\
                       & = [2(b+r) - (b\alpha/a) -1](\theta R)^{\frac{1}{2(b+r)-(b\alpha/a)-1}}.
	\end{align*}
	Hence,
	% \begin{align*}
	% 	\beta = \frac{2(b+r)-b\alpha/a}{2(b+r)-(b\alpha/a)-1} \quad \text{and} \quad A=[2(b+r) - (b\alpha/a) -1](\theta R)^{\frac{1}{2(b+r)-(b\alpha/a)-1}}
	% \end{align*}
	% or by removing the fractions in the denominator
	\begin{align*}
		\beta = \frac{2a(b+r)-b\alpha}{2a(b+r)-(b\alpha)-a} \quad \text{and} \quad
		  A   = [2(b+r) - (b\alpha/a) -1](\theta R)^{\frac{a}{2a(b+r)-b\alpha-a}}.
	\end{align*}
	Finally, an application of \cite[Lemma A.2]{BCC21} shows that
	\begin{align*}
		\lim_{t\to \infty} & \frac{1}{t^\beta} \log \mathbb{E}(|u(t,x)|^2) \\
                       & = \theta^{\frac{a}{2a(b+r)-b\alpha-a}}\left( \frac{2a}{2a(b+r)- b\alpha} \right)^\beta \rho^{\frac{a}{2a(b+r)-b\alpha-a}}\left(2(b+r)-\frac{b\alpha}{a}-1\right),
	\end{align*}
	which proves Lemma \ref{L:p=2}.
\end{proof}

Now we are ready to prove \eqref{E:momentasym-upper}.

\begin{proof}[Proof of \eqref{E:momentasym-upper}]
	By the hypercontractivity \eqref{E:HyContr} and the scaling property \eqref{E:normscale}, we see
	that for all $p \ge 2$,
	\begin{align*}
		\Norm{u(t,0)}_p^2  \le
		\Norm{u_{(p-1)\theta}(t,0)}_2^2 &= \sum_{n \ge 0}(n!)^n\theta^n(p-1)^n \Norm{\widetilde{f}_n(\cdot,0,t)}_{\mathcal{H}^{\otimes n}}^2 \\
																		&= \sum_{n \ge 0}(n!)^n\theta^n \Norm{\widetilde{f}_n\left(\cdot,0,t(p-1)^{\frac{1}{2(b+r)-b\alpha/a}}\right)}_{\mathcal{H}^{\otimes n}}^2 \\
																		&= \Norm{u\left(t(p-1)^{\frac{1}{2(b+r)-b\alpha/a}}, 0\right)}_2^2.
	\end{align*}
	Hence,
	\begin{align}
		\label{E:pMom2}
		\Norm{u(t,0)}_p \le \Norm{u\left(t(p-1)^{\frac{1}{2(b+r)-b\alpha/a}}, 0\right)}_2 =: \Norm{u(t_p,0)}_2,
	\end{align}
	where $t_p$ is defined in \eqref{E:beta-tp}. Finally, an application of Lemma \ref{L:p=2} proves
	\eqref{E:momentasym-upper}.
\end{proof}

% \section{Lower bound of the asymptotics --- the proof of \eqref{E:momentasym-lower}}
\section{Lower bound of the asymptotics}
\label{S:lowerbd}

% The goal of this section is to prove the following theorem.
%
% \begin{theorem} \label{T:lbound} Under the assumptions of Theorem \ref{T:solvable}
% \begin{equation*} \liminf_{n \to \infty} t_p^{-\beta} \log \Norm{u(t,x)}_p \ge
% \left(\frac{1}{2}\right) \left( \frac{2a}{2a(b+r)- b\alpha} \right)^\beta 		(\theta
% \rho)^{\frac{a}{2a(b+r)-b\alpha-a}}\left(2(b+r)-\frac{b\alpha}{a}-1\right).  \end{equation*}
% \end{theorem}

In this section, we will prove the lower bound of \eqref{E:momentasym}, namely,
\begin{equation}
	\label{E:momentasym-lower}
	\begin{aligned}
		\liminf_{t_p \to \infty} t_p^{-\beta} \log \Norm{u(t,x)}_p \ge & \left(\frac{1}{2}\right)\left(\frac{2a}{2a(b+r)- b\alpha} \right)^\beta \\
																																	 & \times \left(\theta\nu^{-\alpha/a} \mathcal{M}_a^{\frac{2a-\alpha}{a}}\right)^{\frac{a}{2a(b+r)-b\alpha-a}}\left(2(b+r)-\frac{b\alpha}{a}-1\right).
	\end{aligned}
\end{equation}
Through out this section, we assume that Assumptions \ref{A:Nonnegative} and \ref{A:Noise}, and
condition \eqref{E:alpha} hold.

We start by defining the function $W_n(t,\phi)$ on $(0,\infty) \times L_{\mathbb{C}}^2(\mu)$ by
\begin{equation}
	W_n(t,\phi) := \int_{[0,t]^n<}
	\int_{(\R^d)^n} \prod_{k=1}^n \phi(\xi_k) \prod_{k=1}\mathcal{F}G(s_k-s_{k-1},\cdot)(\xi_k + \cdots + \xi_n) \mu(\ud \xi_1)\cdots \mu(\ud \xi_n) \ud s.
\end{equation}
with $s_0=0$. We now give conditions under which $W_n$ is well defined.

\begin{lemma}
	\label{L:LWn}
	If the measure $\mu$ satisfies the relation in \eqref{E:muassumption}, then $W_n(t,\phi)$ is well
	defined and for any $d \ge 1$, $t>0$ and $\phi \in L^2_{\mathbb{C}}(\mu)$. Moreover,
	\begin{equation}
		\label{E:LWn}
		\int_0^\infty e^{-t}W_n(t,\phi) \ud t = \int_{(\R^d)^n} \prod_{k=1}^n \phi(\xi_k) \prod_{k=1}^n \frac{1}{1+\frac{\nu}{2}|\xi_k + \cdots + \xi_n|^a} \mu(\ud \xi) \cdots \mu(\ud \xi_n).
	\end{equation}
\end{lemma}
\begin{proof}
	The proof follows the exact arguments as those in the proof of Lemma 6.2 of \cite{BCC21} except
	that one needs to use the following Laplace transform:
	\begin{align*}
		\int_0^\infty e^{-t} \mathcal{F}G(t,\cdot)(\xi) \; \ud t =\frac{1}{1+\frac{\nu}{2}|\xi|^a};
	\end{align*}
	see the second equation in \eqref{E:LandFL}.
	% By Fubini's theorem, the left hand side of \eqref{E:LWn} is equal to
	% \begin{align*}
	% 	 \int_{(\R^d)^n} \prod_{k=1}^n \phi(\xi_k) \int_0^\infty e^{-t} \int_{[0,t]^n_<} \prod_{k=1}^n \mathcal{F}G(s_k-s_{k-1}, \cdot)(\xi_k + \cdots + \xi_n) \ud \vec{s} \ud t \mu(\ud \vec {\xi}),
	% \end{align*}
	% which is equal to, by the change of variables $t-s_n = u$ and $s_k - s_{k-1} = u_k$ for $k=1,...,n$,
	% \begin{align*}
	% 	 \int_{(\R^d)^n} \prod_{k=1}^n \phi(\xi_k) \left( \int_0^\infty e^{-u} \ud u \right) \prod_{k=1}^n \left( \int_0^\infty e^{-u_k} \mathcal{F}G(u_k,\cdot)(\vec{\xi}) \; \ud u_k \right) \prod_{k=1}^n \mu(\ud \xi_k).
	% \end{align*}
	% Now recall that $\int_0^\infty e^{-t} \mathcal{F}G(t,\cdot)(\xi) \; \ud t = 1/(1+\frac{\nu}{2}|\xi|^a)$. Then the above quantity is equal to
	% \begin{align*}
	% 	 \int_{(\R^d)^n} \prod_{k=1}^n \phi(\xi_k) \prod_{k=1}^n \frac{1}{1+|\sum_{j=k}^n \xi_j|^2} \prod_{k=1}^n \mu(\ud \xi_k).
	% \end{align*}
	% The result it follows from applying the Cauchy-Schwarz inequality. In addition, by using the same
	% arguments that lead to \eqref{E:lem62step}, we see the above quantity is finite.
\end{proof}

The next proposition is a restatement of Proposition 6.3 of \cite{BCC21}. The proof follows the same
proof as that of Proposition 6.3 of \cite{BCC21}. We will not repeat here.
\begin{proposition}
	\label{P:Wnprop1}
	For $f\in \mathcal{H}$, $t>0$, and $p,q>1$ with $p^{-1} + q^{-1}=1$, it holds that
	\begin{align}
		\label{E:Wnprop1}
		\Norm{u(t,0)}_p \ge \exp\left\{ -\frac{1}{2}(q-1)\Norm{f}_{\mathcal{H}}^2 \right\} \left| \sum_{n \ge 0} \theta^{n/2} W_n(t,\mathcal{F}f) \right|
	\end{align}
	and as a consequence, the series $\left| \sum_{n \ge 0} \theta^{n/2} W_n(t,\mathcal{F}f) \right|$
	converges provided that $\Norm{u(t,0)}_p < \infty$, which is the case under Theorem
	\ref{T:solvable}.
\end{proposition}

Now we are going to apply a scaling argument to \eqref{E:Wnprop1} in order to put $t$ and $p$
together, from which we can determine the constants $t_p$ and  $\beta$ defined in
\eqref{E:beta-tp}.

\begin{proposition}
	\label{P:Wnprop2}
	For $p,q > 1$, $p^{-1}+q^{-1}=1$ and for any $f \in \mathcal{H}$,
	\begin{align}
		\label{E:Wnprop2}
		\Norm{u(t,0)}_p \ge \exp\left\{ -\frac{1}{2} t_p^\beta \Norm{f}_{\mathcal{H}}^2 \right\}\left| \sum_{n \ge 0} \theta^{n/2} W_n\left(t_p^\beta, \mathcal{F}f\right) \right|,
	\end{align}
	where the constants $\beta$ and $t_p$ are defined in \eqref{E:beta-tp}.
\end{proposition}
\begin{proof}
	From Proposition \ref{P:Wnprop1}, we see that for any $f\in \mathcal{H}$, the inequality
	\eqref{E:Wnprop1} holds. For some constants $V,W>0$, which will be determined in this proof, let
	$f_\tau(x):=\tau^V f(\tau^W x)$ be a scaled version of $f$. It is clear that $f_\tau\in
	\mathcal{H}$. By some elementary scaling arguments (see the proof of the Lemma 6.4 of
	\cite{BCC21}), one can show that
	\begin{gather}
		\Norm{f_\tau}_{\mathcal{H}}^2        = \tau^{2(V-dW)+W\alpha} \Norm{f}_\mathcal{H}^2 \quad \text{and}\\
		W_n \left(t,\mathcal{F}f_\tau\right) = \tau^{n\left[ V-W\left( (d-\alpha)+ \frac{a}{b}(b+r) \right) \right]}W_n\left(t\tau^{\frac{a}{b}W}, \mathcal{F}f\right).
	\end{gather}
	Hence, an application of Proposition \ref{P:Wnprop1} to $f_\tau$ shows that
	\begin{align*}
    & \Norm{u(t,0)}_p  \ge \exp\left\{ -\frac{1}{2}(q-1)\Norm{f_\tau}_{\mathcal{H}}^2 \right\} \left| \sum_{n \ge 0} \theta^{n/2} W_n(t,\mathcal{F}f_\tau) \right| \\
    & =\exp\left\{ -\frac{1}{2}(q-1)\tau^{2(V-dW)+W\alpha} \Norm{f}_\mathcal{H}^2 \right\} \left| \sum_{n \ge 0} \theta^{n/2} \tau^{n\left[ V-W\left( (d-\alpha)+ \frac{a}{b}(b+r) \right) \right]}W_n\left(t\tau^{\frac{a}{b}W}, \mathcal{F}f\right) \right|.
	\end{align*}
  Comparing the above lower bound with that in \eqref{E:Wnprop2}, we obtain the following two
	relations with three unknowns $W$, $V$ and $\tau$:
	\begin{gather}
		\label{E:VW-1}
		V-W\left( (d-\alpha)+ \frac{a}{b}(b+r) \right) = 0, \\
		(p-1)^{-1}\tau^{2(V-dW)+W\alpha}               = t\tau^{\frac{a}{b}W}.
		\label{E:VW-2}
	\end{gather}
	Since \eqref{E:VW-2} should hold for all $t>0$ and $p\ge 2$, one can choose $\tau = (p-1)t$ to
	reduce the relation \eqref{E:VW-2} to
	\begin{align*}
		\tau^{2(V-dW)+W\alpha} = \tau^{\frac{a}{b}W+1},
			% (p-1)^{2(V-dW)+W\alpha-1}t^{2(V-dW)+W\alpha} &= (p-1)^{\frac{a}{b}W}t^{\frac{a}{b}W+1}.
	\end{align*}
	which then gives the following equation
	% Equating either the exponents of the term $(p-1)$ or that of $t$ leads to
	\begin{equation}
		\label{E:VW-3}
		2(V-dW)+W\alpha= 1 + \frac{a}{b}W .
	\end{equation}
	Now solve the linear equations \eqref{E:VW-1} and \eqref{E:VW-3} for $W$ and $V$ to see that
	\begin{align*}
		\begin{cases}
			W = \displaystyle \dfrac{b}{a}\left( \beta -1 \right), \bigskip \\
			V = \displaystyle \left(\frac{a}{b}(b+r)-\alpha +d\right)\frac{b}{a}(\beta -1),
		\end{cases}
		\qquad
		\text{with }\beta := \dfrac{2(b+r)-\dfrac{b\alpha}{a}}{ 2(b+r)-\dfrac{b\alpha}{a} -1 }.
	\end{align*}
	Therefore, the scaling $f_\tau$ and $t_p^\beta$ should be
	\begin{align*}
		f_\tau(x) = \tau^{\left(\frac{a}{b}(b+r)-\alpha+d\right)\frac{b}{a}(\beta -1)} f\left(\tau^{\frac{b}{a}(\beta -1)}x\right)
		\quad \text{and} \quad
		t_p^\beta := t\tau^{\frac{a}{b}W} = (p-1)^{\beta -1} t^\beta,
	\end{align*}
	respectively. This completes the proof of Proposition \ref{P:Wnprop2}.
\end{proof}

To remove the absolute value sign in Proposition \ref{P:Wnprop2}, we want to identify all the $f \in
\mathcal H$ for which $W_n(t,\mathcal{F}f)$ is nonnegative. In fact, if we consider the space
\begin{align*}
	\mathcal{H}_+ = \left\{ f\in \mathcal{H} : \text{$f$ is nonnegative and nonnegative definite}\right\},
\end{align*}
then by Plancherel's theorem, for $f\in\mathcal{H}_+$,
\begin{align*}
		W_n(t,\mathcal{F}f) = \int_{[0,t]^n<} \int_{\R^{nd}} \prod_{k=1}^n(f*\gamma)(x_k) \prod_{k=1}
		G(s_k - s_{k-1}, x_k-x_{k-1}) \ud \vec{x} \ud \vec{s} =: U_n(t,f) \ge 0
\end{align*}
with the convention that $s_0=0$ and $x_0=0$, where we have used the fact that the fundamental
solution $G(t,x)$ is nonnegative (under Assumption \ref{A:Nonnegative}).  Now define
\begin{align*}
	W_n(t) := \sup_{f\in\mathcal{H},\: \Norm{f}_{\mathcal{H}}=1} W_n(t,\mathcal{F}(f)) \quad \text{and} \quad
	U_n(t) :=  \sup_{f\in\mathcal{H}_+, \Norm{f}_{\mathcal{H}}=1} U_n(t,f).
\end{align*}
It is clear that $W_n(t) \ge U_n(t) \ge 0$.

\begin{proposition}
	\label{P:LogU}
	If $\tau$ is an exponential random variable with mean one, then
	\begin{align*}
		\liminf_{n\to \infty} \frac{1}{n} \log \mathbb{E} (U_n(\tau) ) \ge \log(\rho^{1/2}).
	\end{align*}
\end{proposition}
\begin{proof}
	We start by letting $\tau$ be an exponential random variable with mean one. With this,
	\begin{align*}
		\mathbb{E}(U_u(\tau)) = \int_0^\infty e^{-t} U_n(t) \ud t
		\ge \sup_{f\in \mathcal{H}_+,\: \Norm{f}_\mathcal{H} =1} \int_0^\infty e^{-t} U_n(t,f) \ud t.
	\end{align*}
	For any $f \in \mathcal{H}_+$ with $\Norm{f}_{\mathcal{H}}= 1$, by Bochner's theorem,
	$\mathcal{F}f$ is nonnegative and nonnegative definite, which further implies that $\mathcal{F}f$
	is even. In addition, Lemma \ref{L:LWn} gives us that
	\begin{align*}
		\int_0^{\infty} e^{-t} U_n(t,f) \ud t & = \int_0^\infty e^{-t} W_n(t,\mathcal{F}f) \ud t \\
                                          & = \int_{\R^{dn}} \prod_{k=1}^n \mathcal{F}f(\xi_k) \prod_{k=1}^n \frac{1}{1+ \frac{\nu}{2}|\xi_k+\cdots + \xi_n|^a} \mu(\ud \xi_1) \cdots \mu(\ud \xi_n).
	\end{align*}
	By the same arguments as  \cite[Assumption C]{BCC21} with the replacement \eqref{E:Changes} (see
	also the proof of \cite[Theorem 1.2]{BCR09}), we have that
	\begin{align*}
		\liminf_{n\to \infty}\frac{1}{n} \log \int_{\R^{dn}} \prod_{k=1}^n(\mathcal{F}f)(\xi_k) \prod_{k=1}^n\frac{1}{1+\frac{\nu}{2}|\xi_k+\cdots+\xi_n|^a} \mu(\ud \xi_1) \cdots \mu(\ud \xi_n) \ge  \log\rho(\mathcal{F}f),
	\end{align*}
	where $\rho(\cdot)$ is defined on the space of nonnegative and nonnegative definite functions and
	is given by
	\begin{align*}
		\rho(g) := \sup_{\Norm{h}_{L^2(\R^d)}=1} \int_{\R^d}g(\xi)\left[ \int_{\R^d} \frac{h(\xi + \eta)h(\eta)}{\sqrt{1+\frac{\nu}{2}|\xi+\eta|^a}\sqrt{1+\frac{\nu}{2}|\eta|^a}}  \ud \eta\right] \mu(\ud \xi).
	\end{align*}
	Note that since both $\mu$ and $g$ are nonnegative, we may assume $h$ is also nonnegative for the
	remainder of this proof. With this being said, we see that for any $f \in \mathcal{H}_+$ with
	$\Norm{f}_H =1$,
	\begin{align*}
		\liminf_{n \to \infty} \frac{1}{n}\log \mathbb{E}(U_n(\tau)) \ge \log \rho(\mathcal{F}f).
	\end{align*}
	We now need to calculate
	\begin{align*}
		\sup_{f \in \mathcal{H}_+, \: \Norm{f}_{\mathcal{H}}=1} \rho\left(\mathcal{F}f\right).
	\end{align*}
	Consider a nonnegative function $h\in L^2(\R^d)$. The function
	$h(\cdot)/\sqrt{1+\frac{\nu}{2}|\cdot|^a} \in L^2(\R^d)$ so that $g_h(x) := (2\pi)^{d/2}
	\mathcal{F}^{-1}\left( \frac{h(\cdot)}{\sqrt{1+\frac{\nu}{2}|\cdot|^a}} \right)(x)$ is well
	defined. Under these conditions, $g \in W^{1,a}(\R^d)$ with
	\begin{align*}
		\Norm{g}_{W^{1,a}(\R^d)} = \frac{1}{(2\pi)^d} \int_{\R^d} (1+|\xi|^a) |\mathcal{F}g(\xi)|^2\ud\xi
		= \int_{\R^d} \frac{1+|\xi|^a}{1+\frac{\nu}{2}|\xi|^a} |h(\xi)|^2 \ud \xi
		\le C_\nu \Norm{h}_{L^2(\R^d)}^2 <\infty.
	\end{align*}
	Notice that
	\begin{align*}
		\int_{\R^d} \frac{h(\xi +
		\eta)h(\eta)}{\sqrt{1+\frac{\nu}{2}|\xi+\eta|^a}\sqrt{1+\frac{\nu}{2}|\eta|^a}}  \ud \eta
    & = (2\pi)^{-d}\int_{\R^d} \mathcal{F}g_h(\xi-\eta)\mathcal{F}g_h(-\eta) \ud \eta \\
    & =(2\pi)^{-d} (\mathcal{F}g_h*\widetilde{\mathcal{F}g_h})(\xi),
	\end{align*}
	where we have used the notation that $\widetilde{h}(x) = h(-x)$.  Since $h(\cdot)$ is real valued,
	we see that $\widetilde{\mathcal{F}g_h} = \overline{\mathcal{F}\bar{g_h}}= \mathcal{F}\bar{g_h} $.
	Using this and the fact that  $\mathcal{F}(fg) = (2\pi)^{-d}\mathcal{F}(f)*\mathcal{F}(g)$, we see
	that $(2\pi)^{-d}(\mathcal{F}g_h*\widetilde{\mathcal{F}g_h})(\xi)
	=\mathcal{F}\left[|g_h|^2\right](\xi)$. Hence,
	\begin{align*}
		\rho(\mathcal{F}f) & = \sup_{\Norm{h}_{L^2(\R^d)}=1}(2\pi)^{-d} \int_{\R^d}\mathcal{F}f(\xi)(\mathcal{F}g_h*\widetilde{\mathcal{F}g_h})(\xi) \mu(\ud \xi) \\
                       & = \sup_{\Norm{h}_{L^2(\R^d)}=1} \int_{\R^d}\mathcal{F}f(\xi)\mathcal{F}[|g_h|^2](\xi) \mu(\ud \xi)                                   \\
                       & = \sup_{\Norm{h}_{L^2(\R^d)}=1}\langle f, |g_h|^2 \rangle _{\mathcal{H}}.
	\end{align*}
	With this we see that
	\begin{align}
		\label{E:supf}
		\sup_{ f\in \mathcal{H}, \Norm{f}_{\mathcal{H}=1} } \rho(\mathcal{F}f) = 	\sup_{ f\in \mathcal{H}, \Norm{f}_{\mathcal{H}=1} }\sup_{\Norm{h}_{L^2}=1} \langle f, |g_h|^2 \rangle _{\mathcal{H}} = \sup_{\Norm{h}_{L^2}=1} \langle |g_h|^2, |g_h|^2 \rangle^{1/2} _{\mathcal{H}}
	\end{align}
	where the optimal $f$ is chosen to be $|g_h|^2/\Norm{|g_h|^2}_{\mathcal{H}}$. Now we claim that
	\begin{equation}
		\label{E:GH+}
		|g_h|^2 \in \mathcal{H}_+,
	\end{equation}
	which implies that the supremum in \eqref{E:supf} can be restricted to $\mathcal{H}_+$ and
	\begin{align*}
		\sup_{ f \in \mathcal{H},\: \Norm{f}_{\mathcal{H}}=1 } \rho(\mathcal{F}f) =
		\sup_{ f \in \mathcal{H}_+,\: \Norm{f}_{\mathcal{H}}=1 } \rho(\mathcal{F}f) =
		\sup_{\Norm{h}_{L^2}=1} \left\langle |g_h|^2, |g_h|^2 \right\rangle^{1/2} _{\mathcal{H}}.
	\end{align*}
	Then notice that
	\begin{align}
		\label{E:gsq}
		\mathcal{F}( |g_h|^2 )(\xi) = (2\pi)^{-d} (\mathcal{F}g_h*\widetilde{\mathcal{F}g_h})(\xi)
                                = \int_{\R^d} \frac{h(\xi + \eta)h(\eta)}{\sqrt{1+\frac{\nu}{2}|\xi+\eta|^a}\sqrt{1+\frac{\nu}{2}|\eta|^a}} \ud \eta.
	\end{align}
	Hence, from \eqref{E:rho} and \eqref{E:gsq}, we see that
	\begin{align*}
		\sup_{ \Norm{h}_{L^2(\R^d)}=1 }\left\langle |g_h|^2, |g_h|^2 \right\rangle_{\mathcal{H}}
		=\sup_{ \Norm{h}_{L^2(\R^d)}=1 } \int_{\R^d} \left|\mathcal{F}( |g_h|^2 )(\xi)\right|^2 \mu(\ud \xi)
		= \rho_{\nu,a}
	\end{align*}
	which then leads us to the lower bound:
	\begin{align*}
		\liminf_{n\to \infty} \frac{1}{n} \log \mathbb{E} (U_n(\tau) ) \ge \log\left(\rho_{\nu,a}^{1/2}\right).
	\end{align*}
	Therefore, it remains to proving \eqref{E:GH+}. First notice that from the above we see that
	\begin{align*}
		\int_{\R^d} \left| \mathcal{F}\left(|g_h|^2\right)(\xi) \right|^2 \mu(\ud \xi) \le \rho_{\nu,a} < \infty
		\quad \Longrightarrow \quad |g_h(\cdot)|^2 \in \mathcal{H}.
	\end{align*}
	Moreover, since $h$ is nonnegative, from  \eqref{E:gsq}, we see that
	$\mathcal{F}\left(|g_h|^2\right)(\cdot)$ is also nonnegative. The Bochner-Schwarz theorem then
	implies that $|g_h(\cdot)|^2$ is nonnegative definite. It is clear that $|g_h(\cdot)|^2$ is
	nonnegative. This shows that $|g_h(\cdot)|^2 \in \mathcal{H}_+$.  This completes the whole proof of
	Proposition \ref{P:LogU}.
\end{proof}

\begin{lemma}
	\label{L:LogU}
	It holds that
	\begin{align}
		\label{E:LogU}
		\liminf_{n\to \infty} \frac{1}{n}\log \left[ (n!)^{ \frac{b}{a}(a-\frac{\alpha}{2} + \frac{a}{b}r)  } U_n(1) \right]
		\ge \log\left( \frac{\rho_{\nu,a}^{1/2}}{ \left(\frac{b}{a}[ a-\frac{\alpha}{2} + \frac{a}{b}r ]\right)^{\frac{b}{a}(a-\frac{\alpha}{2} + \frac{a}{b}r)} }\right).
	\end{align}
\end{lemma}
\begin{proof}
	Let $\tau$ be an exponential random variable with mean one.
	Notice that by some elementary scaling arguments, we have that for all $t>0$,
	\begin{align}
		\label{E:WU}
		W_n(t) = t^{n\frac{b}{a}\left(a- \frac{\alpha}{2}+\frac{a}{b}r \right)} W_n(1) \quad \text{and} \quad
		U_n(t) = t^{n\frac{b}{a}\left(a- \frac{\alpha}{2}+\frac{a}{b}r \right)} U_n(1),
	\end{align}
	which then imply that
	\begin{align*}
		\mathbb{E}\left[ U_n(\tau) \right] = \mathbb{E}\left( \tau^{n\frac{b}{a}(a-\frac{\alpha}{2}+\frac{a}{b}r)} \right) U_n(1)
                                       = \Gamma\left( n\frac{b}{a}\left(a-\frac{\alpha}{2} +\frac{a}{b}r \right) +1 \right) U_n(1).
	\end{align*}
	Then by Proposition \ref{P:LogU},
	\begin{align}\label{E:lb1}
		\liminf_{n\to \infty} \frac{1}{n} \log \left[ \Gamma\left( n\frac{b}{a}\left(a-\frac{\alpha}{2} +\frac{a}{b}r \right) +1 \right) U_n(1) \right] \ge \log\left(\rho_{\nu, a}^{1/2}\right).
	\end{align}
	Therefore, \eqref{E:LogU} is proven by noticing that, thanks to \eqref{E:Stirling},
	\begin{align*}
		\liminf_{n\to \infty} \frac{1}{n} \log \left( \frac{ \Gamma\left( n\frac{b}{a}\left(a-\frac{\alpha}{2} +\frac{a}{b}r \right) +1 \right) }{ (n!)^{\frac{b}{a}\left(a-\frac{\alpha}{2} +\frac{a}{b}r \right)} } \right )
		= \frac{b}{a}\left(a-\frac{\alpha}{2} +\frac{a}{b}r \right) \log\left( \frac{b}{a}\left(a-\frac{\alpha}{2} +\frac{a}{b}r \right) \right),
	\end{align*}
	where the condition $\frac{b}{a}\left(a-\frac{\alpha}{2} +\frac{a}{b}r \right) >0$ is guaranteed
	by \eqref{E:alpha4} (or \eqref{E:alpha}).
\end{proof}

We need one last lemma before the proof of the lower bound:

\begin{lemma}
	\label{L:logUnlb}
	For any $k,\theta >0$, there exists a constant $c_1 = c_1(\alpha,\mathcal{M}, k, \theta) >0$ such
	that, by setting $n_t = [c_1 t]$, it holds that
	\begin{equation}\label{E:69ineq}
		\liminf_{t \to \infty} \frac{1}{t} \log( k^{n_t} \theta^{n_t/2}U_{n_t}(t)) \ge
		\left( k \sqrt{\theta} \right)^{\frac{2a}{2ab-\alpha b + 2ar}} \left( \rho_{\nu,a}^{1/2} \right)^{\frac{2a}{2ab-\alpha b + 2ar}}.
	\end{equation}
\end{lemma}
\begin{proof}
	Fix an arbitrary $\epsilon > 0$. Lemma \ref{L:LogU} guarantees the existence of an $N_\epsilon >0$ so for all $n>N_\epsilon$
	\begin{align}
		\label{E:LogUnlb-1}
		(n!)^{\frac{b}{a}\left( a - \frac{\alpha}{2} + \frac{a}{b}r \right)} U_n(1) \ge \exp\left( n(\log(R)-\epsilon) \right) = R^n e^{-n\epsilon }
	\end{align}
	where
	\begin{align*}
		R = \rho_{\nu,a}^{1/2}  \left( \frac{b}{a}\left[ a - \frac{\alpha}{2} + \frac{a}{b}r \right ] \right)^{- \frac{b}{a}\left( a - \frac{\alpha}{2} + \frac{a}{b}r \right) }.
	\end{align*}
	Now fix a $c>0$ and let $n_t := [ct]$. Notice that $n_t \ge N_\epsilon$ for any $t>t_\epsilon :=
	(N_\epsilon +1)/c$.  For $t > t_\epsilon$, from \eqref{E:LogUnlb-1}, we have
	\begin{align}
		\label{E:69ineq2}
		k^{n_t}\theta^{n_t/2}U_{n_t}(t)
		= 	k^{n_t} \theta^{n_t/2} t^{\frac{b}{a} \left( a-\frac{\alpha}{2} + \frac{a}{b}r \right)n_t } U_{n_t}(1)
		\ge \frac{ k^{n_t} \theta^{n_t/2} t^{\frac{b}{a} \left( a-\frac{\alpha}{2} + \frac{a}{b}r \right)n_t } }{ (n!)^{\frac{b}{a} \left(a-\frac{\alpha}{2} +\frac{a}{b}r \right)} } R^{n_t}e^{-n_t \epsilon}.
	\end{align}
	Notice that $[ct]/t \to c$ as $t \to \infty$ which means that $n_t/t \to c$ as $t \to \infty$.
	With this we can say
	\begin{align*}
		\liminf_{t \to \infty} \frac{1}{t} \log \left( k^{n_t} \theta^{n_t/2} U_{n_t}(t) \right)
    = c \liminf_{t\to \infty} \frac{1}{n_t} \log \left( k^{n_t} \theta^{n_t/2} U_{n_t}(t) \right)
    = : I(n_t).
	\end{align*}
	Now, by \eqref{E:69ineq2}, we have that
	\begin{align*}
		I(n_t) & \ge c \liminf_{t\to \infty} \frac{1}{n_t} \log \left( (kR\sqrt{\theta})^{n_t} \frac{ t^{\frac{b}{a} \left( a-\frac{\alpha}{2} + \frac{a}{b}r \right)n_t } }{ (n_t!)^{\frac{b}{a}\left( a -\frac{\alpha}{2} + \frac{a}{b}r \right) }  }\right) - c\epsilon                                                                                                  \\
           & = c \log(k \sqrt{\theta}R) + c \liminf_{t \to \infty} \frac{1}{n_t} \log \left[  \left(\frac{t}{n_t}\right)^{\frac{b}{a}\left(a-\frac{\alpha}{2} + \frac{a}{b}r\right)n_t} \frac{ n_t^{\frac{b}{a}\left( a- \frac{\alpha}{2} + \frac{a}{b}r \right)n_t } } { (n_t!)^{\frac{b}{a}\left( a - \frac{\alpha}{2} + \frac{a}{b}r \right) } }\right] - c \epsilon \\
           & = c \log(k \sqrt{\theta}R) - c\frac{b}{a}\left( a - \frac{\alpha}{2} + \frac{a}{b}r \right) \log(c) + c\frac{b}{a}\left( a - \frac{\alpha}{2} + \frac{a}{b}r \right) \liminf_{t \to \infty} \frac{1}{n_t}\log\left(\frac{n_t^{n_t}}{(n_t)!}\right) - c\epsilon                                                                                             \\
           & = c \log(k \sqrt{\theta}R) - c\frac{b}{a}\left( a - \frac{\alpha}{2} + \frac{a}{b}r \right) \log(c) + c\frac{b}{a}\left( a - \frac{\alpha}{2} + \frac{a}{b}r \right)  - c\epsilon
	\end{align*}
	and letting $\epsilon$ tend to $0$ we see that
	\begin{align*}
		I(n_t) \ge c \left[ \log(k \sqrt{\theta}R) - \frac{b}{a}\left( a - \frac{\alpha}{2} + \frac{a}{b}r \right) \log(c) + \frac{b}{a}\left( a - \frac{\alpha}{2} + \frac{a}{b}r \right) \right] =: h(c).
	\end{align*}
	In order to maximize $h(c)$, notice that
	\begin{align*}
		h'(c) = 0 \quad \Longleftrightarrow  \quad  c^*=(k \sqrt{\theta}R)^{\frac{a}{b\left( a - \frac{\alpha}{2} + \frac{a}{b}r \right)}}.
	\end{align*}
	After plugging $c^*$ and replacing $R$ we arrive at the following inequality
	\begin{align*}
		I(n_t) \ge (k\sqrt{\theta})^{\frac{a}{b\left( a - \frac{\alpha}{2} + \frac{a}{b}r \right)} }(\rho_{\nu,a}^{1/2})^{\frac{a}{b\left( a - \frac{\alpha}{2} + \frac{a}{b}r \right)}},
	\end{align*}
	which proves \eqref{E:69ineq} after some simplification.
\end{proof}

We are now ready to prove \eqref{E:momentasym-lower}.

\begin{proof}[Proof of \eqref{E:momentasym-lower}]
	By proposition \ref{P:Wnprop2}, for and $p,q >0$ with $p^{-1}+q^{-1} = 1$ we have that
	\begin{align*}
		\Norm{u(t,0)}_p \ge \exp\left\{ -\frac{1}{2} t_p^\beta \Norm{f}_{\mathcal{H}}^2 \right\}\left| \sum_{n \ge 0} \theta^{n/2} W_n(t_p^\beta, \mathcal{F}f) \right|.
	\end{align*}
	We now take the supremum over all $f \in \mathcal{H}_+$ with $\Norm{f}_{\mathcal{H}}=k >0 $.
	Recall that $W_n(t,\phi) = k^nW_n(t,\phi/k)$ for $\phi \in L^2(\mu)$ and the non-negativity of $W_n(t,\cdot)$ on $\mathcal{H}_+$. Let $c>0$.
	\begin{align*}
		\Norm{u(t,0)}_p & \ge 	\exp\left\{ -\frac{1}{2} t_p^\beta k^2 \right\} \sup_{f\in \mathcal{H}_+ \; \Norm{f}_{\mathcal{H}}=k}\sum_{n \ge 0} \theta^{n/2} W_n(t_p^\beta, \mathcal{F}f) \\
                    & = \exp\left\{ -\frac{1}{2} t_p^\beta k^2 \right\} \sup_{f\in \mathcal{H}_+ \; \Norm{f}_{\mathcal{H}}=1}\sum_{n \ge 0} k^n \theta^{n/2} W_n(t_p^\beta, \mathcal{F}f) \\
                    & \ge \exp\left\{ -\frac{1}{2} t_p^\beta k^2 \right\} \sup_{f\in \mathcal{H}_+ \; \Norm{f}_{\mathcal{H}}=1} k^{n_t} \theta^{{n_t}/2} U_{n_t}(t_p^\beta, f)
	\end{align*}
	where $n_t = [ct_p^\beta]$. Now by choosing $c$ as in Lemma \ref{L:logUnlb} we get that
	\begin{align*}
		\liminf_{n \to \infty} t_p^{-\beta} \log \Norm{u(t,x)}_p & \ge 	\liminf_{n \to \infty} \left( \frac{-\frac{1}{2}t_p^\beta k^2}{t_p^\beta} + \frac{1}{t_p^\beta}\log \left[ \sup_{f\in \mathcal{H}_+ \; \Norm{f}_{\mathcal{H}}=1} k^{n_t} \theta^{{n_t}/2} U_{n_t}(t_p^\beta, f) \right]  \right) \\
                                                             & = -\frac{1}{2}k^2 + k^{ \frac{2a}{2ab-\alpha b +2ar} }(\rho \theta)^{\frac{a}{2ab -\alpha b + 2ar}} =: h(k).
	\end{align*}
	By maximizing $h$ for $k>0$, we see that $h$ is maximized at the point
	\begin{align*}
		k^* = \left( \frac{2ab -\alpha b + 2ar}{2aB} \right)^{ \frac{2ab-\alpha b +2ar}{2a-2[2ab-\alpha b + 2ar]}} \quad \text{with $B=(\theta\rho)^{ \frac{a}{2ab-\alpha b + 2ar} }$}.
	\end{align*}
	Inserting $k^*$ into $h$ gives us that
	\begin{align*}
		h(k^*) = B^{\beta}\left( \frac{2a}{2ab - \alpha b +2ar} \right)^\beta \left( \frac{2ab-\alpha b + 2ar -a }{2a} \right)
	\end{align*}
	and plugging in the value for $B$ proves \eqref{E:momentasym-lower}.
\end{proof}

\section{Appendix}
\label{S:Appendix}

\begin{proof}[Proof of Lemma \ref{L:rhovar}]
	In this proof, $\mu(\ud x) = \varphi(x)\ud x = C_{\alpha,d}|x|^{-(d-\alpha)}\ud x$. By the change of
	variables $x' = \left(\nu/2\right)^{1/a} x$ and $y' = \left(\nu/2\right)^{1/a} y$, we see that
	\begin{align*}
		\rho_{\nu,a} & \left(|\cdot|^{-\alpha}\right) = \sup_{\Norm{f}_{L^2(\R^d)} =1} \int_{\R^d} \left[ \int_{\R^d} \frac{f(x+y)f(y)}{\sqrt{1+\frac{\nu}{2}|x+y|^a } \sqrt{1+\frac{\nu}{2}|y|^a} } \ud y \right]^2 \mu(\ud x) \\
                 & = \left( \frac{\nu}{2} \right)^{-\alpha/a} \sup_{\Norm{f}_{L^2(\R^d)} =1} \int_{\R^d} \left[ \int_{\R^d} \frac{f\left(\left(\frac{\nu}{2}\right)^{-1/a}(x+y)\right)f\left(\left(\frac{\nu}{2}\right)^{-1/a}y\right)}{\sqrt{1+|x+y|^a } \sqrt{1+|y|^a} } \left(\frac{\nu}{2}\right)^{-d/a} \ud y \right]^2 \varphi(x)\ud x.
	\end{align*}
	By setting $f^*(x) =\left(\nu/2\right)^{-d/(2a)} f\left( \left(\nu/2\right)^{-1/a} x \right)$, we see that
	\begin{align*}
	\rho_{\nu,a} \left(|\cdot|^{-\alpha}\right) & = \left( \frac{\nu}{2} \right)^{-\alpha/a} \sup_{\Norm{f}_{L^2(\R^d)} =1} \int_{\R^d} \left[ \int_{\R^d} \frac{f^*\left(x+y\right)f^*\left(y\right)}{\sqrt{1+|x+y|^a } \sqrt{1+|y|^a} }  \ud y \right]^2 \varphi(x)\ud x \\
                                        & = \left( \frac{\nu}{2} \right)^{-\alpha/a}  \sup_{\Norm{f^*}_2 =1}  \int_{\R^d} \left[ \int_{\R^d} \frac{f^*\left(x+y\right)f^*\left(y\right)}{\sqrt{1+|x+y|^a } \sqrt{1+|y|^a} }  \ud y \right]^2 \varphi(x)\ud x       \\
                                        & = \left( \frac{\nu}{2} \right)^{-\alpha/a}  \rho_{2,a}\left(|\cdot|^{-\alpha}\right),
	\end{align*}
	where the second equality is due to the fact that $\int_{\R^d} f(x)^2 \ud x = \int_{\R^d} f^*(x)^2
	\ud x$.  Then an application of \eqref{E:oldrhovarrep} proves \eqref{E:rhoLam}.

	Similarly, for \eqref{E:LogRho}, by change of variables $\xi_{\sigma(j)}' =
	\left(\nu/2\right)^{1/a} \xi_{\sigma(j)}$, we see that
	\begin{align*}
		\lim_{n \to \infty} \frac{1}{n} & \log \left[ \frac{1}{(n!)^2} \int_{(\R^d)^n} \left( \sum_{\sigma \in \Sigma_n} \prod_{k=1}^n \frac{1}{1+ \frac{\nu}{2}|\sum_{j=k}^n \xi_{\sigma(j)}|^a}\right)^2 \mu (\ud \vec{\xi})\right]                                                                                               \\
                                    & = \lim_{n \to \infty} \frac{1}{n} \log \left[ \frac{1}{(n!)^2}  \left( \frac{\nu}{2} \right)^{-n\alpha/a}  \int_{(\R^d)^n} \left[\sum_{\sigma \in \Sigma_n} \prod_{k=1}^n \frac{1}{1+\left|\sum_{j=k}^n \xi_{\sigma(j)} \right|^a}\right]^2 \mu(\ud \vec{\xi} )\right]                    \\
                                    & = \log\left[ \left( \frac{\nu}{2} \right)^{-\alpha/a} \right] +  \lim_{n \to \infty} \frac{1}{n} \log \left[ \frac{1}{(n!)^2} \int_{(\R^d)^n} \left[\sum_{\sigma \in \Sigma_n} \prod_{k=1}^n \frac{1}{1+\left|\sum_{j=k}^n \xi_{\sigma(j)} \right|^a}\right]^2 \mu(\ud \vec{\xi} )\right] \\
                                    & = \log\left( \left( \frac{\nu}{2} \right)^{-\alpha/a} \right) + \log\left( \rho_{2,a} \left(|\cdot|^{-\alpha}\right) \right)
																		= \log\left( \rho_{\nu, a} \left(|\cdot|^{-\alpha}\right)\right),
	\end{align*}
	where we have applied \eqref{E:oldlim} and \eqref{E:rhoLam}. This completes the proof of Lemma
	\ref{L:rhovar}.
\end{proof}

\begin{proof}[Proof of \eqref{E:fHn}]
	% We first derive the following scaling property for $\mathcal{F}\widetilde{f}_n$ where
	% $\widetilde{f}_n$ is the kernel of the $n^{th}$ Weiner integral in the chaos expansion of the
	% solution \eqref{E:solwexp}:
	% \begin{equation*}
	% 	% \label{E:Ffnscale}
	% 	\mathcal{F}\widetilde{f}_n(\cdot,0,ct)(\xi_1,\cdots,\xi_n) = c^{n(b+r)} \mathcal{F}\widetilde{f}_n(\cdot,0,t)(c^{b/a}\xi_1, \cdots, c^{b/a}\xi_n).
	% \end{equation*}
	% For an arbitrary $c>0$. Because
	% \begin{align*}
	% 	\mathcal{F}\widetilde{f}_n(\cdot,0,ct)(\xi_1,\cdots,\xi_n) = \frac{1}{n!} \sum_{\sigma \in \Sigma_n} \mathcal{F}f_n(\cdot,0,ct)(\xi_{\sigma(1)},\cdots, \xi_{\sigma(n)}),
	% \end{align*}
	% by the linearity, it suffices to just look at the
	% $\mathcal{F}f_n(\cdot,0,ct)\left(\xi_{\sigma(1)},\cdots, \xi_{\sigma(n)}\right)$ terms individually.
	Starting from \eqref{E:fweker}, by the change of variables $t_i' = t_i/c$ and the scaling property in
	\eqref{E:Ffnscale}, we have that
	\begin{align*}
		\mathcal{F}f_n(\cdot,0,ct)(\xi_1,\cdots, \xi_n)
    & = \int_{[0,ct]^n_<} \prod_{k=1}^n \overline{ \mathcal{F}G(t_{k+1} - t_k, \cdot)\left( \sum_{j=1}^k \xi_j \right) }  \ud \vec{t}                                 \\
    & = \int_{[0,t]^n_<} \prod_{k=1}^n  \overline{ \mathcal{F}G\left(c(t_{k+1} - t_k), \cdot \right)\left( \sum_{j=1}^k \xi_j \right) } c^n \ud \vec{t}               \\
    & = \int_{[0,t]^n_<} \prod_{k=1}^n \overline{ c^{b+r-1} \mathcal{F}G\left(t_{k+1} - t_k, \cdot \right)\left( c^{b/a} \sum_{j=1}^k \xi_j \right) } c^n \ud \vec{t} \\
    & = c^{n(b+r)} \mathcal{F}f_n\left(\cdot,0,t\right)\left(c^{b/a}\xi_1, \cdots, c^{b/a}\xi_n\right),
	\end{align*}
	from which we see that
	\begin{align*}
		\int_0^\infty e^{-t} \Norm{ \widetilde{f}_n(\cdot,0,t) }_{\mathcal{H}^{\otimes n}}^2 \ud t
     & = \int_0^\infty e^{-t} \int_{\R^{nd}} \left| \mathcal{F}\widetilde{f_n}(\cdot,0;t)(\xi_1,\cdots,\xi_n) \right|^2 \mu(\ud \vec{\xi})  \ud t                                                             \\
     & = \int_0^\infty e^{-2t} \int_{\R^{nd}} \left| \mathcal{F}\widetilde{f_n}(\cdot,0;2t)(\xi_1,\cdots,\xi_n) \right|^2 \mu(\ud \vec{\xi}) \: 2\: \ud t                                                     \\
     & = 2^{2n(b+r)}\int_0^\infty 2e^{-2t} \int_{\R^{nd}} \left| \mathcal{F}\widetilde{f_n}(\cdot,0,t)(2^{b/a}\xi_1,\cdots,2^{b/a}\xi_n) \right|^2 \mu(\ud \vec{\xi})   \ud t                                 \\
     & = 2^{2n(b+r)}\int_0^\infty 2e^{-2t} \int_{\R^{nd}} \left| \mathcal{F}\widetilde{f_n}(\cdot,0;t)(\xi_1,\cdots,\xi_n) \right|^2 2^{\frac{-nbd}{a}} 2^{\frac{nb(d-\alpha)}{a}} \mu(\ud \vec{\xi})   \ud t \\
     & = 2^{n(2(b+r) - b\alpha/a)}  \int_0^\infty 2e^{-2t} \int_{\R^{nd}} \left| \mathcal{F}\widetilde{f_n}(\cdot,0;t)(\xi_1,\cdots,\xi_n) \right|^2 \mu(\ud \vec{\xi})   \ud t                               \\
     & = \frac{2^{n(2(b+r) - b\alpha/a)}}{(n!)^2}  \int_0^\infty 2e^{-2t} \int_{\R^{nd}} H_n(t,\vec{x})^2 \ud \vec{x}   \ud t,
	\end{align*}
	where we have applied \eqref{E:greenscale} in the fourth equality. This proves \eqref{E:fHn}.
\end{proof}

\bigskip

{\bf Acknowledgement.} We would like to thank Raluca Balan for some useful comments on our paper.

\end{document}